\documentclass[11pt]{amsart}

\usepackage{comment}
\usepackage[rgb,dvipsnames]{xcolor}
\usepackage{stmaryrd}
\usepackage{amsfonts}
\usepackage{amssymb}
\usepackage{amsmath,amsthm}
\usepackage{latexsym}
\usepackage{amscd}
\usepackage{esint}
\usepackage{mathrsfs}
\usepackage{dsfont}
\usepackage{setspace}
\usepackage{enumitem}
\usepackage[margin=1.25in]{geometry}
\usepackage{bm}

\usepackage{hyperref}
\usepackage{comment}
\usepackage{enumitem}

\usepackage{graphicx}
\usepackage[all]{xy}
\newtheorem{theorem}{Theorem}[section]
\newtheorem{corollary}[theorem]{Corollary}
\newtheorem{claim}[]{Claim}
\newtheorem{lemma}[theorem]{Lemma}

\newtheorem{proposition}[theorem]{Proposition}

\theoremstyle{definition}
\newtheorem{definition}[theorem]{Definition}

\theoremstyle{remark}
\newtheorem{remark}[theorem]{Remark}

\numberwithin{equation}{section}

\newcommand{\an}{\textnormal{An}^G}
\newcommand{\ann}{\mathcal{AN}^G}

\newcommand{\al}{\alpha}

\newcommand{\V}{\mathcal{V}}

\newcommand{\IV}{\mathcal{IV}}
\newcommand{\R}{\mathbb{R}}
\newcommand{\N}{\mathbb{N}}
\newcommand{\mH}{\mathcal{H}}
\newcommand{\F}{\mathcal{F}}

\newcommand{\si}{\sigma}

\newcommand{\mZ}{\mathbb{Z}}
\newcommand{\Z}{\mathcal{Z}}

\newcommand{\f}{\mathbf{f}}

\newcommand{\M}{\mathbf{M}}

\newcommand{\bL}{\mathbf{L}}

\newcommand{\mf}{\mathbf{f}}

\newcommand{\mF}{\mathbf{F}}
\newcommand{\mI}{\mathbf{I}}

\newcommand{\bmu}{\boldsymbol \mu}
\newcommand{\btau}{\boldsymbol \tau}
\newcommand{\bleta}{\boldsymbol \eta}

\newcommand{\spt}{\operatorname{spt}}
\newcommand{\dist}{\operatorname{dist}}

\newcommand{\Clos}{\operatorname{Clos}}

\newcommand{\rom}[1]{\expandafter\romannumeral #1}

\title[Min-max Theory for G-Invariant Minimal Hypersurfaces]{Min-max Theory for G-Invariant Minimal Hypersurfaces}
\author{Tongrui Wang}
\address{Department of Mathematics, Nanjing University, Gulou, Nanjing}
\email{DZ1721003@smail.nju.edu.cn}

\begin{document}
\maketitle
\begin{abstract}
In this paper, we consider a closed Riemannian manifold $M^{n+1}$ with dimension $3\leq n+1\leq 7$, and a compact Lie group $G$ acting as isometries on $M$ with cohomogeneity at least $3$. 
After adapting the Almgren-Pitts min-max theory to a $G$-equivariant version, we show the existence of a nontrivial closed smooth embedded $G$-invariant minimal hypersurface $\Sigma\subset M$ provided that the union of non-principal orbits 
forms a smooth embedded submanifold of $M$ with dimension at most $n-2$. 
Moreover, we also build upper bounds as well as lower bounds of $(G, p)$-widths, which are analogs of the classical conclusions derived by Gromov and Guth. 
An application of our results combined with the work of Marques-Neves shows the existence of infinitely many $G$-invariant minimal hypersurfaces when ${\rm Ric}_M>0$ and orbits satisfy the same assumption above. 
\end{abstract}

\section{Introduction}
Min-max theory was proposed by Almgren \cite{almgren1962homotopy}\cite{almgren1965theory} in the 1960s to find minimal hypersurfaces in compact manifolds. 
Subsequently, Pitts \cite{pitts2014existence} and Schoen-Simon \cite{schoen1981regularity} advanced the theory in the 1980s. 
The main idea of the Almgren-Pitts min-max theory is to consider the min-max width defined with all non-trivial continuous closed curves, namely sweepouts, in the $n$-cycle space $\mathcal{Z}_n(M^{n+1}; E)$ where $E$ is $\mathbb{Z}$ or $\mathbb{Z}_2$. 
After constructing stationary varifolds (i.e. weak solutions) to realize the min-max width, they further built the regularity result in closed manifolds with dimension $3\leq n+1\leq 7$. 
Thanks to the famous Almgren's isomorphism theorem in \cite{almgren1962homotopy}: $\pi_1(\mathcal{Z}_n(M;E)) \cong  H_{n+1}(M;E),~E=\mathbb{Z} {\rm ~or~}\mathbb{Z}_2$, 
one can consider maps from a higher parameter space into the $n$-cycle space. 
In recent years, plenty of excellent works were based on these multi-parameter sweepouts. 
For example, Marques-Neves have used this technique to prove the Willmore Conjecture \cite{marques2014min}, and constructed infinitely many minimal hypersurfaces in closed manifolds with positive Ricci curvature \cite{marques2017existence}. 

Meanwhile, there are also other settings of min-max. 
In the min-max setting of C. De~Lellis and D. Tasnady \cite{de2013existence}, a sweepout is a family of {\it generalized hypersurfaces}, which are smooth except for a finite number of points (\cite[Definition 0.2]{de2013existence}). 
Despite the lack of multi-parameter sweepouts to the author's knowledge, some topological properties can be controlled under this setting (for example \cite{ketover2019genus}). 

When a Lie group $G$ acts by isometries on $M$, can min-max theory generate a minimal hypersurface invariant under the action of $G$?
J. Pitts and J. H. Rubinstein firstly announced a version of the equivariant min-max for finite groups acting on $3$-dimensional manifolds in \cite{pitts1987applications} and \cite{pitts1988equivariant} without a full proof. 
It was Daniel Ketover\cite{ketover2016equivariant} who finished the construction of $G$-equivariant min-max for finite $G$ and $n+1=3$. 
Inspired by Ketover's work, a more general version of $G$-equivariant min-max was built by Z. Liu in \cite{liu2021existence} considering a compact connected Lie group $G$ acting as isometries with ${\rm Cohom}(G)\neq 0,2$, on manifolds with dimension $3\leq n+1\leq 7$. 
Liu's paper mainly focused on the existence and regularity results of $G$-invariant minimal hypersurfaces, and the case of ${\rm Cohom}(G)=1$ was treated exceptionally by the classification of the orbit space. 
Meanwhile, both Ketover and Liu have some additional restrictions on the compact Lie group $G$. 
Specifically, Liu needs $G$ to be connected, while Ketover needs the actions of $G$ to be orientation preserving and $M/G$ is an orientable orbifold without boundary. 
We also mention that neither of their results is built under the Almgren-Pitts setting, and thus lacks multi-parameter sweepout analogs.

To go further than the existence and regularity results, we set out the $G$-equivariant min-max theory under the Almgren-Pitts setting in this paper, which allows us to build the infinitely existence result by multi-parameter $G$-sweepouts.

Since the key ingredient of the Almgren-Pitts min-max theory is the Almgren's Isomorphism, one difficulty in building $G$-equivariant min-max theory is to characterize $\pi_1(\mathcal{Z}_n^G(M;\mathbb{Z}_2))$ the fundamental group of the $G$-invariant $n$-cycles space. 
In \cite{marques2021morse}, Marques-Neves used a double cover argument to show the space of $n$-cycles $\mathcal{Z}_n(M;\mathbb{Z}_2)$ is weakly homotopically equivalent to $\mathbb{RP}^\infty$. 
By a $G$-isoperimetric lemma (Lemma \ref{Lem:isoperimetric}), the weak homotopical equivalence also holds for the space of $G$-invariant $n$-cycles $\mathcal{Z}_n^G(M;\mathbb{Z}_2)$. 

Additionally, since discrete mappings are used in the Almgren-Pitts setting, the discretization and interpolation results (\cite[Section 3]{marques2017existence}) are needed to build a bridge between discrete and continuous maps. 
To obtain these results under $G$-invariant restrictions, the main difficulty is to construct a $G$-equivariant triangulation of $M$ together with $G$-equivariant deforming maps (Theorem \ref{Thm:interpolation}). 
The existence of equivariant triangulation is guaranteed by Illman's works \cite{illman1978smooth}\cite{illman1983equivariant}. 
However, since the injectivity radii of orbits generally admit no positive lower bound, the star neighborhood in the triangulation may contain cut points of its center orbit, which leads to more detailed work in constructing $G$-equivariant deforming maps. 

The classification of $\pi_1(\mathcal{Z}_n(M;\mathbb{Z}_2))$ ($\pi_1(\mathcal{Z}_n^G(M;\mathbb{Z}_2))$) implies the existence of a nontrivial ($G$-invariant) stationary varifold as a weak solution for minimal ($G$-invariant) hypersurfaces. 
To get a smooth minimal hypersurface, Pitts \cite{pitts2014existence} proposed the concept of {\em almost minimizing varifolds} and showed the existence, as well as the regularity, of varifolds that are almost minimizing in annuli. 
In the $G$-equivariant setting, the definition of almost minimizing varifolds can be naturally extended by considering the $G$-invariant comparison cycles (Definition \ref{am.varifolds}). 
However, the difference in the metrics used for the definitions leads to a technical obstacle on the existence of almost minimizing varifolds. 
Specifically, the almost minimizing varifolds are defined with the flat metric, while the discrete sweepouts are defined with the mass norm. 
To handle this difficulty, a technical result \cite[Theorem 3.9]{pitts2014existence} was proved by Pitts with lengthy work in \cite[Section 3.4-3.10]{pitts2014existence}. 
But under $G$-invariant restrictions, this technical obstacle is especially problematic since the injectivity radii of orbits generally do not admit a uniformly positive lower bound on $M$. 
Nevertheless, we notice that if $W$ is an open $G$-subset of $M$ so that every orbit in $\Clos(W)$ is principal, 
then there exists a constant $r_0>0$ with $\inf_{p\in \Clos(W)}{\rm Inj}(G\cdot p)\geq r_0$. 
Since the goal is to find varifolds that are $G$-almost minimizing in $G$-annuli, it will be convenient if the $G$-annuli are all contained in $M^{reg}$, where $M^{reg}$ is the union of all principal orbits in $M$. 
Thus we assume that the union of non-principal orbits $M\setminus M^{reg}$ 
is a smooth submanifold of $M$ with dimension at most $n-2$. 
With this assumption, we can regard $M\setminus M^{reg}$ as a whole and consider annuli around $M\setminus M^{reg}$. 
The smoothness of $M\setminus M^{reg}$ is to ensure the smoothness of the tubular neighborhood around it. 
The assumptions on ${\rm dim}(M\setminus M^{reg})\leq n-2$ and ${\rm Cohom}(G)\geq 3$ are used to guarantee the singularity on $M\setminus M^{reg}$ and $G\cdot p$ can be removed by the regularity result in \cite{wickramasekera2014general}. 
Now we present our main results. 
\begin{theorem}\label{main.thm}
	Let $M^{n+1}$ be a closed Riemannian manifold with a compact Lie group $G$ acting as isometries of cohomogeneity ${\rm Cohom}(G)\geq 3$.
	Suppose $2\leq n\leq 6$ and the union of non-principal orbits $M\setminus M^{reg}$ is a smoothly embedded submanifold of $M$ with dimension at most $n-2$. 
	Then there exists a closed smooth embedded $G$-invariant minimal hypersurface $\Sigma^n\subset M^{n+1}$, i.e. $g\cdot x\in \Sigma, \forall x\in\Sigma,g\in G$. 
\end{theorem}
In the work of Marques and Neves \cite{marques2017existence}, a dichotomy theorem was shown using upper bounds of the $p$-width derived by Gromov\cite{gromov1988dimension} and Guth\cite{guth2009minimax}. 
As a corollary, they found infinitely many minimal hypersurfaces in closed manifolds with positive Ricci curvature. 
Combining their procedure with the co-area formula, we can build upper bounds and lower bounds of $(G,p)$-width:

\noindent {\it {\bf (Theorem \ref{Thm-upper-bound}, \ref{Thm-lower-bound-of-width})}
	Let $G$ be a compact Lie group acting as isometries on a closed manifold $M^{n+1}$ with ${\rm Cohom}(G)=l\geq 3$. 
	Then there exist constants $C_1, C_2>0$ depending only on $M,G$, such that
	$$ C_1p^{\frac{1}{l}} \leq  \omega_p^G(M)\leq C_2 p^{\frac{1}{l}},$$
	for all $p\in \mathbb{N}$.   
}

When $G$ is trivial, the cohomogeneity $l$ of $G$ is ${\rm dim}(M)=n+1$. Hence, our results are coincide with the classical ones \cite[Theorem 5.1, 8.1]{marques2017existence}. 
As an application, the dichotomy theorem for $G$-invariant minimal hypersurfaces is still valid:

\noindent {\it {\bf (Theorem \ref{Thm:dichotomy})}
	Under the same assumption of Theorem \ref{main.thm}. 
	Then: 
	\begin{itemize}
		\item[(i)]  either there exists a disjoint collection $\{\Sigma_1,\dots,\Sigma_l \}$ of $l$ closed smooth embedded $G$-connected $G$-invariant minimal hypersurfaces;
		\item[(ii)] or there exist infinitely many closed smooth embedded $G$-connected $G$-invariant minimal hypersurfaces.
		\end{itemize}
}

 \paragraph{Outline}
In Section \ref{preliminary}, we offer some notations and useful lemmas of $G$-invariant cycles. 
In Section \ref{Amgren-Pitt-min-max}, we define the $G$-invariant min-max under the Almgren-Pitts setting. 
The discretization and interpolation results under $G$-invariant restrictions are shown in Section \ref{almgren.homo}. 
After providing the definition and existence of $G$-almost minimizing varifolds in Section \ref{G-a.m.v}, we show the regularity of $G$-almost minimizing varifolds in Section \ref{regul-G-a.m.v}.
In Section \ref{minmax.families}, we generalize the min-max widths and sweepouts into the $G$-invariant setting. 
In Section \ref{Sec-upper-bound}, we build upper bounds of $(G,p)$-widths and show the dichotomy theorem for $G$-invariant minimal hypersurfaces. 
Finally, lower bounds of $(G,p)$-width are shown in Section \ref{Sec-lower-bound-witdh}.

\vspace{0.25em}
{\bf Acknowledgement.} The author would like to thank Prof. Gang Tian for his constant encouragement. 
He also thanks Xin Zhou, Zhiang Wu, Yangyang Li, and Zhenhua Liu for helpful discussions.

\section{Preliminary}\label{preliminary}

In this paper, we always assume $(M,g_{_M})$ to be an orientable connected closed Riemannian $(n+1)$-dimensional manifold, and assume $G$ to be a compact Lie group acting as isometries on $M$ of cohomogeneity ${\rm Cohom}(G)=l\geq 3$. 
Let $\mu$ be a bi-invariant Haar measure on $G$ which has been normalized to $\mu(G)=1$. 
A preliminary introduction to the isometric actions of Lie groups can be found in Appendix \ref{Appendix-good-partition}.

 \subsection{Basic notations}\label{notation} 
We borrow the following notations from \cite{liu2021existence} and \cite{marques2017existence}.
Sometimes, we add $G$- in front of objects meaning they are $G$-invariant:
  \begin{itemize}
    \item[$\bullet$] a $G$-varifold $V$ satisfies $g_{\#} V=V$ for all $g\in G$; 
    \item[$\bullet$] a $G$-current $T$ satisfies $g_{\#} T=T$ for all $g\in G$; 
    \item[$\bullet$] a $G$-vector field $X$ satisfies $g_{*} X=X$ for all $g\in G$; 
    \item[$\bullet$] a $G$-map $F$ satisfies $g^{-1}\circ F\circ g=F,\forall g\in G$, (i.e. $F$ is $G$-equivariant); 
    \item[$\bullet$] a $G$-set ($G$-neighborhood) is an (open) set which is a union of orbits. 
  \end{itemize}
  We will also sometimes add a subscript or superscript `$G$' to signify $G$-invariance: 
	\begin{itemize}
		\item[$\bullet$] $\pi$: the projection $\pi:M\mapsto M/G$ defined by $p \mapsto [p]$;
		\item[$\bullet$] $B_\rho^G(p),~\overline{B}_\rho^G(p)$:
		open and closed geodesic tubes with radius $\rho$ around $G\cdot p$;
		\item[$\bullet$] $\mathfrak{X}^G(M)$: the space of $G$-vector fields on $M$; 
		\item[$\bullet$] $\an(p,s,t)$: the open tube $B_t^G(p)\setminus \overline{B}^G_s(p)$; 
		\item[$\bullet$] $\ann_r(p)$: the set $\{\an(p,s,t) : 0<s<t<r\}$; 
		\item[$\bullet$] $T_qG\cdot p$: the tangent space of the orbit $G\cdot p$ at some point $q\in G\cdot p$; 
		\item[$\bullet$] ${\bf N}_qG\cdot p$: the normal vector space of the orbit $G\cdot p$ in $M$ at some point $q\in G\cdot p$; 
		\item[$\bullet$] ${\bf N}(G\cdot p)$: the normal vector bundle of the orbit $G\cdot p$ in $M$; 
		\item[$\bullet$] $M^{reg}$: the union of orbits with principal orbit type. 
	\end{itemize}

For $p\in M$, denote $G_p:=\{g\in G: g\cdot p=p\}$ as the isotropy group of $p$, and $(G_p)$ as the conjugacy class of $G_p$ in $G$. 
Then the orbit type of $G\cdot p$ is determined by $(G_p)$, which gives a stratification structure on $M$ as well as $M/G$. 
The orbit space $M/G$ is a Hausdorff metric space with the induced metric $d_{M/G}$. 
Although $M/G$ may not be a smooth manifold, it can be locally modeled by slices (see \cite{verona1979triangulation}). 
Additionally, $M/G$ is triangulable which also provides an equivariant triangulation of $M$ \cite{illman1978smooth}\cite{illman1983equivariant}. 
More notations on Lie group actions are collected in Appendix \ref{Appendix-good-partition}. 



Just like in \cite{marques2017existence}, the spaces we will work with in this paper are:
\begin{itemize}
\item[$\bullet$] the space ${\bf I}_k(M;\mathbb{Z}_2)$ (${\bf I}_k^G(M;\mathbb{Z}_2)$) of $k$-dimensional ($G$-invariant) mod $2$ flat chains in $\mathbb{R}^L$ with support contained in $M$ (see \cite[4.2.26]{federer2014geometric} for more details);
\item[$\bullet$] the space ${\mathcal Z}_k(M;\mathbb{Z}_2)$ (${\mathcal Z}_k^G(M;\mathbb{Z}_2)$) of ($G$-invariant) mod 2 flat chains  $T \in {\bf I}_k(M;\mathbb{Z}_2)$ ($T \in {\bf I}_k^G(M;\mathbb{Z}_2)$) with $\partial T=0$;
\item[$\bullet$] the closure $\mathcal{V}_k(M)$ ($\mathcal{V}_k^G(M)$), in the weak topology, of the space of $k$-dimensional ($G$-invariant) rectifiable varifolds in $\mathbb{R}^L$ with support contained in $M$. The space of integral rectifiable $k$-dimensional ($G$-invariant) varifolds with support contained in $M$ is denoted by $\mathcal{IV}_k(M)$ ($\mathcal{IV}_k^G(M)$).
\end{itemize}

For any $T\in {\bf I}_k(M;\mathbb{Z}_2)$,  we denote $|T|$ and $\|T\|$ to be the integral varifold and the Radon measure induced by $T$. 
Given $V\in \mathcal{V}_k(M)$, we use $\|V\|$ to denote the Radon measure induced by $V$.

As for metrics of these spaces, the  {\it flat semi-norm} $\mathcal{F}$ and the {\it mass} ${\bf M}$ for chains in ${\bf I}_k(M;\mathbb{Z}_2)$ are defined in \cite[Page 423]{federer2014geometric} and  \cite[Page 426]{federer2014geometric},  respectively.
The  ${\bf F}$-{\it metric} on $\mathcal{V}_k(M)$ is defined in  {Pitts's book} \cite[Page 66]{pitts2014existence}, which induces the varifold weak topology on bounded subsets of $\mathcal{V}_k(M)$, and satisfies
$${\bf F}(|S|,|T|)\leq{\bf M}(S,T),$$
for all $S,T\in {\bf I}_k(M;\mathbb{Z}_2)$.
For $\mathcal{A,B}\subset \mathcal{V}_k(M)$, we also define
  ${\bf F}(\mathcal{A},\mathcal{B}):=\inf\{{\bf F}(V,W):V\in \mathcal{A}, ~ W\in \mathcal{B}\}.$
Finally,  the ${\bf F}$-{\it metric} on ${\bf I}_k(M;\mathbb{Z}_2)$ is defined by
$$ {\bf F}(S,T):=\mathcal{F}(S-T)+{\bf F}(|S|,|T|).$$

We assume ${\bf I}_k(M;{\bf v};\mathbb{Z}_2)$, ${\mathcal Z}_k(M;{\bf v};\mathbb{Z}_2)$ have the topology induced by the metric ${\bf v}$, where ${\bf v}$ could be ${\bf M}$, ${\bf F}$, or omitted for $\mathcal{F}$. 
 For $G$-invariant chains and varifolds, we can similarly define those metrics as above. Indeed, ${\bf I}_k^G(M;{\bf v};\mathbb{Z}_2) $ is a closed subspace of ${\bf I}_k(M;{\bf v};\mathbb{Z}_2)$.

\subsection{$G$-invariant currents and varifolds}\label{Sec-G-current}
In this subsection, we show some useful lemmas for $G$-currents and $G$-varifolds.

\begin{lemma}[Compactness theorem for ${\bf I}_k^G(M;\mathbb{Z}_2)$]\label{Lem:compactness for G-current}
	For any $C>0$, the set
$$\{ T\in {\bf I}_k^G(M;\mathbb{Z}_2) : {\bf M}(T)+{\bf M}(\partial T) \leq C\} $$ is compact under the flat semi-norm $\mathcal{F}$.
\end{lemma}

\begin{proof}
	By compactness theorem \cite[27.3]{simon1983lectures}, for any sequence $\{T_i\}_{i=1}^\infty$ in ${\bf I}_k^G(M;\mathbb{Z}_2)$ with ${\bf M}(T)+{\bf M}(\partial T) \leq C$, 
	there exists a limit $T\in{\bf I}_k(M;\mathbb{Z}_2) $ of $T_i$ under the flat topology up to a subsequence. Since $G$ acts as isometries on $M$, we have ${\bf M}(S) = {\bf M}(g_\# S)$, for all $S \in{\bf I}_k(M;\mathbb{Z}_2),g\in G$. 
	This implies $\mathcal{F}(T-T_i)=\mathcal{F}(g_\#T-g_\#T_i) $, and $ T = \lim_{i\rightarrow \infty}T_i = \lim_{i\rightarrow \infty}g_\#T_i =g_\#T $ for all $g\in G$. 
	Thus $T\in {\bf I}_k^G(M;\mathbb{Z}_2)$.
\end{proof}

\begin{remark}\label{Rem:compactness for G-varifold}
	A similar argument can show the compactness theorem for $G$-varifolds also holds as the classical one. 
	Specifically, for any $C>0$, the set
$$\{ V\in \mathcal{V}_k^G(M) : \|V\|(M) \leq C\} $$ is compact under the weak topology of varifolds.
\end{remark}

The following three statements follow directly from definitions (\cite[Chapter 6]{simon1983lectures}).
\begin{lemma}[Restrict to $G$-sets]\label{Lem:restrict}
	For any $k$-dimensional $G$-current $T$ and Borel $G$-set $U$, we have $T\llcorner U$ is also a $k$-dimensional $G$-current. 
\end{lemma}

\begin{lemma}[$G$-equivariant pushing forward]\label{Lem:G-push}
	Let $F$ be a $G$-equivariant Lipschitz proper map from $U$ to $V$, where $U,V\subset M$ are open $G$-sets. 
	Suppose $T\in {\bf I}_k^G(U;\mathbb{Z}_2)$, then  $F_\#T \in {\bf I}_{k}^G(V;\mathbb{Z}_2)$.
\end{lemma}

\begin{lemma}[$G$-invariant slice]\label{Lem:slice}
	Suppose $f$ is a $G$-invariant Lipschitz function on $M$, $T\in {\bf I}_k^G(M;\mathbb{Z}_2)$, then for almost all $t\in\mathbb{R}$, the slice of $T$ by $f$ at $t$ exists and $\langle T,f,t\rangle \in {\bf I}_{k-1}^G(M;\mathbb{Z}_2)$.
\end{lemma}

Since $G$ acts as isometries on $M$, the distant function ${\rm dist}(G\cdot p,\cdot)$ to an orbit $G\cdot p$ is a $G$-invariant Lipschitz function. 
Furthermore, if we define the $G$-actions on the normal bundle ${\bf N}(G\cdot p)$ of $G\cdot p$ as $g\cdot(p,v):=(g\cdot p, g_*v)$, for all $g\in G$, $(p,v)\in {\bf N}(G\cdot p)$. 
Then the normal exponential map ${\rm exp}^\perp_{G\cdot p}$ gives a $G$-equivariant diffeomorphism in a $G$-neighborhood of the zero section in ${\bf N}(G\cdot p)$.

The following lemma is a crucial ingredient of the Almgren's Isomorphism as well as plenty of constructions. 
Note we only prove for codimensional-one $G$-cycles since we need the constancy theorem \cite[Theorem 26.27]{simon1983lectures}. 
\begin{lemma}[$G$-invariant isoperimetric lemma]\label{Lem:isoperimetric}
	There exist positive constants $\nu_M$ and $C_M$ such that for any $T_1,T_2\in {\mathcal Z}_n^G(M;\mathbb{Z}_2)$ with
	$$\mathcal{F}(T_1-T_2)<\nu_M, $$
	there exists a unique $Q\in {\bf I}_{n+1}^G(M;\mathbb{Z}_2)$ such that
	\begin{itemize}
    	\item[$\bullet$] $\partial Q = T_1-T_2$;
    	\item[$\bullet$] ${\bf M}(Q)\leq C_M\cdot\mathcal{F}(T_1-T_2)$.
  	\end{itemize}
\end{lemma}

\begin{proof}
	By \cite[Lemma 3.1]{marques2017existence}, one can choose positive numbers $\nu_M,C_M$ such that, for any such $T_1,T_2\in {\mathcal Z}_n^G(M;\mathbb{Z}_2)$, there exists a unique isoperimetric choice $Q\in {\bf I}_{n+1}(M;\mathbb{Z}_2)$ with $\partial Q = T_1-T_2$ and ${\bf M}(Q)\leq C_M\cdot\mathcal{F}(T_1-T_2)$. We only need to check that $Q$ is $G$-invariant.
	Indeed, we have
	$$ \partial(g_{\#}Q) = g_{\#}(\partial Q) = g_{\#}(T_1-T_2) = T_1-T_2,$$
	for all $g\in G$.
	Since $G$ acts as isometries on $M$, it's clear that ${\bf M}(Q)={\bf M}(g_{\#}Q) $. Thus we get $g_{\#}Q=Q$ by the uniqueness of the isoperimetric choice, which implies $Q\in{\bf I}_{n+1}^G(M;\mathbb{Z}_2)$.
\end{proof}

For a $G$-varifold $V\in\V_k^G(M)$, we say $V$ is {\em $G$-stationary in $M$} if 
$$\delta V(X)=0,\quad \forall X\in\mathfrak{X}^G(M).$$
Inspired by \cite{de2013existence}, Z. Liu has shown the equivalence between $G$-stationary and stationary for $G$-varifolds in \cite[Lemma 2.2]{liu2021existence}:

\begin{lemma}[Z. Liu]\label{stationary}
	 A $G$-varifold $V$ is $G$-stationary in $M$ if and only if it is stationary in $M$.
\end{lemma}

Let $U\subset M$ be an open $G$-set, and $\Sigma\subset\partial U$ be a smoothly embedded compact $2$-sided minimal $G$-hypersurface. 
Recall that $\Sigma$ is stable if and only if the first eigenvalue of the Jacobi operator $L$ is non-negative, where $L$ acts on $\mathfrak{X}^\perp (\Sigma)$ the space of smooth sections of the normal bundle vanishing on $\partial \Sigma$. 
We can replace $\mathfrak{X}^\perp (\Sigma)$ by the $G$-invariant sections space $\mathfrak{X}^{\perp,G} (\Sigma)$ and define the {\em $G$-eigenvalues} as well as the {\em $G$-stability} similarly. 
The following lemma shows the equivalence between the stability and $G$-stability of such boundary type minimal $G$-hypersurfaces.

\begin{lemma}[$G$-stability]\label{Lem:G-stable}
	Let $\Sigma$ be a compact smooth embedded $G$-invariant minimal hypersurface in $M$. 
	If $\Sigma$ is a part of the boundary of an open $G$-subset $U\subset M$. 
	Then $\Sigma$ is $G$-stable if and only if it is stable. 
\end{lemma}

\begin{proof}
	Since $\Sigma$ is a part of the boundary of an open $G$-subset $U\subset M$, it must be $2$-sided.
	Additionally, there exists a unit normal vector field $\nu$ on $\Sigma$ pointing inward $U$.  
	Due to the facts that $U$ is an open $G$-set and $G$ acts as isometries, we have $g_*\nu =\nu$ for all $g\in G$, and thus $\nu\in \mathfrak{X}^{\perp,G}(\Sigma)$. 
	Let $X=u\nu$ be the first eigenvector field of the Jacobi operator $L$ of $\Sigma$ with $LX=-\lambda_1 X$. 
	Denote $\{\phi_t\}$ to be the diffeomorphisms generated by $X$. 
	Then the vector field corresponding to $\phi_{g,t}=g^{-1}\circ \phi_t\circ g$ is 
	\begin{eqnarray*}
		 X_g(p) = \frac{d}{dt}\Big\vert_{t=0}~ \phi_{g,t}(p) =  (g^{-1})_*(\frac{d}{dt}\Big\vert_{t=0}~ \phi_{t}\circ g(p)) =  (g^{-1})_*(X(g(p))) = u(g\cdot p)\nu(p),
	\end{eqnarray*}
	i.e. $X_g=(g^{-1})_* X$. Define then 
	\begin{equation}\label{Eq:average}
		X_G(p):=\int_G X_g(p) ~d\mu(g)=\Big(\int_G u(g\cdot p)~d\mu(g)\Big) \nu(p),
	\end{equation}
	where the integral is carried out in $T_pM$. 
	As in \cite{liu2021existence}, $X_G$ is $G$-invariant. 
	Since the first eigenfunction $u$ does not change sign on $\Sigma$, we have $X_G\neq 0$. 
	We also mention that $X_G, X_g $ are all contained in $\mathfrak{X}^{\perp} (\Sigma)$ since $G$ acts as isometries and $\Sigma$ is $G$-invariant. Thus $0\neq X_G\in \mathfrak{X}^{\perp,G} (\Sigma)$. 
	Moreover, a direct computation shows:
		\begin{eqnarray}\label{Eq:eigenvector field}
		LX_G &=& \int_G L(g^{-1}_*X) ~d\mu(g) =  \int_G g^{-1}_*(LX) ~d\mu(g)\nonumber
		\\
		&=& \int_G g^{-1}_*(-\lambda_1 X) ~d\mu(g) = -\lambda_1 X_G. 
	\end{eqnarray}
	Hence, $X_G\in\mathfrak{X}^{\perp,G} (\Sigma)$ is a $G$-invariant first eigenvector field of $L$. 
	This implies that 
	$ \Sigma$ is stable $ \Leftrightarrow \lambda_1\geq 0 \Leftrightarrow \Sigma$ is $G$-stable.
\end{proof}

\section{Equivariant Almgren-Pitts Min-max Theory}\label{Amgren-Pitt-min-max}

This section is parallel to \cite[Section 2]{marques2017existence}. 
In the following content, let $X$ be a cubical subcomplex of $I^m=[0,1]^m$. 

\subsection{Cubical Complex}\label{Sec-sub-cubical-complex}
For any $j\in \mathbb{N}$, we denote $I(1,j)$ to be the cube complex on $I=[0,1]$ with $1$-cells 
$[0,3^{-j}], [3^{-j},2 \cdot 3^{-j}],\dots,[1-3^{-j}, 1],$ 
and $0$-cells (vertices)
$[0], [3^{-j}],\dots,[1].$
The cell complex on $I^m:=[0,1]^m$ is denoted as
$I(m,j):=I(1,j)\otimes\dots \otimes I(1,j)$, ($m$ times). 
Let $\partial: I(m, j)\rightarrow I(m, j)$ be the boundary homeomorphism given by
$$\partial(\al_1\otimes\cdots\otimes\al_m):=\sum_{i=1}^m(-1)^{\si(i)}\al_1\otimes\cdots\otimes \partial \al_i\otimes\cdots\otimes\al_m,$$
where $\si(i):=\sum_{l<i}\dim(\al_l)$, $\partial[a, b]:=[b]-[a]$, 
and $\partial[a]:=0$. 
Define $\alpha:=\alpha_1 \otimes \cdots\otimes \alpha_m$ to be a {\em $q$-cell} of $I(m,j)$ if $\alpha_i$ is a cell
of $I(1,j)$ for each $i$, and $\sum_{i=1}^m {\rm dim}(\alpha_i) =q$.

We denote $X(j)$ as the union of all cells of $I(m,j)$ whose support is contained in some cells of $X$. 
Denote then $X(j)_q$ as the set of all $q$-cells in $X(j)$. 
The {\em distance} between two vertices $x, y\in X(j)_0$ is defined by ${\bf d}(x,y) := 3^j\cdot\sum_{i=1}^m|x_i-y_i|$, and $x,y$ are said to be {\em adjacent} if ${\bf d}(x,y)=1$.
Given $i,j\in \mathbb{N}$, define ${\bf n}(i,j):X(i)_0\rightarrow X(j)_0$ as a map so that ${\bf n}(i,j)(x)$ is the unique element of $X(j)_0$ with 
$${\bf d}(x,{\bf n}(i,j)(x))=\inf\{{\bf d}(x,y): y\in X(j)_0\}.$$
For any map $\phi:X(j)_0\rightarrow  \mathcal{Z}_n^G(M;\mathbb{Z}_2)$, the {\em fineness} of $\phi$ is defined as
$${\bf f}(\phi):=\sup\left\{{\bf M}(\phi(x)-\phi(y)) : {\bf d}(x,y)=1,~ x,y\in  X(j)_0\right\}.$$

\subsection{Homotopy notions in discrete settings}
Let $\phi_i:X(k_i)_0\rightarrow  \mathcal{Z}_n^G(M;\mathbb{Z}_2)$, $i=1,2$, be two discrete mappings. We say $\phi_1$ and $\phi_2$ are {\em $X$-homotopic in $\mathcal{Z}_n^G(M;{\bf M};\mathbb{Z}_2)$ with fineness $\delta$}  if we can find a map
$$\psi: I(1,k)_0\times X(k)_0\rightarrow  \mathcal{Z}_n^G(M;\mathbb{Z}_2)$$
for some $k\in \mathbb{N}$ satisfying ${\bf f}(\psi)<\delta$ and 
$\psi([i-1],x)=\phi_i({\bf n}(k,k_i)(x))$, for all $x\in X(k)_0$, $i=1,2$.

\begin{definition}\label{Def:homotopy sequence}
	A sequence of mappings $S=\{\phi_i\}_{i=1}^\infty$ with
	$$\phi_i:X(k_i)_0\rightarrow \mathcal{Z}_n^G(M;\mathbb{Z}_2)$$ 
	is called an
 	$$\mbox{{\em $(X,{\bf M})$-homotopy sequence of mappings into $\mathcal{Z}_n^G(M;{\bf M};\mathbb{Z}_2)$}},$$ 
	if $\phi_i$ and $\phi_{i+1}$ are $X$-homotopic in  $\mathcal{Z}_n^G(M;{\bf M};\mathbb{Z}_2)$ with fineness $\delta_i$ such that
	\begin{itemize}
		\item[(i)] $\lim_{i\rightarrow\infty} \delta_i=0$;
		\item[(ii)]$\sup\{{\bf M}(\phi_i(x)):x\in X(k_i)_0, i\in \mathbb{N}\}<+\infty.$
	\end{itemize}
\end{definition}

\begin{definition}\label{Def:homotopy class}
	Given $S^j=\{\phi^j_i\}_{i=1}^\infty$, $j=1,2$, as two $(X,{\bf M})$-homotopy sequences of mappings into $\mathcal{Z}_n^G(M;{\bf M};\mathbb{Z}_2)$, we say  that {\it $S^1$ is $G$-homotopic to $S^2$} if there exists a sequence $\{\delta_i\}_{i\in \mathbb{N}}$ such that
 \begin{itemize}
\item[(i)] $\phi^1_i$  is $X$-homotopic to $\phi^2_i$ in  $\mathcal{Z}_n^G(M;{\bf M};\mathbb{Z}_2)$ with fineness $\delta_i$;
\item[(ii)] $\lim_{i\rightarrow\infty} \delta_i=0.$
 \end{itemize}
Moreover, we call  the equivalence class
of any such sequence an
$$\mbox{{\it $(X,{\bf M})$-homotopy class of mappings into $ \mathcal{Z}_n^G(M;{\bf M};\mathbb{Z}_2)$},} $$
and denote $[X,\mathcal{Z}_n^G(M;{\bf M};\mathbb{Z}_2)]^{\#}$ to be the set of all such equivalence classes.
\end{definition}

For any $\Pi \in  [X,\mathcal{Z}_n^G(M;{\bf M};\mathbb{Z}_2)]^{\#}$, define the function ${\bf L}: \Pi\rightarrow [0,+\infty]$ by 
\begin{eqnarray*}
	{\bf L}(S) := \limsup_{i\rightarrow\infty} \max_{x\in \mathrm{dmn}(\phi_i)} {\bf M}(\phi_i(x)),
\end{eqnarray*}
where $S=\{\phi_i\}_{i\in \mathbb{N}}\in\Pi$.

\begin{definition}[Width]\label{Def:width continuous}
	The {\it width} of $\Pi \in  [X,\mathcal{Z}_n^G(M;{\bf M};\mathbb{Z}_2)]^{\#}$ is defined by
$${\bf L}(\Pi):=\inf_{S\in \Pi}{\bf L}(S).$$
Additionally, we say  $S\in \Pi$ is a {\it critical sequence} for $\Pi$ if ${\bf L}(S)={\bf L}(\Pi).$
\end{definition}

By a diagonal argument in \cite[Lemma 15.1]{marques2014min}, there exists a critical sequence for all $\Pi \in  [X,\mathcal{Z}_n^G(M;{\bf M};\mathbb{Z}_2)]^{\#}$.

\begin{definition}[Critical set]
	For any $S=\{\phi_i\}_{i\in \mathbb{N}}\in \Pi$, we define the {\em image set} of $S$ as the compact subset ${\bf K}(S)\subset\mathcal{V}^G_n(M)$ given by
	\begin{multline*}
		{\bf K}(S):=\{V \in \mathcal{V}^G_n(M) : V=\lim_{j\rightarrow\infty}|\phi_{i_j}(x_j)|\mbox{ for some sequence}~ i_1<i_2<\dots\\
		\mbox{ and some }x_j\in \mathrm{dmn}(\phi_{i_j})\}.
	\end{multline*}
	If $S$ is a critical sequence for $\Pi$, the {\em critical set} ${\bf C}(S)$ of $S$ is defined by
	$${\bf C}(S):=\{V\in {\bf K}(S) : \|V\|(M)={\bf L}(\Pi)\}.$$
\end{definition}

\subsection{Homotopy notions in continuous settings}\label{Sec:homotopy continuous}
The previous notions are discrete analogs for the usual continuous homotopy notions. 
In this subsection, we collect some homotopy notations under the $G$-equivariant continuous setting, which are generalized from \cite[Section 3]{marques2015morse}. 
Bold symbols and capital letters are used for notations in the continuous setting to distinguish them from symbols in the discrete case.

In the following paper, we denote $\Phi:X\to \Z_{n}^G(M;{\bf v};\mZ_2)$ to be a map from a cubical subcomplex $X\subset I^m$ to $\Z_{n}^G(M;\mZ_2)$ which is continuous in the metric ${\bf v} = \F$, $\M$ or $\mF$. 

\begin{definition}\label{Def:homotopy continuous}
	Let $\Phi_i: X\to \Z_{n}^G(M;\mF;\mZ_2)$, $i=1,2$, be two $\mF$-continuous maps. 
	We say $\Phi_1$ is {\em $G$-homotopic} to $\Phi_2$ if there exists an $\F$-continuous map $\Psi:I\times X\to \Z_{n}^G(M;\F;\mZ_2)$ so that $\Psi(0,x)=\Phi_1(x),~\Psi(1,x)=\Phi_2(x)$, for all $x\in X$. 
	Moreover, we denote ${\bf \Pi}$ to be a {\em continuous $G$-homotopy class}, and denote $\big[X,\Z_{n}^G(M;\mF;\mZ_2)\big]$ to be the set of all such $G$-homotopy classes.
\end{definition}


Given ${\bf \Pi}\in  \big[X,\Z_{n}^G(M;\mF;\mZ_2)\big]$, we denote ${\bf L}(\Phi) := \sup_{x\in X}\M(\Phi(x))$ for every $\Phi\in{\bf \Pi}$. 
For any sequence $\{\Phi_i\}_{i\in\N}\subset{\bf \Pi}$, let
$${\bf L}(\{\Phi_i\}_{i\in\N}):=\limsup_{i\to\infty}\sup_{x\in X}\M(\Phi_i(x)).$$

\begin{definition}[Width]\label{Def:width}
	Let ${\bf \Pi}\in \big[X, \Z_{n}^G(M;\mF;\mZ_2)\big]$ be a continuous $G$-homotopy class. 
	The {\em width} of ${\bf \Pi}$ is defined by
	$${\bf L}({\bf \Pi}):=\inf_{\Phi\in {\bf \Pi}}{\bf L}(\Phi).$$
	Additionally, we say $\{\Phi_i\}_{i\in\N}\subset {\bf \Pi}$ is a {\it min-max sequence} for ${\bf \Pi}$ if 
	${\bf L}(\{\Phi_i\}_{i\in\N})={\bf L}({\bf \Pi}).$
\end{definition}

The existence of min-max sequence for any ${\bf \Pi}\in \big[X, \Z_{n}^G(M;\mF;\mZ_2)\big]$ is obvious by taking an ${\bf L}$-minimizing sequence.

\begin{definition}[Critical set]\label{Def:critical set}
	Let ${\bf \Pi}\in \big[X, \Z_{n}^G(M;\mF;\mZ_2)\big]$ be a continuous $G$-homotopy class, and $\{\Phi_i\}_{i\in \mathbb{N}}$ be any sequence in  ${\bf \Pi}$. 
	The {\em image set} of $\{\Phi_i\}$ is defined as a compact subset ${\bf K}(\{\Phi_i\}_{i\in\N})\subset\mathcal{V}^G_n(M)$ given by
	\begin{multline*}
		{\bf K}(\{\Phi_i\}_{i\in\N}):=\{V\in\V^G_n(M) :V=\lim_{j\rightarrow\infty}|\Phi_{i_j}(x_j)|\mbox{ for some sequence}~ i_1<i_2<\dots\\
		\mbox{ and some }x_j\in X\}.
	\end{multline*}
	If $\{\Phi_i\}_{i\in\N}$ is a min-max sequence for ${\bf \Pi}$, then the {\em critical set of $\{\Phi_i\}_{i\in\N}$} is defined as 
	$${\bf C}(\{\Phi_i\}_{i\in\N}) :=\{V\in {\bf K}(\{\Phi_i\}_{i\in\N}) : \|V\|(M)={\bf L}({\bf \Pi})\}.$$
\end{definition}



\section{Interpolation Results and Pull-tight}\label{almgren.homo}

The discretization and interpolation results (\cite[Section 13,14]{marques2014min}) build a bridge between discrete and continuous maps. 
Almgren used integer coefficients as the parameter space in \cite{almgren1962homotopy}, but Marques-Neves have pointed out in \cite[Section 3]{marques2017existence} that everything extends to the setting of $\mathbb{Z}_2$ coefficients. 
In this section, we generalize these results to the $G$-invariant setting.

To begin with, for any continuous map $\Phi:X\rightarrow \mathcal{Z}_n(M;\mathbb{Z}_2)$, let 
$${\bf m}(\Phi,r):=\sup\{\|\Phi(x)\|(B_r(p)) : x\in X,~p\in M  \}.$$
We say $\Phi$ {\em has no concentration of mass} if 
$\lim_{r\rightarrow 0}{\bf m}(\Phi,r)=0.$ 
In \cite[Lemma 3.5]{marques2017existence} and \cite[Lemma 15.2]{marques2014min}, Marques-Neves have shown $\Phi$ has no concentration of mass and $\sup_{x\in X}{\bf M}(\Phi(x))<\infty$, if $\Phi$ is continuous in the mass norm or the $\mF$-metric. 
We now modify this definition as follows:

\begin{definition}\label{Def:no concent mass on orbit}
	Let $\Phi:X\rightarrow \mathcal{Z}_n^G(M;\mathbb{Z}_2)$ be an $\F$-continuous map. 
	Define 
	$${\bf m}^G(\Phi,r):=\sup\{\|\Phi(x)\|(B_r^G(p)) : x\in X,~p\in M  \}.$$
	Then $\Phi$ is said to {\it have no concentration of mass on orbits} if 
	$$\lim_{r\rightarrow 0}{\bf m}^G(\Phi,r)=0.$$
\end{definition}

This is a mild technical condition, and we also have the following lemma. 

\begin{lemma}\label{Lem:mass continu no concent}
	Suppose $\Phi:X\rightarrow \mathcal{Z}_n^G(M;\mathbb{Z}_2)$ is a map continuous in the mass norm or the ${\bf F}$-metric, 
	then $\sup_{x\in X}{\bf M}(\Phi(x))<\infty$, and $\Phi$ has no concentration of mass on orbits. 
\end{lemma}

\begin{proof}
	By the compactness of $X$ and the fact that $X\to {\bf M}(\Phi(\cdot))$ is a continuous map, it's clear that $\sup_{x\in X}{\bf M}(\Phi(x))$ is bounded.
	
	Fix $\epsilon>0$. Then for any $x\in X$ and $p\in M$, since ${\rm dim}(G\cdot p)\leq n+1-l\leq n-2$, there exists a positive number $r=r(x,p,\epsilon)>0$ such that
	$\|\Phi(x)\|(B^G_r(p))\leq \frac{\epsilon}{2}. $
	Noting $\Phi$ is continuous in the mass norm, there exists a neighborhood $U_{x,p}\subset X$ of $x$ satisfying $\|\Phi(y)\|(B^G_r(p))\leq \epsilon$, for all $y\in U_{x,p} $.
	 Thus $\{U_{x,p}\times B^G_{r(x,p,\epsilon)}(p) \}$ forms an open cover of $X\times M$. 
	 After taking a finite cover $\{U_{x_i,p_i}\times B^G_{r(x_i,p_i,\epsilon)}(p_i) \}_{i=1}^k$ by compactness, we can choose $\tilde{r}<\min_{i=1,\dots,k} r_i$.
	 Hence, we have $\|\Phi(x)\|(B^G_{\tilde{r}}(p))\leq \epsilon$ for all $x\in X ,p\in M$, implying $\Phi$ has no concentration of mass on orbits.
	 
	 Noting that $\lim_{i\rightarrow\infty} {\bf F}(T_i,T)=0$ if and only if $\lim_{i\rightarrow\infty}{\bf M}(T_i)={\bf M}(T)$ and $\lim_{i\rightarrow\infty}\mathcal{F}(T_i-T)=0$, for $T_i,T\in \mathcal{Z}_n^G(M;\mathbb{Z}_2)$ (see \cite[Page 68]{pitts2014existence}), a same result can be made for ${\bf F}$-continuous maps by the above arguments. 
\end{proof}

\subsection{Discretization and interpolation results}\label{Sec-interpolation-discretization}
The following discretization result generates an $(X,{\bf M})$-homotopy sequence of mappings into $\mathcal{Z}_n^G(M;{\bf M};\mathbb{Z}_2)$ from a flat continuous map with bounded mass and no concentration of mass on orbits.

\begin{theorem}\label{Thm:discritization}
	Let $\Phi:X\subset I^m \rightarrow \mathcal{Z}_n^G(M;\mathbb{Z}_2)$ be a continuous map in the flat topology with no concentration of mass on orbits and $\sup_{x\in X}\M(\Phi(x))<\infty$.
	Then there exists a sequence of maps
	$$\phi_i:X(k_i)_0 \rightarrow \mathcal{Z}_n^G(M;\mathbb{Z}_2),$$
	with $k_i<k_{i+1}$, and positive numbers $\{\delta_i\}_{i\in\mathbb{N}}$ converging to zero such that
	\begin{itemize}
		\item[(i)] $S=\{\phi_i\}_{i\in\mathbb{N}}$ is an $(X,{\bf M})$-homotopy sequence of mappings into $\mathcal{Z}_n^G(M;{\bf M};\mathbb{Z}_2)$ with ${\bf f}(\phi_i)<\delta_i$;
		\item[(ii)] $\sup\{\mathcal F(\phi_i(x)-\Phi(x)) : x\in X(k_i)_0\}\leq \delta_i;$
		\item[(iii)] there exists a sequence $l_i\to\infty$ so that for any $x\in X(k_i)_0$, 
		$$\M(\phi_i(x))\leq \sup\{\M(\Phi(y)) : x,y\in\alpha,~\alpha\in X(l_i)_m \}+\delta_i, $$
		which implies ${\bf L}(\{\phi_i\}_{i\in\N}) \leq \sup\{{\bf M}(\Phi(x)) : x\in X\}+\delta_i.$
	\end{itemize}
\end{theorem}
\begin{proof}
	The proof is essentially the same as \cite[Theorem 13.1]{marques2014min}. 
	First, we point out that the proof of \cite[Lemma 13.4]{marques2014min} is fully feasible under $G$-invariant restrictions by using Lemma \ref{Lem:compactness for G-current}, \ref{Lem:slice}, \ref{Lem:isoperimetric}, and \ref{Lem:mass continu no concent} in place of those lemmas for usual cycles. 
	Other constructions in the proof of \cite[Theorem 13.1]{marques2014min} are combinatorial, which can be directly adapted to the $G$-invariant setting. 
\end{proof}

The following theorem shows how to generate a $G$-invariant ${\bf M}$-continuous map from a discrete $G$-invariant mapping with small fineness. 
\begin{theorem}\label{Thm:interpolation}
	There exist positive constants $C_0=C_0(M,G,m)$ and $\delta_0=\delta_0(M,G)$  so that if $Y$ is a cubical subcomplex of $I(m,k)$ and
	$$\phi:Y_0\rightarrow \mathcal{Z}_n^G(M;\mathbb{Z}_2)$$
	has ${\bf f}(\phi)<\delta_0$, then there exists a map
	$$ \Phi:Y\rightarrow \mathcal{Z}_n^G(M;{\bf M};\mathbb{Z}_2)$$
	continuous in the ${\bf M}$-norm satisfying:
	\begin{itemize}
		\item[(i)] $\Phi(x)=\phi(x)$ for all $x\in Y_0$;
		\item[(ii)] if $\alpha$ is a $j$-cell in $Y_j$, then $\Phi$ restricted to $\alpha$ depends only on the values of $\phi$ assumed on the vertices of  $\alpha$;
		\item[(iii)] $\sup\{{\bf M}(\Phi(x) - \Phi(y)) : x,y\mbox{ lie in a common cell of } Y\}\leq C_0{\bf f}(\phi).$
	\end{itemize}
\end{theorem}

We call the map $\Phi$ given by Theorem \ref{Thm:interpolation} {\em the Almgren $G$-extension} of $\phi$. 

\begin{proof}
	We only consider the case of $Y=I(m,0)$, and for general $Y$, the conclusion can be made as in \cite[Theorem 14.2]{marques2014min}. 
	By Proposition \ref{Lem:isoperimetric}, we can choose $\delta_0$ small enough such that for any $\beta\in I(m,0)_1$ there is a unique $G$-isoperimetric choice $Q(\beta)$ for $\phi(\partial \beta)$. 
	Following the proof of \cite[Theorem 14.1]{marques2014min}, we are going to construct $G$-equivariant deforming and cutting maps.

	Let $\triangle$ be a Lipschitz triangulation of $M/G$ which generates a $G$-equivariant triangulation $\widetilde{\triangle} := \{ \pi^{-1}(\sigma) : \sigma\in\triangle \}$ of $M$. 
	For any $\tilde{\sigma} = \pi^{-1}(\sigma)\in\widetilde{\triangle}$, define $\partial\tilde{\sigma} := \pi^{-1}(\partial \sigma)$. 
	Let $[p_\sigma]$ be the center point of $\sigma$, and $p_{\tilde{\sigma}}\in \pi^{-1}([p_\sigma])$. 
	As explained in Appendix \ref{Appendix-good-partition}, the triangulation $\triangle$ can be taken so that 
	\begin{eqnarray}\label{prop-orbit-type}
		\qquad (G_{p_{\tilde{\sigma}}}) \subset (G_{p}),  ~\forall p \in \tilde{\sigma}; \qquad
		(G_{p_{\tilde{\sigma}}}) = (G_{p}),  ~\forall p \in \pi^{-1}({\rm Int}(\sigma)),
	\end{eqnarray}
	by the works of Verona\cite{verona1979triangulation} and Illman\cite{illman1983equivariant}, in which $G_p:=\{g\in G : g\cdot p=p\}$ is the isotropy group of $p$, and $(G_p)$ is the conjugate class of $G_p$ in $G$. 
	By Lemma \ref{Lem:partition for any G}, we can further assume that for each $\sigma\in\triangle$, there exists a fundamental domain $\Omega\subset\tilde{\sigma}$ with a bi-Lipschitz homeomorphism $\pi\llcorner\Omega : \Omega\to\sigma$ such that 
	\begin{equation}\label{prop-no-cut-point}
		G_{(\pi\llcorner\Omega)^{-1}([p_{\sigma}])}\subset G_p, ~\forall p\in\Omega.
	\end{equation}
	Let $U(\tilde{\sigma})$ be the star neighborhood of $\tilde{\sigma}$ given by $U(\tilde{\sigma}) := \cup_{\tilde{\sigma}\subset\tilde{\sigma}'}\tilde{\sigma}'=\pi^{-1}(U(\sigma))$ for $U(\sigma) = \cup_{\sigma\subset\sigma'}\sigma'$.
	Define $\mathring{U}(\tilde{\sigma}):=\pi^{-1}({\rm Int}(U(\sigma)))$.

	We first construct equivariant cutting maps as in \cite[Theorem 14.3]{marques2014min}. 
	Using Lemma \ref{Lem:slice}, we can show that \cite[Proposition 1.18]{almgren1962homotopy} holds in the $G$-invariant setting.
	Consequently, after fixing a finite set $\Lambda\subset {\bf I}_{n+1}^G(M;\mathbb{Z}_2)$, we associate to every $\tilde{\sigma}\in\widetilde{\triangle}$ a $G$-neighborhood $L(\tilde{\sigma})\subset U(\tilde{\sigma})$ of $\tilde{\sigma}$ (by the distance function to $\tilde{\sigma}$ in $M$), and construct the following maps just as Almgren did in \cite[Section 5]{almgren1962homotopy}:
	$$C_\Lambda: \widetilde{\triangle}\times\Lambda \rightarrow {\bf I}_{n+1}^G(M;\mathbb{Z}_2),$$
	which satisfies
	\begin{eqnarray}
		C_\Lambda(\tilde{\sigma},Q) &=& \Big( Q-\sum_{\tilde{\sigma}'\prec\tilde{\sigma}} C_\Lambda(\tilde{\sigma}',Q) \Big) \cap L(\tilde{\sigma}),
		\\
		{\bf M}\Big( \partial C_\Lambda(\tilde{\sigma},Q) &-& \partial\big(Q-\sum_{\tilde{\sigma}'\prec\tilde{\sigma}} C_\Lambda(\tilde{\sigma}',Q) \big )\cap L(\tilde{\sigma})  \Big) 
		\\
		&\leq & C_0\cdot (\#\Lambda)\cdot{\bf M}\big(Q-\sum_{\tilde{\sigma}'\prec\tilde{\sigma}} C_\Lambda(\tilde{\sigma}',Q)\big), \nonumber
		\\
		{\rm spt}(C_\Lambda(\tilde{\sigma},Q))&\subset & U(\tilde{\sigma}),\quad \forall \tilde{\sigma}\in\widetilde{\triangle},~Q\in\Lambda,
	\end{eqnarray}
	where $C_0>0$ only depends on $\widetilde{\triangle}$ and $m$. 
	Although the choice of $L(\tilde{\sigma})$ depends on the finite set $\Lambda$, there exists a uniform bound for all $L(\tilde{\sigma})$ (see the choice of $r$ in \cite[Definition 5.4]{almgren1962homotopy}). 
	Indeed, there exists a $G$-neighborhood $N(\tilde{\sigma})\subset \mathring{U}(\tilde{\sigma})$ of $\tilde{\sigma}$ such that ${\rm spt}(C_\Lambda(\tilde{\sigma},Q))\subset N(\tilde{\sigma})$ for any finite set $\Lambda\subset {\bf I}_{n+1}^G(M;\mathbb{Z}_2)$ and $Q\in\Lambda$. 
	Thus, we mainly focus on the property of deformation maps in $N(\tilde{\sigma})$. 
	
	Secondly, we construct the equivariant deformation maps in $N(\tilde{\sigma})\subset \mathring{U}(\tilde{\sigma})$ based on the constructions of Almgren-Pitts. 
	Let $I^m = [0,1]^m$, $m\geq 1$, be a unit cube. 
	Denote 
	$$E_0^m := \bigcup_{i=1}^m I^m\cap\{x: x_i=0\}, ~E_1^m := \bigcup_{i=1}^m I^m\cap\{x: x_i=1\},$$
	as the union of $I^m$'s faces containing $(0,\dots,0)$ or $(1,\dots,1)$. 
	For any $x\in \mathbb{R}^m$, $\gamma_x^m : \mathbb{R}\rightarrow \mathbb{R}^m$ is the diagonal line through $x$ given by $\gamma_x^m(r) := (x_1-r,\dots,x_m - r)$. 
	Let $\mathcal{B}^m_\epsilon(x) := \{y\in I^m : |\gamma_y^m(r) - x|<\epsilon \mbox{ for some } r\in \mathbb{R} \}$ be a neighborhood of $\gamma_x^m(r)$ in $I^m$. 
	For any $x\in I$, let $x^{(j)}\in\{0,1\}$, $j\in\mathbb{Z}^+$, be the nonterminating binary expansion of $x$, i.e. $x = \sum_{j=1}^\infty x^{(j)}2^{-j}$, $\limsup_{j\to\infty}x^{(j)} = 1$ if $x>0$. 
	Then for any $x=(x_1,\dots,x_m)\in I^m$ with nonterminating binary expansions $\{x_i^{(j)}\}_{j\in\mathbb{Z}^+}$, $i=1,\dots,m$, we can rearrange the binary expansions as $x_m^{(1)}, x_{m-1}^{(1)},\dots, x^{(1)}_1,  x_m^{(2)}, x_{m-1}^{(2)},\dots, x^{(2)}_1,\dots$, which corresponds to a real number $\xi(x) = \sum_{j=1}^\infty\sum_{i=1}^m x_{m-i+1}^{(j)} 2^{m-mj-i}\in I$. 
	Conversely, for any $t\in I$, 
	we can rearrange the (terminating and nonterminating) binary expansion of $1-t$ to get two points $a^m(t),b^m(t)\in I^m$ (possibly equal) with $\xi(a^m(t))= \xi(b^m(t)) = 1-t$. 
	Define $A^m_t := \Clos(\{x\in I^m : \xi(x)\leq 1-t \} )$ for any $t\in [0,1]$. 
	Then, by \cite[4.4]{pitts2014existence}, there exists a map 
	$$ f^m_t : I^m \rightarrow I^m,$$
	satisfying: 
	\begin{itemize}
		\item[(i)] $f^m_t(x) := \gamma^m_x(r_{x,t})$, where $r_{x,t}$ is the smallest nonnegative number such that $\gamma^m_x(r_{x,t})\in A^m_t\cup E^m_0$;
		\item[(ii)] $f^m_1(x)\in E^m_0$ for all $x$ in $I^m$;
		\item[(iii)] ${\rm Lip}(f^m_t) \leq C(m)$;
		\item[(iv)] for any $\epsilon>0$, there exists $\delta>0$ so that 
		$$f^m_t (x)=f^m_s (x) , \quad \forall x \in I^m\setminus \mathcal{B}^m_\epsilon(t),$$
		provided $|t-s| < \delta$, where $\mathcal{B}^m_\epsilon(t) := \mathcal{B}^m_\epsilon(a^m(t))\cup \mathcal{B}^m_\epsilon(b^m(t))$.
	\end{itemize}
	Through an inductive construction on faces of $I^m$ in $E^m_0$, Pitts obtained a map $f:[0,m]\times I^m\rightarrow I^m$ satisfying ${\rm Lip}(f)\leq \sqrt{m!}$ and $f(t,x)\in E$ for some $(m - \lfloor t \rfloor)$-face $E$ of $I^m$ containing $(0,\dots,0)$ (see \cite[4.4(5)]{pitts2014existence}). 
	We now generalize this construction into $G$-equivariant settings and get an ${\bf M}$-continuous deforming map on each $\mathring{U}(\tilde{\sigma})$. 
	
	For any cell $\sigma\in\triangle$ with $\dim(\sigma ) = m, ~1\leq m\leq l = {\rm Cohom}(G)$, let $G\cdot p_{\tilde{\sigma}} = \pi^{-1}([p_\sigma])$ be the center orbit of $\tilde{\sigma} = \pi^{-1}(\sigma)$. 
	Note $M_{(G_{p_{\tilde{\sigma}}})}:=\{q\in M: (G_q) = (G_{p_{\tilde{\sigma}}})\}$ is the set of points with same orbit type $(G_{p_{\tilde{\sigma}}})$, which is a union of some disjoint Riemannian submanifolds of $M$ (see Appendix \ref{Appendix-good-partition}). 
	By (\ref{prop-orbit-type}), we have $\tilde{\sigma}\subset \Clos(M_{(G_{p_{\tilde{\sigma}}})})$. 
	Using (\ref{prop-no-cut-point}), 
	there exists a subset $\Omega\subset \tilde{\sigma}$ so that $\pi\llcorner_\Omega :\Omega\rightarrow\sigma$
	is a bi-Lipschitz homeomorphism and $G_{(\pi\llcorner\Omega)^{-1}([p_{\sigma}])}\subset G_p$, for all $p\in\Omega$. 
	For simplicity, we take $p_{\tilde{\sigma}} = (\pi\llcorner\Omega)^{-1}([p_{\sigma}])\in\Omega$. 
	Additionally, $X_{(G_{p_{\tilde{\sigma}}})}:=\pi(M_{(G_{p_{\tilde{\sigma}}})})$ has the induced Riemannian metric $g_{_{X_{(G_{p_{\tilde{\sigma}}})}}}$ making $\pi\llcorner M_{(G_{p_{\tilde{\sigma}}})}$ a Riemannian submersion, 
	which helps us to use the co-area formula. 
	
	Given any face $\tau$ of $\sigma$ with $\tau\neq\sigma$, let $[p_\tau]$ be the center point of $\tau$, $p_{\tilde{\tau}} := (\pi \llcorner_\Omega)^{-1}([p_\tau])$, and 
	$$E_0^\tau := \cup \{\tau'\in \triangle : \tau\subset\tau'\subset\partial\sigma \} , ~~E_1^\tau := \cup\{\tau'\in \triangle : \tau \not\subset \tau'\subset\partial\sigma \} . $$
	We can get a bi-Lipschitz homeomorphism $h_\tau: I^m\rightarrow \sigma$ satisfying $h_\tau(0) = [p_\tau]$, and $h_\tau(E^m_0) = E^{\tau}_0$, $h_\tau(E^m_1)=E^\tau_1$. 
	Denote $A_t^\tau := h_\tau(A^m_t)$, $a^\tau(t):=h_\tau(a^m(t))$, $b^\tau(t) := h_\tau(b^m(t))$, $\gamma_x^\tau := h_\tau(\gamma^m_x))$, $\mathcal{B}_\epsilon^\tau(t) := h_\tau(\mathcal{B}^m_\epsilon(t))$. 
	Thus, the map 
	$$F^{m,\tau} :[0,1]\times\sigma\rightarrow\sigma$$
	given by $F^{m,\tau}(t,\cdot) := h_\tau\circ f^m_t\circ h_\tau^{-1}$ satisfies (i)-(iv) with $\tau$, $A_t^\tau$, $a^\tau(t)$, $b^\tau(t)$, $\gamma_x^\tau$, $\mathcal{B}_\epsilon^\tau(t)$, $E_i^\tau$ in place of $I^m$, $A^m_t$, $a^m(t)$, $b^m(t)$, $\gamma^m_x$, $\mathcal{B}^m_\epsilon(t)$, $E_i^m$. 
	Define 
	$$\sigma_0 := \sigma\setminus E^\tau_1,\quad \tilde{\sigma}_0:=\pi^{-1}(\sigma_0).$$
	We then get a $G$-equivariant map 
	$$\widetilde{F}^{m,\tau} : [0,1]\times \tilde{\sigma}_0\rightarrow \tilde{\sigma}_0$$
	defined as $\widetilde{F}^{m,\tau}(t,q):=g\cdot (\pi\llcorner_{\Omega})^{-1}\big( F^{m,\tau}(t, \pi(p))\big)$ if $q=g\cdot p\in \tilde{\sigma}_0$ for some $p\in \Omega$, $g\in G$. 
	We claim this map $\widetilde{F}^{m,\tau}$ is well defined. 
	Indeed, it's clear that $\widetilde{F}^{m,\tau}={\rm id}$ on $\pi^{-1}(E^\tau_0)$. 
	Suppose $p\in \Omega\setminus\pi^{-1}(E^\tau_0)$, and there is another $g'\in G$ with $g'\cdot p=g\cdot p$, i.e. $g^{-1}g'\in G_p$. 
	By (\ref{prop-orbit-type}) and (\ref{prop-no-cut-point}), we conclude that $G_p=G_{p_{\tilde{\sigma}}} = G_y$ for all $y\in \Omega\setminus\pi^{-1}(\partial\sigma)$. 
	Note $p' = (\pi\llcorner_{\Omega})^{-1}\big( F^m(t, \pi(p))\big)\in \Omega$. 
	Hence, by (\ref{prop-no-cut-point}), we have $g^{-1}g'\in G_{p} = G_{p_{\tilde{\sigma}}} \subset G_{p'}$ implying $g\cdot p'=g'\cdot p'$. 
	
	\begin{claim}
		For any $T\in {\bf I}_n^G(\tilde{\sigma}_0;{\bf M};\mathbb{Z}_2)$, the map 
		$$D^{m,\tau}:[0,1]\times {\bf I}_n^G(\tilde{\sigma}_0 ;{\bf M};\mathbb{Z}_2) \rightarrow {\bf I}_n^G(\tilde{\sigma}_0 ;{\bf M};\mathbb{Z}_2)$$
		 given by $D^{m,\tau}(t,T):=\widetilde{F}^{m,\tau}(t,\cdot)_{\#}T$ is continuous with respect to the ${\bf M}$-norm.
	\end{claim}
	\begin{proof}
		Denote $\widetilde{F}_t = \widetilde{F}^{m,\tau}(t,\cdot)$, $\widetilde{E}_i=\pi^{-1}(E^\tau_i)$, $i=0,1$, for simplicity. 
		Let $\widetilde{\mathcal{B}}_\epsilon(t) := \pi^{-1}(\mathcal{B}_\epsilon^\tau(t) )$ and $s,t\in [0,1]$. 
		For any $\epsilon>0$, let $\delta>0$ be given by the property (iv) of $f^m_t$. 
		Then we have $\widetilde{F}_t = \widetilde{F}_s$ on $\tilde{\sigma}_0 \setminus \widetilde{\mathcal{B}}_\epsilon(t) $ whenever $|s-t|<\delta$. 
		This implies 
		\begin{eqnarray*}
			{\bf M}(\widetilde{F}_{t~\#}T - \widetilde{F}_{s~\#}T) &= & {\bf M}(\widetilde{F}_{t~\#} T_\epsilon - \widetilde{F}_{s~\#}T_\epsilon)
			\\
			&\leq & {\bf M}(\widetilde{F}_{t~\#} T_\epsilon) + {\bf M}(\widetilde{F}_{s~\#} T_\epsilon),
		\end{eqnarray*}
		where $T_\epsilon := T\llcorner \widetilde{\mathcal{B}}_\epsilon(t)$, $|s-t|<\delta$. 
		On the other hand, since $f^m_s(B^m_\epsilon(t))\subset B^m_\epsilon(t)$, 
		\begin{eqnarray*}
			{\bf M}(\widetilde{F}_{s~\#} T_\epsilon) &=& {\bf M}\big((\widetilde{F}_{s~\#} T_\epsilon)\llcorner(\widetilde{\mathcal{B}}_\epsilon(t)\setminus \widetilde{E}_0 ) \big) + 
			{\bf M}\big((\widetilde{F}_{s~\#} T_\epsilon)\llcorner(\widetilde{\mathcal{B}}_\epsilon(t)\cap \widetilde{E}_0 ) \big)
			\\
			&\leq &
			{\bf M}\big((\widetilde{F}_{s~\#} T_\epsilon)\llcorner(\widetilde{\mathcal{B}}_\epsilon(t)\setminus \widetilde{E}_0 ) \big) + 
			\mathcal{H}^n(\widetilde{\mathcal{B}}_\epsilon(t)\cap \widetilde{E}_0).
		\end{eqnarray*}
		Note $\bigcap_{\epsilon>0}\mathcal{B}^m_\epsilon(t) = (\gamma^m_{a^m(t)}\cup\gamma^m_{b^m(t)})\cap I^m$ are two lines in $I^m$, and $\gamma^m_x \cap E^m_0$ contains only one point for $x\in I^m$. 
		Hence, $\bigcap_{\epsilon>0}(\widetilde{\mathcal{B}}_\epsilon(t)\cap \widetilde{E}_0)$ contains at most two orbits. 
		Since ${\rm Cohom}(G)\geq 3$, we have $\dim(G\cdot p)\leq n-2$ for all $p\in M$, and thus the second term $\mathcal{H}^n(\widetilde{\mathcal{B}}_\epsilon(t)\cap \widetilde{E}_0)$ will vanish with $s\to t$ and $\epsilon\to 0$. 
		
		Denote $\Sigma_\epsilon\subset \widetilde{\mathcal{B}}_\epsilon(t)$ to be the support of $T_\epsilon$. 
		Then we have $\pi(\widetilde{F}_s(\Sigma_\epsilon)) = F^{m,\tau}_s(\pi(\Sigma_\epsilon))$. 
		Since $\widetilde{\mathcal{B}}_\epsilon(t)\setminus \widetilde{E}_0 \subset M_{(G_{p_{\tilde{\sigma}}})}$, we can apply the co-area formula to see 
		\begin{eqnarray*}
			{\bf M}\big((\widetilde{F}_{s~\#} T_\epsilon)\llcorner(\widetilde{\mathcal{B}}_\epsilon(t)\setminus \widetilde{E}_0 ) \big) &=& \mathcal{H}^n\big(\widetilde{F}_s(\Sigma_\epsilon) \cap (\widetilde{\mathcal{B}}_\epsilon(t)\setminus \pi^{-1}(\partial\sigma)  )\big)
			\\
			&= & \int_{{\tiny \pi(\widetilde{F}_s(\Sigma_\epsilon) ) \cap (\mathcal{B}_\epsilon^\tau(t)\setminus (\partial\sigma))}} \mathcal{H}^{n+1-k}(\pi^{-1}( [q])) ~d\mathcal{H}^{k-1}([q])
			\\
			&\leq & c_1 \mathcal{H}^{k-1}( {F}^{m,\tau}_s(\pi(\Sigma_\epsilon))\cap  {\rm Int}(\sigma))
			\\
			&\leq & c_1 c_2^{k-1}\mathcal{H}^{k-1}( \pi(\Sigma_\epsilon) \cap {\rm Int}(\sigma)),
		\end{eqnarray*}
		where $k={\rm codim}(G\cdot p_{\tilde{\sigma}}) \geq {\rm Cohom}(G) \geq 3$, $c_1 = \sup_{q\in M_{(G_{p_{\tilde{\sigma}}})}}\{\mathcal{H}^{n+1-k}(G\cdot q) \} < \infty$, $c_2=c_2(\sigma,\tau)$ is the uniform Lipschitz constant of ${F}^{m,\tau}_s$. 
		The proof is finished by observing $\bigcap_{\epsilon>0} \mathcal{B}_\epsilon^\tau(t)  \subset  (\gamma_{a^\tau(t)}^\tau\cup\gamma_{b^\tau(t)}^\tau)$ has dimension $1$. 
	\end{proof}
	
	Note $\dim(G\cdot p)\leq n-2<n$ for all $p\in M$. 
	Through an inductive construction on faces of $\sigma$ in $E^\tau_0$ as in \cite[4.4 (4),(5)]{pitts2014existence}, we can obtain a $G$-equivariant map $\widetilde{F}^{\sigma,\tau} :[0,m]\times \tilde{\sigma}_0 \rightarrow \tilde{\sigma}_0$, and an ${\bf M}$-continuous map
	$$D^{\sigma,\tau} :[0,1]\times {\bf I}_n^G(\tilde{\sigma}_0 ;{\bf M};\mathbb{Z}_2) \rightarrow {\bf I}_n^G(\tilde{\sigma}_0 ;{\bf M};\mathbb{Z}_2)$$
	given by $D^{\sigma,\tau}(t,T)=\widetilde{F}^{\sigma,\tau}(\frac{t}{m} ,\cdot)_{\#}T$ satisfying 
	$$D^{\sigma,\tau}(0,T)=T, \quad D^{\sigma,\tau}(1,T)=0, $$
	for all $T\in {\bf I}_n^G(\tilde{\sigma}_0 ;{\bf M};\mathbb{Z}_2) $. 
	
	Additionally, 
	there exists a constant $C(\sigma,\tau)>1$ so that ${\bf M}(D^{\sigma,\tau}(t,T))\leq C(\sigma,\tau) {\bf M}(T)$ for $T\in {\bf I}_n^G(\tilde{\sigma}_0\cap N(\tilde{\tau}) ;{\bf M};\mathbb{Z}_2) $. 
	Indeed, 
	for any $v\in T_qG\cdot q$ with $|v|=1$, there exists $g(s)\subset G$ with $g(0) = e$ and $\frac{d}{ds}\big|_{s=0} g(s)\cdot q = v$. 
	Since $\widetilde{F}^{m,\tau}_t$ is equivariant and $(G_{q})\subset (G_{\widetilde{F}^{m,\tau}_t(q)})$, we have $d\widetilde{F}^{m,\tau}_t(v) = \frac{d}{ds}\big|_{s=0} g(s)\cdot \widetilde{F}^{m,\tau}_t(q)$ always exists and is continuous in $q,t$. 
	The continuity of $(d\widetilde{F}^{m,\tau}_t\llcorner G\cdot q)$ gives a uniform upper bound of its norm in a compact subset $K\subset\subset (\tilde{\sigma}_0 \setminus E^\tau_0)$. 
	As for $q$ sufficiently closed to the boundary $E^\tau_0 \cap \Clos(N(\tilde{\tau}))$, the Lipschitz curve $\widetilde{F}^{m,\tau}_t(q)$ is almost a line $(1-\zeta(t))q + \zeta(t)\widetilde{F}^{m,\tau}_1(q)$ by the equivariant embedding $\rho : M\to \R^L$. 
	By regarding $g(s)$ as a curve of matrices in $O(L)$, we have 
	$A(q,v) = \frac{d}{ds}\big|_{s=0} g(s)$ is a matrix with $|A(q,v)\cdot q| = |v| = 1$ and 
	$\frac{d}{ds}\big|_{s=0} g(s)\cdot \widetilde{F}^{m,\tau}_t(q) = A(q,v)\cdot \widetilde{F}^{m,\tau}_t(q) \approx (1-\zeta(t))v + \zeta(t) A(q,v) \cdot \widetilde{F}^{m,\tau}_1(q)$. 
	If $v\in T_qG_{\widetilde{F}^{m,\tau}_1(q)} \cdot q$, then we can choose $g(s)\subset G_{\widetilde{F}^{m,\tau}_1(q)}$, and thus $|\frac{d}{ds}\big|_{s=0} g(s)\cdot \widetilde{F}^{m,\tau}_t(q)| \approx (1-\zeta(t))\leq 2$ is uniformly bounded. 
	If $v\in T_qG\cdot q\cap (T_qG_{\widetilde{F}^{m,\tau}_1(q)} \cdot q)^\perp$, then the maxtix $A(q,v) = \frac{d}{ds}\big|_{s=0} g(s)$ is continuous up to the boundary $E^{\tau}_0$ (a neighborhood of $\widetilde{F}^{m,\tau}_1(q)$ in $E^{\tau}_0$), which also implies $|\frac{d}{ds}\big|_{s=0} g(s)\cdot \widetilde{F}^{m,\tau}_t(q)| \leq C$ by the uniform bound of $A(q,v)$. 
	Therefore, $\widetilde{F}^{m,\tau}_t$ 
	is differentiable in the tangent direction of orbits, and the total derivative in the tangent direction is bounded in the compact neighborhood $\tilde{\sigma}_0\cap \Clos(N(\tilde{\tau}))$ of $G\cdot p_{\tilde{\tau}}$, which implies $\mH^{n+1-k}(G\cdot \widetilde{F}^{m,\tau}_t(q))\leq C'(\tau,\sigma)\mH^{n+1-k}(G\cdot q)$, for all $q\in \tilde{\sigma}_0\cap \Clos(N(\tilde{\tau}))$, $t\in [0,1]$, by the area formula. 
	Now for any $G$-invariant $n$-rectifiable set $\Sigma\subset \tilde{\sigma}_0\cap \Clos(N(\tilde{\tau}))$, we have $\mH^n(\widetilde{F}^{m,\tau}_t(\Sigma\cap \pi^{-1}(\partial \sigma ))) = \mH^n(\Sigma\cap \pi^{-1}(\partial \sigma )) $ since $\widetilde{F}^{m,\tau}_t\llcorner \pi^{-1}(\partial\sigma)\cap\tilde{\sigma}_0 = id$. 
	By applying the co-area formula (on each stratum), we have 
	\begin{eqnarray*}
		\mH^n(\widetilde{F}^{m,\tau}_t(\Sigma\setminus \pi^{-1}(\partial \sigma)) &\leq & \int_{F^{m,\tau}_t(\pi(\Sigma)\setminus\partial\sigma )} \mH^{n+1-k}(\pi^{-1}([q])) d\mH^{k-1}([q])
		\\
		&\leq & \int_{\pi(\Sigma)\setminus\partial\sigma} c_2^{k-1}C'(\tau,\sigma)\mH^{n+1-k}(\pi^{-1}([q]))  d\mH^{k-1}([q])
		\\
		&= & C(\tau,\sigma) \mH^n(\Sigma\setminus \pi^{-1}(\partial\sigma) ),
	\end{eqnarray*} 
	where $k={\rm codim}(G\cdot p_{\tilde{\sigma}})$, $c_2$ is the uniform Lipschitz constant of ${F}^{m,\tau}_t$. 
	Therefore, we have $\mH^n(\widetilde{F}^{m,\tau}_t(\Sigma))\leq C(\tau,\sigma)\mH^n(\Sigma)$, which implies ${\bf M}(D^{\sigma,\tau}(t,T))\leq C(\sigma,\tau) {\bf M}(T)$ for $T\in {\bf I}_n^G(\tilde{\sigma}_0\cap N(\tilde{\tau}) ;{\bf M};\mathbb{Z}_2) $.

	In above constructions, we supposed $\tau$ is a face of $\sigma$ and $\tau\neq\sigma$. 
	As for the case that $\tau=\sigma$, let $Sd(\sigma)$ be the barycentric subdivision of $\sigma$. 
	Then for any $m$-dimensional cell $\sigma'\in Sd(\sigma)$, $[p_\tau]=[p_\sigma]$ is a vertex of $\sigma'$. 
	Since $f^m_t={\rm id}$ on $E^m_0$, we can define $F^{m,\sigma}:[0,1]\times\sigma\to\sigma$ by defining such function on every $\sigma'$ as above. 
	Let $\sigma_0 := \sigma\setminus \partial\sigma$ and $\tilde{\sigma}_0:=\pi^{-1}(\sigma_0)$. 
	Then the $G$-equivariant map 
	$\widetilde{F}^{m,\sigma} : [0,1]\times \tilde{\sigma}_0\rightarrow \tilde{\sigma}_0$ can be defined in the same way with the same properties. 
	Again, an inductive construction gives $\widetilde{F}^{\sigma,\sigma}$ and an ${\bf M}$-continuous map $D^{\sigma,\sigma}$. 
	
	Furthermore, by the definition of $\sigma_0$, the map $F^{m,\tau}$ can be defined on ${\rm Int}(U(\tau))$. 
	Hence, for any $\sigma\in\triangle$, there exists a $G$-equivariant map $\widetilde{F}(\tilde{\sigma}) :[0,l]\times  \mathring{U}(\tilde{\sigma}) \rightarrow  \mathring{U}(\tilde{\sigma})$, $l={\rm Cohom}(G) = \dim(M/G)$, and an ${\bf M}$-continuous map 
	\begin{eqnarray*}
		&&D(\tilde{\sigma}) : I\times {\bf I}_n^G( \mathring{U}(\tilde{\sigma}) ;{\bf M};\mathbb{Z}_2) \rightarrow {\bf I}_n^G( \mathring{U}(\tilde{\sigma}) ;{\bf M};\mathbb{Z}_2),
	\end{eqnarray*}
	given by $D(\tilde{\sigma})(t, T) := \widetilde{F}(\tilde{\sigma})(\frac{t}{l},\cdot)_\# T$, 
	which satisfies $D(\tilde{\sigma})(0,T)=T$, $D(\tilde{\sigma})(1,T)=0$, for all $T\in {\bf I}_n^G(\mathring{U}(\tilde{\sigma});{\bf M};\mathbb{Z}_2)$, and ${\bf M}(D(\tilde{\sigma})(t,T))\leq C(M,G) {\bf M}(T)$ if $\spt(T)\subset N(\tilde{\sigma})$. 
	We mention that the restriction on $\spt(T)\subset N(\tilde{\sigma})$ can be achieved by the cutting procedure $C_\Lambda(\tilde{\sigma}, \cdot )$. 
	
	Then for any $k$-cell $\alpha\in I(m,0)_k$, we consider the continuous function $h_\alpha: I^k\rightarrow \mathcal{Z}_n^G(M;{\bf M};\mathbb{Z}_2)$ defined by:
	\begin{equation*}
		h_\alpha(0) := \phi(\alpha),\quad {\rm for} ~k=0,
	\end{equation*}
	and for $k\geq 1$:
	\begin{equation*}
		\begin{split}
			&h_\alpha(x_1,\dots,x_k) := \sum_{\gamma\in\Gamma_\alpha} {\rm sign}(\gamma)\cdot 
			\\
			&\sum_{\tilde{s}_1,\dots,\tilde{s}_k\in\widetilde{\triangle}} D(\tilde{s}_1,x_1)\circ\cdots\circ D(\tilde{s}_k,x_k) \circ \partial \circ C_{\Lambda(\gamma_k)}(\tilde{s}_k)\circ\cdots\circ C_{\Lambda(\gamma_1)}(\tilde{s}_1) (Q(\gamma_1)),
		\end{split}
	\end{equation*}
	where $\gamma = (\gamma_1,\dots,\gamma_k ),~{\rm sign}(\gamma),~\Gamma_\alpha,~\Lambda(\gamma_i )$ are all defined as in \cite[Theorem 14.3]{marques2014min}.
	The rest of the proof is just the same as in \cite[Theorem 14.3]{marques2014min}.
\end{proof}

\begin{remark}\label{Rem:Almgren isomorphism}
	Combining the arguments in \cite{almgren1962homotopy} with Lemma \ref{Lem:isoperimetric} and Theorem \ref{Thm:interpolation}, we can show the Almgren's Isomorphism \cite[Theorem 7.1]{almgren1962homotopy} and Almgren's Homotopies \cite[Theorem 8.2]{almgren1962homotopy} remain true for $\mathcal Z_{n}^G(M;\mathbb{Z}_2)$. 
	Moreover, 
	we also have the isomorphism theorem for $\mathcal Z_{n}^G(M;{\bf M};\mathbb{Z}_2)$.
\end{remark}

We can also adapt the proof of \cite[Corollary 3.9]{marques2017existence} straightforwardly into the $G$-invariant setting by using Theorem \ref{Thm:discritization}, \ref{Thm:interpolation}, and Remark \ref{Rem:Almgren isomorphism} in place of \cite[Theorem 3.6, 3.7]{marques2017existence} and \cite[Theorem 8.2]{almgren1962homotopy}: 
 
\begin{corollary}\label{Cor:discrete continuous homotopy}
	Let $S=\{\phi_i\}_{i\in \mathbb{N}}$ and $S'=\{\phi'_i\}_{i\in \mathbb{N}}$ be two $(X,{\bf M})$-homotopy sequences  of mappings into  $\mathcal{Z}_n^G(M;{\bf M};\mathbb{Z}_2)$ such that  $S,S'$ are $G$-homotopic.
	\begin{itemize}
 		\item[(i)] The Almgren $G$-extensions of $\phi_i$, $\phi_i'$:
 		$$\Phi_i, \Phi_i':X\rightarrow \mathcal{Z}_n^G(M;{\bf M};\mathbb{Z}_2),$$are $G$-homotopic to each other in the flat topology for $i$ large enough.
 		\item[(ii)] If $S$ is given by Theorem \ref{Thm:discritization} (i) applied to $\Phi:X\rightarrow \mathcal{Z}_n^G(M;\mathbb{Z}_2)$, an $\mathcal{F}$-continuous map with no concentration of mass on orbits, then $\Phi_i$ is $G$-homotopic to $\Phi$ in the flat topology for every sufficiently large $i$.
	 	Moreover, $$\limsup_{i\rightarrow \infty} \sup_{x\in X}\{ {\bf M}(\Phi_i(x))\}={\bf L}(\{\Phi_i\}_{i\in\N})={\bf L}(S) \leq  \sup_{x\in X}\{ {\bf M}(\Phi(x))\}.$$
	\end{itemize}
\end{corollary}

\subsection{Pull-tight}
Through a pull-tight argument (see \cite[Theorem 4.3]{pitts2014existence}, \cite[Section 15]{marques2014min}), a critical sequence $S$ can be constructed in any $\Pi \in [X,\mathcal{Z}_n(M;{\bf M};\mathbb{Z}_2)]^{\#}$ so that every varifold in the critical set ${\bf C}(S)$ is stationary in $M$. 
In our equivariant case, the pull-tight procedure can be made similarly to \cite[Section 15]{marques2014min}. 
Besides, we apply an argument involving continuous homotopy settings (see \cite{marques2015morse}) to avoid a technical obstacle (\cite[Page 766]{marques2014min}) and make the article more concise. 

\begin{proposition}[Pull-tight]\label{Prop:pulltight}
	Let ${\bf \Pi}\in \big[ X, \Z_{n}^G(M;\mF;\mZ_2)\big]$ be a continuous $G$-homotopy class. 
	For any min-max sequence $\{\Phi_i^*\}_{i\in\N}$ in ${\bf \Pi}$, there exists a {\em tight} min-max sequence $\{\Phi_i\}_{i\in\N}$ of ${\bf \Pi}$ such that
	\begin{itemize}
		\item[$\bullet$] ${\bf C}(\{\Phi_i\}_{i\in\N}) \subset {\bf C}(\{\Phi_i^*\}_{i\in\N})$;
		\item[$\bullet$] every $G$-varifold $V\in {\bf C}(\{\Phi_i\}_{i\in\N})$ is stationary in $M$.
	\end{itemize}
	Moreover, there exists a sequence $\{k_i\}_{i\in\N}\subset \N $ with $k_i\to\infty$, and a sequence $S=\{\varphi_i\}_{i\in\N}$ of discrete mappings 
	$$\varphi_i: X( k_i)_{0}\rightarrow \Z_{n}^G(M ;\mZ_2) $$
	satisfying 
	\begin{itemize}
		\item[(i)] $\mf(\varphi_i)=\delta_i\to 0$, as $i\to\infty$; 
		\item[(ii)] the Almgren $G$-extension $\Phi_i'$ of $\varphi_i$ is $G$-homotopic to $\Phi_i$ in the flat metric; 
		\item[(iii)] ${\bf L}(S)={\bf L}(\{\Phi_i\}_{i\in\N})={\bf L}({\bf \Pi})$;
		\item[(iv)] ${\bf C}(S)={\bf C}(\{\Phi_i\}_{i\in\N})$.
	\end{itemize}
\end{proposition}

\begin{remark}
	Note $S=\{\varphi_i\}_{i\in\N}$ in the above proposition may not be an $(X,{\bf M})$-homotopy sequence of mappings into $\Z_{n}^G(M;\M;\mZ_2)$. 
	However, we can still define 
	$${\bf L}(S):=\limsup_{i\rightarrow\infty} \max_{x\in \mathrm{dmn}(\varphi_i)} {\bf M}(\varphi_i(x)),$$ 
	and ${\bf C}(S)$ as in Definition \ref{Def:critical set} (\cite[Section 2.5]{marques2017existence}).
\end{remark}

\begin{proof}
	Denote $A := \{V\in\mathcal{V}^G_n(M): \|V\|(M)\leq 2{\bf L}({\bf \Pi}) \}$, and denote $A_0$ to be the set of all stationary varifolds in $A$. 
	For any vector field $Y\in \mathfrak{X}(M)$, let $Y_G\in \mathfrak{X}^G(M)$ be defined as in (\ref{Eq:average}). 
	Then by \cite[Lemma 2.2]{liu2021existence}, we have 
	$ \delta V(Y)=\delta V(Y_G)$
	for any $G$-varifold $V\in\V^G_n(M)$. 
	Combining the compactness theorem for $G$-varifolds (Remark \ref{Rem:compactness for G-varifold}) with the arguments in \cite[Page 765]{marques2014min}, we can construct two maps 
	$$Y:A\to \mathfrak{X}^G(M),\quad h:A\to [0,1],$$
	continuous in the $\mF$-metric on $A$ and the $C^1$-topology on $\mathfrak{X}^G(M)$ such that 
	\begin{itemize}
		\item[$\bullet$] $Y(V)=0$, $h(V)=0$, if $V\in A_0$;
		\item[$\bullet$] $\delta V(Y(V))<0$, $h(V)>0$, if $V\in A\setminus A_0$;
		\item[$\bullet$] $\|(f^{Y(V)}_{t})_{\#}V\|(M) < \|(f^{Y(V)}_{s})_{\#}V\|(M)$, if $0\leq s<t\leq h(V)$,
	\end{itemize}
	where $\{f^{Y(V)}_{t}\}_{t\in[0,1]}$ are $G$-equivariant diffeomorphisms generated by $Y(V)$. 
	Now we define the following $G$-homotopy map continuous in the $\mF$-metric:
	\begin{eqnarray}\label{Eq-pulltightmap}
		H:~I\times \big(\mathcal{Z}_n^G(M;{\bf F};\mathbb{Z}_2)&\cap & \{T~:~ {\bf M}(T)\leq 2{\bf L}({\bf \Pi}) \} \big)\nonumber
		\\
		&\rightarrow & \mathcal{Z}_n^G(M;{\bf F};\mathbb{Z}_2)\cap  \{T~:~ {\bf M}(T)\leq 2{\bf L}({\bf \Pi}) \},
		\\
		H(t,T) &:=& \big( f^{Y(|T|)}_{h(|T|)t} \big)_\# T, \nonumber
	\end{eqnarray}
	which satisfies:
	\begin{itemize}
		\item[$\bullet$] $H(0,T)=T$;
		\item[$\bullet$] if $|T|$ is stationary in $M$, then $H(t,T)=T$ for all $t\in[0,1]$;
		\item[$\bullet$] if $|T|$ is not stationary in $M$, then ${\bf M}(H(1,T))<\M(T)$. 
	\end{itemize}

	Let $\Phi_i := H(1,\Phi_i^*)$. 
	It's clear that $\{\Phi_i\}_{i\in\N}\subset {\bf \Pi}$ is also a continuous min-max sequence.
	Furthermore, every $G$-varifold in ${\bf C}(\{\Phi_i\}_{i\in\N})$ is stationary in $M$ and contained in ${\bf C}(\{\Phi_i^*\}_{i\in\N})$. 
	After applying Theorem \ref{Thm:discritization} to each $\Phi_i$, we get a sequence $S_i=\{\phi_i^j\}_{j=1}^\infty$ of maps 
	$\phi_i^j:X(k_i^j)\to \Z_{n}^G(M;\mZ_2)$ 
	as an $(X,{\bf M})$-homotopy sequence of mappings into $\Z_{n}^G(M;\M;\mZ_2)$ for every $i\in\N$. 
	Noting $X\to \M(\Phi_i(\cdot))$ is continuous for each $\Phi_i$, we can further combine Theorem \ref{Thm:discritization} (ii)(iii) with \cite[Lemma 4.1]{marques2014min} to show
	\begin{equation}\label{Eq: pull tight estimate}
		\lim_{j\to\infty} \sup\{\mF(\phi_i^j(x),\Phi_i(x)) : x\in X(k_i^j)_0   \}=0,
	\end{equation}
	for every $i\in\N$. 
	By Corollary \ref{Cor:discrete continuous homotopy} (ii), the Almgren $G$-extension $\Phi_i^j$ of $\phi_i^j$ is $G$-homotopic to $\Phi_i$ for $j$ large enough, which implies $\Phi_i^j\in {\bf \Pi}$. 
	Furthermore, by the definition of ${\bf L}({\bf \Pi})$ and Corollary \ref{Cor:discrete continuous homotopy} (ii), we also have 
	\begin{equation}\label{Eq: pull tight width}
		{\bf L}({\bf \Pi}) \leq {\bf L}(\{\Phi_i^j\}_{j=1}^\infty) ={\bf L}(S_i) \leq  \sup_{x\in X}\{ {\bf M}(\Phi_i(x))\}\to {\bf L}({\bf \Pi})~{\rm as}~i\to\infty.
	\end{equation}
	Therefore, after taking a subsquence $j(i)\to\infty$ as $i\to\infty$, we can define $S:=\{\varphi_i\}_{i=1}^\infty$ with $\varphi_i:=\phi_i^{j(i)}$ so that $\f(\varphi_i)\to 0$ and 
	\begin{itemize}
		\item[$\bullet$] the Almgren's $G$-extension $\Phi_i':=\Phi_i^{j(i)}$ of $\varphi_i$ is $G$-homotopic to $\Phi_i$;
		\item[$\bullet$] ${\bf L}({\bf \Pi})={\bf L}(\{\varphi_i\}_{i\in\N})$ by (\ref{Eq: pull tight width});
		\item[$\bullet$] $\sup\{\mF(\varphi_i(x),\Phi_i(x)) : x\in X(k_i^{j(i)})_0   \}\leq a_i\to 0$ as $i\to\infty$ by (\ref{Eq: pull tight estimate});
		\item[$\bullet$] $\sup\{\mF(\Phi_i(x),\Phi_i(y)) : x,y\in\alpha,\alpha\in X(k_i^{j(i)})   \}\leq a_i\to 0$ as $i\to\infty$ by the $\mF$-continuity of $\Phi_i$. 
	\end{itemize}
	Finally, ${\bf C}(S)={\bf C}(\{\Phi_i\}_{i\in\N})$ follows from the last two bullets. 
\end{proof}


\section{$G$-Almost Minimizing Varifolds}\label{G-a.m.v}
 
In this section, we introduce the notion of $G$-almost minimizing varifolds.

\begin{definition}\label{Def:a.m.deform}
	For any $\epsilon>0,~\delta>0$, open $G$-set $U\subset M$, and $T\in \mathcal Z_{n}^G(M;\mathbb{Z}_2)$, we call a finite sequence $\{T_i\}_{i=1}^q\subset \mathcal Z_{n}^G(M;\mathbb{Z}_2)$ a {\em $G$-invariant $(\epsilon,\delta)$-deformation} of $T$ in $U$ under the metric ${\bf v}$ if
	\begin{itemize}
		\item[$\bullet$] $T_0=T$ and ${\rm spt}(T-T_i)\subset U$ for all $i=1,\dots, q;$
		\item[$\bullet$] ${\bf v}(T_i-T_{i-1})\leq \delta$ for all $i=1,\dots, q;$
		\item[$\bullet$] ${\bf M}(T_i)\leq {\bf M}(T)+\delta$ for all $i=1,\dots, q;$
		\item[$\bullet$] ${\bf M}(T_q)< {\bf M}(T)-\epsilon$,
	\end{itemize}
	where ${\bf v}$ is any one of $\mathcal{F},~{\bf F},~{\bf M}$. 
	Moreover, we define
	$$ \mathfrak{a}^G_n(U;\epsilon,\delta; {\bf v} ) $$
	to be the set of all $G$-cycles $T\in \mathcal Z_{n}^G(M;\mathbb{Z}_2)$ that do not admit any $G$-invariant $(\epsilon,\delta)$-deformation under the metric ${\bf v}$. 
	For simplicity, the notation ${\bf v}$ is sometimes omitted, which means the flat semi-norm $\mathcal{F}$ is used. 
\end{definition}

\begin{definition}\label{Def:a.m.}
	A varifold $V\in \mathcal V_{n}^G(M)$ is said to be {\em $(G,\mathbb{Z}_2)$-almost minimizing} in an open $G$-set $U\subset M$, if for every $\epsilon>0$ there exists $\delta>0$ and
	$$T\in \mathfrak{a}^G_n(U;\epsilon,\delta )$$ with ${\bf F}(V,|T|)<\epsilon$. 
	Moreover, if $T$ can be chosen as a boundary of an element in ${\bf I}^G_{n+1}(M;\mathbb{Z}_2)$, then we say $V$ is {\it $(G,\mathbb{Z}_2)$-almost minimizing of boundary type}. 
\end{definition}

In Pitts's original definition (\cite[3.1]{pitts2014existence}), the comparison currents $T\in\mathfrak{a}_n(U;\epsilon,\delta )$ are allowed to have a boundary outside $U$, and the almost minimizing varifold is defined with the ${\bf F}_U$-metric. 
Nevertheless, we here use a stronger version of the definition, which was first observed by X. Zhou in \cite[Definition 6.3]{zhou2015min}.

\begin{lemma}\label{Lem:G.a.m is stable}
	If $V\in \mathcal V_{n}^G(M)$ is $(G,\mathbb{Z}_2)$-almost minimizing in an open $G$-set $U\subset M$, then $V$ is stable under $G$-equivariant variations in $U$. 
\end{lemma}
\begin{proof}
	The proof of \cite[Theorem 3.3]{pitts2014existence} would carry over since we only take $G$-equivariant variations into consideration. 
\end{proof}


\begin{definition}\label{Def:a.m. in annuli}
	A varifold $V \in \mathcal{V}^G_n(M)$ is {\it $(G,\mathbb{Z}_2)$-almost minimizing in annuli} if
	for each $p \in M$, there exists $r=r(G\cdot p)>0$ such that $V$ is $(G,\mathbb{Z}_2)$-almost minimizing in $\an (p,s,t)$ for all $\an (p,s,t)\in\ann_r(p) $.
\end{definition}

As we mentioned in the introduction, the discrete homotopy sequence of mappings is defined under the mass norm ${\bf M}$, while the almost minimizing varifolds are defined under the flat semi-norm ${\bf v}=\mathcal{F}$. 
To use Pitts' combinatorial argument (\cite[Theorem 4.10]{pitts2014existence}), we need the following important theorem. 

\begin{theorem}\label{Thm:equivalence a.m.v}
	Let $V\in\mathcal{V}^G_n(M)$ and $U\subset M$ be an open $G$-set. 
	The following statements satisfy $(a) \Rightarrow (b) \Rightarrow (c)\Rightarrow(d)$:
	\begin{itemize}
		\item[(a)] $V$ is $(G,\mathbb{Z}_2)$-almost minimizing in $U$. 
		\item[(b)] For any $\epsilon>0$, there exist $\delta>0$ and $T\in \mathfrak{a}^G_n(U;\epsilon,\delta; {\bf F} )$ with ${\bf F}(V,|T|)<\epsilon$. 
		\item[(c)] For any $\epsilon>0$, there exist $\delta>0$ and $T\in \mathfrak{a}^G_n(U;\epsilon,\delta; {\bf M} )$ with ${\bf F}(V,|T|)<\epsilon$. 
		\item[(d)] $V$ is $(G,\mathbb{Z}_2)$-almost minimizing in $W$, for any open $G$-subset $W\subset\subset U$ so that every orbit in ${\rm Clos}(W)$ is principal, i.e. $\Clos(W)\subset M^{reg}$. 
	\end{itemize}
\end{theorem}
\begin{proof}
	It's clear that $(a)\Rightarrow(b)\Rightarrow(c)$ by the fact $\F(T-S)\leq \mF(T,S)\leq 2\M(T-S)$. 
	By the assumption that ${\rm Clos}(W)\subset M^{reg}$, we can apply Lemma \ref{L:interpolation} with the arguments in \cite[Theorem 3.10]{pitts2014existence} to see $(c)\Rightarrow(d)$. 
\end{proof}

\begin{corollary}\label{Cor:equivalence a.m.v in annuli}
	Let $V\in\mathcal{V}^G_n(M)$. 
	Suppose every non-principal orbit is isolated. 
	Then $V$ is $(G,\mathbb{Z}_2)$-almost minimizing in annuli if and only if for any $p\in M$ there exists $r(G\cdot p)>0$ so that for any $0<s<t<r(G\cdot p)$ and $\epsilon>0$ there exist $\delta>0$ and $T\in \mathfrak{a}^G_n(\an(p,s,t);\epsilon,\delta; {\bf M} )$ with ${\bf F}(V,|T|)<\epsilon$. 
\end{corollary}
\begin{proof}
	The necessity part comes from Theorem \ref{Thm:equivalence a.m.v} (a)$\Rightarrow$(c). 
	As for the sufficiency part, one notices that if $G\cdot p$ is a principal orbit in $M$, then there naturally exists a $G$-neighborhood of $G\cdot p$ in $M^{reg}$ by the openness of $M^{reg}$. 
	Additionally, since every non-principal orbit is isolated, if $G\cdot p$ is non-principal then there exists a $G$-neighborhood $B_r^G(p)$ of $G\cdot p$ so that $B^G_r(p)\setminus G\cdot p\subset M^{reg}$. 
	Hence, there always exists $r_{G\cdot p}>0$ satisfying ${\rm Clos}(\an(p,s,t))\subset M^{reg}$ for all $0<s<t<r_{G\cdot p}$. 
	Therefore, the sufficiency comes from Theorem \ref{Thm:equivalence a.m.v} (c)$\Rightarrow$(d). 
\end{proof}

In the above corollary, we use the assumption that every non-principal orbit is isolated to guarantee that ${\rm Clos}(\an(p,s,t))\subset M^{reg}$ for all $p\in M$ and $0<s<t$ small enough. 
To weaken the constraints on non-principal orbits, let's consider the following notations and definitions. 
For simplicity, we denote 
\begin{equation}\label{Eq:Pp}
	P_p := \left\{ \begin{array}{ll}
 		p & \textrm{if $p\in M^{reg}$,}\\
 		M_p & \textrm{if $p\in M\setminus M^{reg}$,}
  	\end{array} \right.
\end{equation}
where $M_p$ is the connected component of $M\setminus M^{reg}$ containing $p$. 
Hence, for any $p\in M$, there exists $r_1=r_1(G\cdot P_p)>0$ so that 
\begin{equation}\label{Eq:regular annuli}
	{\rm Clos}(\an(P_p,s,t))\subset M^{reg},~\forall \an (P_p,s,t)\in\ann_{r_1}(P_p),
\end{equation}
by regarding $M\setminus M^{reg}$ as a whole. 
Since the closure of such annulus is contained in $M^{reg}$, we name any such $G$-annulus as a {\em regular annulus}.
In particular, it is easy to check that $G\cdot P_p \equiv G\cdot p$ when every non-principal orbit is isolated.

\begin{definition}\label{Def:a.m. in regular annuli}
	We say a $G$-varifold $V \in \mathcal{V}^G_n(M)$ is {\it $(G,\mathbb{Z}_2)$-almost minimizing in regular annuli} if
	for each $p \in M$, there exists $r=r(G\cdot P_p)\in (0,r_1(G\cdot P_p))$ such that $V$ is $(G,\mathbb{Z}_2)$-almost minimizing in $\an (P_p,s,t)$ for any $\an (P_p,s,t)\in\ann_r(P_p) $, where $r_1(G\cdot P_p)$ is given by (\ref{Eq:regular annuli}). 
\end{definition}

Combining (\ref{Eq:regular annuli}) with Theorem \ref{Thm:equivalence a.m.v}, it is easy to get the following corollary:
\begin{corollary}\label{Cor:equivalence a.m.v in regular annuli}
	Let $V\in\mathcal{V}^G_n(M)$. 
	Then $V$ is $(G,\mathbb{Z}_2)$-almost minimizing in regular annuli if and only if for any $p\in M$ there exists $r=r(G\cdot P_p)\in (0,r_1(G\cdot P_p))$ so that for any $0<s<t<r(G\cdot P_p)$ and $\epsilon>0$ there exist $\delta>0$ and $T\in \mathfrak{a}^G_n(\an(P_p,s,t);\epsilon,\delta; {\bf M} )$ with ${\bf F}(V,|T|)<\epsilon$. 
\end{corollary}

\subsection{Existence of $(G,\mZ_2)$-almost minimizing varifolds}\label{am.varifolds}
The existence of almost minimizing varifolds is achieved in Pitts's book \cite[Theorem 4.10]{pitts2014existence} through a combinatorial argument. 
This procedure can also be completed under $G$-invariant restrictions by lemmas in Section \ref{Sec-G-current} and Corollary \ref{Cor:equivalence a.m.v in regular annuli}. 
Specifically, we have the following theorem.

\begin{theorem}\label{Thm:pitts.min.max}
	Let $S=\{\varphi_i\}_{i\in\N}$ be a sequence of mappings 
	$$\varphi_i: X( k_i)_{0}\rightarrow \Z_{n}^G(M;\mZ_2),$$
	with $\lim_{i\to\infty}k_i=\infty$, $\lim_{i\to\infty}\mf(\varphi_i)= 0$, ${\bf L}(S)>0$, and every $G$-varifold $V\in{\bf C}(S)$ is stationary in $M$. 
	Suppose no element $V\in {\bf C}(S)$ is $(G,\mZ_2)$-almost minimizing in regular annuli, then there exists a sequence $S^*=\{\varphi_i^*\}_{i\in\N}$ of mappings
	$$\varphi_i^*: X( l_i)_{0}\rightarrow \Z_{n}^G(M ;\mZ_2), $$
	for some $l_i\to\infty$ as $i\to\infty$, such that 
	\begin{itemize}
		\item[$\bullet$] $\varphi_i$ and $\varphi_i^*$ are $X$-homotopic to each other in $ \Z_{n}^G(M;\M ;\mZ_2) $ with finenesses tending to zero;
		\item[$\bullet$] ${\bf L}(S^*)=\limsup_{i\rightarrow\infty} \max_{x\in \mathrm{dmn}(\varphi_i^*)} {\bf M}(\varphi_i^*(x))<{\bf L}(S)$.
	\end{itemize}
\end{theorem}
\begin{proof}
	The proof is essentially the same as \cite[Theorem 4.10]{pitts2014existence}. 
	Firstly, for any $G$-annuli $\an_1,\an_2$ and $p\in \an_1\cap\an_2$, it's clear that the orbit $G\cdot p$ is also contained in $\an_1\cap\an_2$. 
	Using this fact, the covering lemmas \cite[Lemma 4.8, 4.9]{pitts2014existence} hold with $G$-annuli $\{\an\}$ in place of $\{{\rm An}\}$. 
	Furthermore, Part1 and Part2 in the proof of \cite[Theorem  4.10]{pitts2014existence} can be done similarly with Definition \ref{Def:a.m. in regular annuli} and Corollary \ref{Cor:equivalence a.m.v in regular annuli}. 
	As for Part5 and Part9, lemmas in Section \ref{Sec-G-current} would be helpful to adapt the results into the $G$-invariant version. 
	Finally, the rest parts are purely combinatorial, which could be followed without changes except for adding $G$- in front of relevant objects.
\end{proof}

For any continuous $G$-homotopy class ${\bf \Pi}\in \big[X, \Z_{n}^G(M;\mF;\mZ_2)\big]$, we can apply Theorem \ref{Thm:pitts.min.max} to the sequence $S=\{\varphi_i\}_{i\in\N}$ constructed in Proposition \ref{Prop:pulltight}. 
Consequently, we have:

\begin{theorem}\label{Thm:exist amv in critical set}
	Let ${\bf \Pi}\in \big[X, \Z_{n}^G(M;\mF;\mZ_2)\big]$ be a continuous $G$-homotopy class. 
	There exists $V\in\V^G_n(M)$ such that 
	\begin{itemize}
		\item[(i)] $\|V\|(M)=\bL({\bf \Pi})$;
		\item[(ii)] $V$ is stationary in $M$;
		\item[(iii)] $V$ is $(G,\mZ_2)$-almost minimizing in regular annuli. 
	\end{itemize}
	Moreover, if $\{\Phi_i\}_{i\in\N}$ is a min-max sequence for ${\bf \Pi}$, and $S=\{\varphi_i\}_{i\in\N}$ is given by Proposition \ref{Prop:pulltight} applied to ${\bf \Pi}$, then we can further choose $V\in {\bf C}(S)\subset {\bf C}(\{\Phi_i\}_{i\in\N})$. 
\end{theorem}

\begin{proof}
	We generalize the arguments in \cite[Theorem 3.8, Page 10]{marques2015morse} to our $G$-invariant setting. 
	Denote $\{\Phi_i\}_{i\in\N}$ to be a min-max sequence for ${\bf \Pi}$. 
	After applying the pull-tight procedure Proposition \ref{Prop:pulltight} to $\{\Phi_i\}_{i\in\N}$, we can assume every elements in ${\bf C}(\{\Phi_i\}_{i\in\N})$ is stationary in $M$, and there is a sequence $S=\{\varphi_i\}_{i\in\N}$ of discrete maps with 
	\begin{itemize}
		\item[$\bullet$] $\f(\varphi_i)\to 0$ as $i\to\infty$;
		\item[$\bullet$] the Almgren $G$-extension $\Phi_i'$ of $\varphi_i$ is $G$-homotopic to $\Phi_i$ in the flat metric; 
		\item[$\bullet$] ${\bf L}(S)={\bf L}({\bf \Pi})$;
		\item[$\bullet$] ${\bf C}(S)= {\bf C}(\{\Phi_i\}_{i\in\N})$.
	\end{itemize}
	
	Suppose no element $V\in {\bf C}(S)$ is $(G,\mZ_2)$-almost minimizing in regular annuli. 
	By Theorem \ref{Thm:pitts.min.max}, there is another sequence $S^*=\{\varphi_i^*\}_{i\in\N}$ of mappings 
	$\varphi_i^*: X( l_i)_{0}\rightarrow \Z_{n}^G(M ;\mZ_2) $ 
	for some $l_i\to\infty$ as $i\to\infty$, such that 
	\begin{itemize}
		\item[$\bullet$] $\varphi_i$ and $\varphi_i^*$ are $X$-homotopic to each other in $ \Z_{n}^G(M;\M ;\mZ_2) $ with finenesses tending to zero;
		\item[$\bullet$] ${\bf L}(S^*)=\limsup_{i\rightarrow\infty} \max_{x\in \mathrm{dmn}(\varphi_i^*)} {\bf M}(\varphi_i^*(x))<{\bf L}(S)$.
	\end{itemize}
	By Corollary \ref{Cor:discrete continuous homotopy} (i), the Almgren $G$-extension $\Phi_i^*$ of $\varphi_i^*$ is $G$-homotopic to $\Phi_i'$ for $i$ large enough, which implies $\Phi_i^*\in {\bf \Pi}$. 
	However, (ii) and (iii) in Theorem \ref{Thm:interpolation} indicate
	$$\bL({\bf \Pi})\leq \bL(\{\Phi_i^*\}_{i\in\N})\leq \bL(S^*)<\bL(S)=\bL({\bf \Pi}),  $$
	which is a contradiction. 
\end{proof}

In particular, if every non-principal orbit is isolated, the $G$-varifolds $V$ in Theorem \ref{Thm:exist amv in critical set} can be $(G,\mathbb{Z}_2)$-almost minimizing in annuli (instead of regular annuli).

\begin{remark}\label{Rem:boundary type amv}
	As pointed out in \cite[Remark 3.10]{marques2015morse}, we can further require the $G$-varifold $V$ in Theorem \ref{Thm:pitts.min.max}, \ref{Thm:exist amv in critical set}, to be $(G,\mathbb{Z}_2)$-almost minimizing of {\em boundary type} in regular annuli when dealing with boundary type $G$-homotopy classes. 
\end{remark}

\section{Regularity of $(G,\mZ_2)$-Almost Minimizing Varifolds}\label{regul-G-a.m.v}

In this section, we show the regularity of $(G,\mZ_2)$-almost minimizing varifolds constructed in Theorem \ref{Thm:exist amv in critical set}.
By Remark \ref{Rem:boundary type amv}, we only consider the regularity of varifolds that are $(G,\mathbb{Z}_2)$-almost minimizing of {\em boundary type} in regular annuli.

\subsection{Good $G$-replacement property}
To begin with, we show the existence and regularity of $G$-replacements for such varifolds. 

\begin{proposition}[Existence of $G$-replacements I]\label{Prop:exist replace 1}
	Let $V\in \mathcal{V}_n^G(M)$ be $(G,\mathbb{Z}_2)$-almost minimizing of boundary type in an open $G$-set $U$, and $K\subset U$ be a compact $G$-set. 
	Then there is a varifold $V^*\in \mathcal{V}_n^G(M)$ called a {\em $G$-replacement} of $V$ in $K$ such that:
	\begin{itemize}
		\item[(i)]  $V \llcorner (M\setminus K) = V^* \llcorner (M\setminus K)$;
		\item[(ii)] $\|V\|(M) = \|V^*\|(M)$;
		\item[(iii)] $V^*$ is $(G,\mathbb{Z}_2)$-almost minimizing of boundary type in $U$;
		\item[(iv)] $V^* = \lim_{i\rightarrow\infty} |T_i^*|$, where $T_i^* \in \mathcal{Z}_n^G(M;\mathbb{Z}_2)$ such that $T_i^*$ is $G$-locally mass minimizing in ${\rm Int}(K)$, and $T_i^*=\partial Q_i^*$ for some $Q_i^*\in {\bf I}^G_{n+1}(M;\mathbb{Z}_2)$. 
	\end{itemize}
\end{proposition}

Here we say $T^*$ is $G$-locally mass minimizing in ${\rm Int}(K)$ if for any $p\in {\rm Int}(K)$, there is a number $r>0$ such that 
		${\bf M}(T^*)\leq{\bf M}(S),$ 
	for all $S\in \mathcal{Z}_n^G(M;\mathbb{Z}_2)$ with $ \mbox{spt}(S-T^*)\subset B^G_r(p)\subset {\rm Int}(K)$.

\begin{proof}
	\noindent {\bf Step1.} For any $\epsilon,\delta>0,~Q\in {\bf I}^G_{n+1}(M;\mathbb{Z}_2)$ and $T=\partial Q \in \mathfrak{a}^G_n(U;\epsilon,\delta )$, we define $\mathcal{C}_T$ to be the set of all $S\in \mathcal{Z}_n^G(M;\mathbb{Z}_2)$ such that there exists a finite sequence $\{T_i\}_{i=1}^q\subset \mathcal Z_{n}^G(M;\mathbb{Z}_2)$ satisfying
	\begin{itemize}
		\item[$\bullet$] $T_0=T$, $T_q=S$, and $\mbox{spt}(T-T_i)\subset K$, for all $i=1,\dots, q$;
		\item[$\bullet$] $\mathcal{F}(T_i-T_{i-1})\leq \delta$, for all $i=1,\dots, q$;
		\item[$\bullet$] ${\bf M}(T_i)\leq {\bf M}(T)+\delta$, for all $i=1,\dots, q$.
	\end{itemize}
	By the $G$-isoperimetric Lemma \ref{Lem:isoperimetric}, every $S\in \mathcal{C}_T$ is a boundary of an element in $\mI_{n+1}^G(M;\mZ_2)$ when $\delta>0$ is small enough. 
	
	Take a sequence $\{S_j\}_{j=1}^\infty \subset \mathcal{C}_T$ such that 
	$$\lim_{j\rightarrow\infty} {\bf M}(S_j) = \mbox{inf}\{ {\bf M}(S): S\in \mathcal{C}_T \}  .$$
	Since $\mbox{spt}(T-S_j)\subset K$, we have ${\bf M}(S_j-T)\leq 2{\bf M}(T)+\delta$ and ${\bf M}(\partial S_j-\partial T) = 0$. 
	By the Compactness Theorem \ref{Lem:compactness for G-current}, there is a subsequence $S_j-T$ (without changing notations) converging in the flat semi-norm to a $G$-invariant $n$-cycle $P\in\Z_n^G(M;\mZ_2)$ supported in $K$. 
	Hence, $T^* := T+P\in \mathcal{Z}_n^G(M;\mathbb{Z}_2)$ satisfies
	\begin{itemize}
		\item[$\bullet$] $\mbox{spt}(T-T^*)\subset K$;
		\item[$\bullet$] ${\bf M}(T^*)\leq \liminf_{j\rightarrow\infty} {\bf M}(S_j) = \mbox{inf}\{ {\bf M}(S): S\in \mathcal{C}_T \}\leq {\bf M}(T)+\delta$.
	\end{itemize}
	Moreover, for $j$ large enough, we have $\mathcal{F}(S_j-T^*)\leq \delta$. 
	Now consider the sequence which makes $S_j$ in $\mathcal{C}_T$, and add one more element $T^*$ into its end. 
	The new sequence satisfies all requirements in the definition of $\mathcal{C}_T$.  
	Thus we get $T^*\in \mathcal{C}_T$, and $T^*$ is a mass minimizer in $\mathcal{C}_T$, i.e. ${\bf M}(T^*)= \mbox{inf}\{ {\bf M}(S): S\in \mathcal{C}_T \}$.
	
	We also have $T^*\in \mathfrak{a}^G_n(U;\epsilon,\delta )$. 
	Otherwise, there is a sequence $T^*_0=T^*,T^*_1,\dots,T^*_q$ forms a $G$-invariant $(\epsilon,\delta)$-deformation of $T^*$ in $U$. 
	Consider the finite sequence $\{T_i\}_{i=1}^{q'}$ which makes $T^*=T^*_0$ in $\mathcal{C}_T$, and insert $\{T^*_i\}_{i=1}^q$ at its end. 
	By Definition \ref{Def:a.m.deform}, we get a $G$-invariant $(\epsilon,\delta)$-deformation $\{T_1,\dots,T_{q'},T_0^*,\dots,T_q^*\}$ of $T$, which contradicts to the choice of $T$.

	\begin{claim}\label{Claim:mass minimizing}
		$T^*$ is $G$-locally mass minimizing in ${\rm Int}(K)$.
	\end{claim}
	\begin{proof}
		For any $p\in {\rm Int}(K)$, we first choose $r>0$ small enough such that $B^G_r(p)\subset {\rm Int}(K)$. 
		Since $|T^*|$ is rectifiable, one can take $r$ even smaller so that $\|T^*\|(B^G_r(p)) \leq \delta/2$. 
		If there is a $G$-cycle $S\in \mathcal{Z}_n^G(M;\mathbb{Z}_2)$ with $ \mbox{spt}(S-T^*)\subset B^G_r(p)\subset {\rm Int}(K)$ such that ${\bf M}(T^*)>{\bf M}(S)$, then
		\begin{eqnarray*}
			\mathcal{F}(S-T^*) &\leq & {\bf M}(S-T^*)
			\\
			&=& \|S-T^*\|(B^G_r(p))
			\\
			&\leq & (\|S\|+\|T^*\|)(B^G_r(p))\leq \delta.
		\end{eqnarray*}
		Similar to before, we can add $S$ to the end of the finite sequence that makes $T^*$ in $\mathcal{C}_T$. 
		Hence, $S\in \mathcal{C}_T$ and ${\bf M}(T^*)>{\bf M}(S)$, which is a contradiction. 
	\end{proof}

	\noindent {\bf Step2.} Since $V$ is $(G,\mathbb{Z}_2)$-almost minimizing of boundary type in an open $G$-set $U\subset M$, there exist $\epsilon_i,\delta_i\rightarrow 0,~Q_i\in {\bf I}^G_{n+1}(M;\mathbb{Z}_2)$ and $T_i=\partial Q_i \in \mathfrak{a}^G_n(U;\epsilon_i,\delta_i )$, such that $V=\lim_{i\to\infty} |T_i|$. 
	By Step1, there is a mass minimizer $T_i^*\in \mathcal{C}_{T_i}$ for every $i\in\N$. 
	Using the Compactness Theorem (Remark \ref{Rem:compactness for G-varifold}), $|T_i^*|$ converges to a $G$-varifold $V^*\in \mathcal{V}_n^G(M)$ after passing to a subsequence. 
	We claim $V^*=\lim_{i\to\infty} |T_i^*|$ is exactly what we want. 
	Indeed, since $\mbox{spt}(T^*_i-T_i)\subset K$, we obtain (i) $\mbox{spt}(V^*-V)\subset K$. 
	Due to the facts that $T_i\in\mathfrak{a}^G_n(U;\epsilon_i,\delta_i ) $ and $T_i^*\in\mathcal{C}_{T_i}$, the inequality 
	$ {\bf M}(T_i)-\epsilon_i\leq{\bf M}(T^*_i)\leq{\bf M}(T_i) $
	holds, and hence (ii) $\|V\|(M) = \|V^*\|(M)$ is true. 
	As for (iii)(iv), they clearly follow from the facts that $T^*_i\in \mathfrak{a}^G_n(U;\epsilon_i,\delta_i )$ is a boundary for $i$ large, and $T^*_i$ is $G$-locally mass minimizing in ${\rm Int}(K)$ by Claim \ref{Claim:mass minimizing}. 
\end{proof}

Note $G$-cycles $\{T_i^*\}_{i\in\N}$ in the above proposition are only $G$-locally mass minimizing. 
However, following an averaging procedure, we can show the $G$-cycle $T_i^*$ is locally mass minimizing and thus has good regularity.

\begin{proposition}[Existence of $G$-replacements II]\label{Prop:exist replace 2}
	The $G$-cycles $\{T_i^*\}_{i\in\N} \subset \mathcal{Z}_n^G(M;\mathbb{Z}_2)$ obtained in Proposition \ref{Prop:exist replace 1} (iv) are all locally mass minimizing in ${\rm Int}(K)$. 
\end{proposition}
\begin{proof}
	For any $\epsilon,\delta>0$ small enough, suppose $Q\in {\bf I}^G_{n+1}(M;\mathbb{Z}_2)$ and $T=\partial Q\in \mathfrak{a}^G_n(U;\epsilon,\delta )$. 
	One can follow Step1 of the proof of Proposition \ref{Prop:exist replace 1} to get a mass minimizer $T^*$ in $\mathcal{C}_T$. 
	In what follows, we are going to show that $T^*$ is locally mass minimizing in ${\rm Int}(K)$. 
	
	For any $p\in {\rm Int}(K)$, let $\rho>0$ such that $\|T^*\|(B^G_{\rho}(p))\leq \delta/2$ and $B^G_\rho(p)\subset {\rm Int}(K)$. 
	Then $T^*$ is $G$-mass minimizing in $B^G_\rho(p)\subset {\rm Int}(K)$ by Claim \ref{Claim:mass minimizing}. 
	Since every element in $\mathcal{C}_T$ is a $G$-boundary, there exists $Q^*\in {\bf I}^G_{n+1}(M;\mathbb{Z}_2)$ so that $T^*=\partial Q^*$. 
	By the $G$-slice Lemma \ref{Lem:slice}, one can choose $\rho_0\in (0,\rho)$ such that $D\in {\bf I}^G_n(\partial B^G_{\rho_0}(p);\mathbb{Z}_2)$ is the slice of $Q^*$ by $d_p={\rm dist}_M(G\cdot p,\cdot)$ at $\rho_0$. 
	Moreover, we have $\partial (Q^*\llcorner \{d_p\leq \rho_0\})= T^*\llcorner \{d_p< \rho_0\} + D$. 
	Denote $X:=Q^*\llcorner \{d_p\leq \rho_0\} \in {\bf I}^G_{n+1}(M;\mathbb{Z}_2)$ and $B:=B^G_{\rho_0}(p)$. 
	
	We are going to show that $T^*$ is mass minimizing in $B$. 
	Suppose $S^*\in  \mathcal{Z}_n(M;\mathbb{Z}_2)$ such that $\mbox{spt}(T^*-S^*)\subset B$ and ${\bf M}(S^*)\leq {\bf M}(T^*)$ ($S^*$ may not be $G$-invariant). 
	It's sufficient to prove ${\bf M}(S^*)= {\bf M}(T^*)$. 
	
	Noting $\mathcal{F}(T^*-S^*)\leq \|T^*-S^*\|(B)\leq 2\|T^*\|(B)\leq \delta$ and $\delta$ is sufficiently small, there exists an $(n+1)$-chain $Y\in {\bf I}_{n+1}(M;\mathbb{Z}_2)$ so that $\partial  Y = S^*\llcorner B+D$ by the Isoperimetric Lemma (\cite[Lemma 3.1]{marques2017existence}). 
	
	Now we apply a similar averaging argument as in \cite[Proposition 4.4]{liu2021existence}. Define a $G$-invariant function on $\overline{B}$ as:
	$$ f_Y(q):=\int_G {\rm 1}_{\mbox{spt}(Y)}(g\cdot q)~d\mu(g)\in[0,1] ,\quad q\in \overline{B}.$$
	Since $\mbox{spt}(Y)$ is closed, the characteristic function $ {\rm 1}_{\mbox{spt}(Y)}$ is upper-semicontinuous. By Fatou's Lemma, $f_Y$ is upper-semicontinuous too. 
	Hence, $$Y_\lambda := f_Y^{-1}[\lambda,1 ]=\overline{B}-f_Y^{-1}[0,\lambda )$$
	is a $G$-invariant closed set in $\overline{B}$ for every $\lambda\in [0,1]$.
	
	Define then $E_f := f_Y\cdot [[\overline{B}]]$, where $[[\overline{B}]]$ is the integral current induced by $\overline{B}$. For any $n$-form $\omega$ on $M$, we have
	\begin{eqnarray}\label{Eq-boundary-Ef}
		\partial E_f(\omega) &=& \int_{\overline{B}}\int_G \langle d\omega, \xi\rangle ~{\rm 1}_{\mbox{spt}(Y)}(g\cdot q)~d\mu(g)d\mathcal{H}^{n+1}(q)
		\\
		&=& \int_G\int_{\overline{B}} \langle d\omega, \xi\rangle ~{\rm 1}_{g^{-1}\cdot\mbox{spt}(Y)}(q)~d\mathcal{H}^{n+1}(q)d\mu(g)\nonumber
		\\
		&=& \int_G \partial ((g^{-1})_\# Y)(\omega) ~d\mu(g).\nonumber
	\end{eqnarray}
	Hence, using the lower semi-continuity of mass and the fact that $G$ acts as isometries, we have
	\begin{equation}\label{Ef.ineq}
		{\bf M}(\partial E_f) \leq \int_G {\bf M}(\partial ((g^{-1})_\# Y))~d\mu(g) = {\bf M}(\partial Y). 
	\end{equation}
	Combining this with ${\bf M}(E_f)\leq {\bf M}(\overline{B})$, it's clear that $E_f$ is a normal current. 
	By \cite[4.5.9(12)]{federer2014geometric}, we have $\partial (f^{-1}[\lambda,1])$ is a rectifiable current for almost all $\lambda\in[0,1]$, which implies $Y_\lambda$ is an integral $G$-current for almost all $\lambda\in[0,1]$. 
	Moreover, by \cite[4.5.9(13)]{federer2014geometric} and (\ref{Ef.ineq}), we have $\partial E_f =\int_0^1 \partial Y_\lambda~d\lambda,$ and 
	\begin{eqnarray}\label{Ef.ineq2}
		\int_0^1 {\bf M}(\partial Y_\lambda)~d\lambda = {\bf M}(\partial E_f) \leq  {\bf M}(\partial Y).   
	\end{eqnarray}

	Since $\mbox{spt}(T^*-S^*)$ is a closed set in $B$, there exists $0<r<\rho_0$ close to $\rho_0$ such that $X=Y$ in $\an(p,r,\rho_0)$. 
	Hence, for any $\lambda\in (0,1)$, we have $f_Y={\rm 1}_{\mbox{spt}(Y)}={\rm 1}_{\mbox{spt}(Y_\lambda)}$ in $\an(p,r,\rho_0)$, and $\partial X=\partial Y = \partial Y_\lambda$ in $\an(p,r,\rho_0)$. 
	We can denote then $S^*_\lambda:=\partial Y_\lambda-D$. 
	
	Noting ${\bf M}(\partial Y)={\bf M}(D)+{\bf M}(S^*\llcorner B)$ 
	and ${\bf M}(\partial Y_\lambda) = {\bf M}(D)+{\bf M}(S^*_\lambda)$, inequality (\ref{Ef.ineq2}) now implies 
	$$\int_0^1 {\bf M}(S^*_\lambda )~d\lambda\leq {\bf M}(S^*\llcorner B), $$
	and thus ${\bf M}(S^*_{\lambda} )\leq{\bf M}(S^*\llcorner B)$ for $\lambda$ on a positive measure set in $[0,1]$.
	Therefore, we can choose $\lambda_0\in(0,1)$ such that $\widetilde{T}^* :=T^*\llcorner (M\setminus B)+S^*_{\lambda_0}\in \mathcal{Z}_n^G(M;\mathbb{Z}_2),$
	and 
	\begin{eqnarray*}
		{\bf M}(\widetilde{T}^*) &= & {\bf M}(T^*\llcorner (M\setminus B))+{\bf M}(S^*_{\lambda_0})
		\\
		&\leq & {\bf M}(T^*\llcorner (M\setminus B))+{\bf M}(S^*\llcorner B)
		\\
		&=& {\bf M}(S^*) \leq {\bf M}(T^*).
	\end{eqnarray*}
	Since $T^*$ is $G$-mass minimizing in $B=B^G_\rho(y)$, it's clear that ${\bf M}(\widetilde{T}^*)={\bf M}(S^*)={\bf M}(T^*)$. 
	Hence, we have $T^*$ is locally mass minimizing in ${\rm Int}(K)$. 
\end{proof}

\begin{proposition}[Regularity of $G$-replacements]\label{Prop:regul replace}
	Suppose $2\leq n\leq 6$. 
	Under the same hypotheses of Proposition \ref{Prop:exist replace 1}, then $V^*\llcorner {\rm Int}(K)$ is integer rectifiable and $\Sigma := {\rm spt}(\|V^*\|)\cap {\rm Int}(K) $ is a $G$-invariant smoothly embedded stable minimal hypersurface.
\end{proposition}
\begin{proof}
	By \cite{morgan1986regularity} (see the proof of \cite[Theorem 2.7]{marques2017existence}), the fact that $T_i^*$ is locally mass minimizing in ${\rm Int}(K)$ (Proposition \ref{Prop:exist replace 2}) implies $T_i^*$ is a $G$-invariant smoothly embedded minimal hypersurface in ${\rm Int}(K)$. 
	\begin{claim}
		$T_i^*$ is stable in ${\rm Int}(K)$. 
	\end{claim}
	\begin{proof}
		Since $T_i^*$ is a boundary of some $Q_i^*\in {\bf I}^G_{n+1}(M;\mathbb{Z}_2)$, we have $\Sigma := {\rm spt}(T_i^*)\cap {\rm Int}(K)\subset\partial ({\rm spt}(Q^*_i)\cap {\rm Int}(K))$ is a boundary type $G$-invariant minimal hypersurface. By Lemma \ref{Lem:G-stable}, we only need to show $\Sigma$ is $G$-stable. 
		
		Suppose $\Sigma$ is $G$-unstable. 
		By (\ref{Eq:eigenvector field}), there exists a normal $G$-vector field $X_G\in \mathfrak{X}^{\perp,G}(\Sigma)$ such that $LX_G=-\lambda_1 X_G$ with a negative eigenvalue $\lambda_1<0$, where $L$ is the Jacobi operator of $\Sigma$. 
		Consider the variation $T_t:=(f^{X_G}_t)_\# T_i^*\in  \mathcal{Z}_n^G(M;\mathbb{Z}_2) ,~t\in (-t_0,t_0)$, where $\{f^{X_G}_t\}$ are diffeomorphisms generated by $X_G$ (extended equivariantly on a $G$-neighborhood of $\Sigma$). 
		By taking $t_0>0$ small enough, we have ${\bf M}(T_t)<{\bf M}(T^*_i) $ for all $t\in (-t_0,t_0)$. 
		Additionally, we have $\mathcal{F}(T_t-T^*_i)< \delta_i$, for $t>0$ sufficiently small. 
		After taking the finite sequence that makes $T_i^*$ in $\mathcal{C}_{T_i}$, we can append $T_t$ at its end and see $T_t\in \mathcal{C}_{T_i}$. 
		Thus we get a contradiction with the minimality of ${\bf M}(T_i^*)$ in $\mathcal{C}_{T_i}$. 
	\end{proof}
	Finally, using the compactness theorem for stable minimal hypersurfaces (c.f. \cite[Theorem 1.3]{de2013existence}), the limit $V^*$ of $|T^*_i|$ (after passing to a subsequence) is a $G$-invariant smooth stable minimal hypersurface in ${\rm Int}(K)$. 
\end{proof}

\subsection{Tangent cones}
We now use the good $G$-replacement property to show the rectifiability of $V$ and classify its tangent cones. 
In the rest of this section, we always assume
$\mathcal{H}^{n-1}(M\setminus M^{reg})=0.$
By the monotonicity formula, this assumption implies that $\|V\|(M\setminus M^{reg})=0$ for any stationary $n$-varifold $V\in \mathcal{V}_n(M)$. 
As a result, for any stationary $n$-varifold $V$, the rectifiability of $V\llcorner M^{reg}$ implies the rectifiability of $V$ itself. 
With this benefit, we only classify the tangent cones on $p\in {\rm spt}(\|V\|)\cap M^{reg}$. 
Let us begin with the following lemma about the uniform volume ratio bound, which gives the rectifiability of $V$. 

\begin{lemma}\label{Lem:volum ratio bound}
	Let $2\leq n\leq 6$ and $\mathcal{H}^{n-1}(M\setminus M^{reg})=0$. 
	Suppose $V \in \mathcal{V}^G_n(M)$ is $(G,\mathbb{Z}_2)$-almost minimizing of boundary type in regular annuli and stationary in $M$. 
	There exists a constant $c=c(M)>1$ such that
	\begin{equation}\label{Eq:volum ratio}
		c^{-1}\leq \frac{\|V\|(B_\rho(p))}{\rho^n}\leq c\cdot \frac{\|V\|(B_{\rho_0}(p))}{\rho_0^n}
	\end{equation}
	 for any $p\in {\rm spt}(\|V\|)\cap M^{reg},~\rho\in(0,\rho_0)$, where $\rho_0=\rho_0(M,G)>0$. 
	 
	 Moreover, $\Theta^n(\|V\|,p)\geq\theta_0>0$, for all $p\in {\rm spt}(\|V\|)\cap M^{reg}$. 
	 Thus, $V$ is rectifiable.
\end{lemma}
\begin{proof}
	Fix any $p\in{\rm spt}(\|V\|)\cap M^{reg}$. Let $r(G\cdot p)>0$ be the constant in Definition \ref{Def:a.m. in regular annuli}. 
	We can choose $r_0\in (0,\min\{\frac{r(G\cdot p)}{20},\frac{{\rm Inj}(M)}{20}\})$ sufficiently small so that $B^G_t(p)$ is mean convex in the sense of \cite{white2009maximum} for any $0<t<r_0$. 
	Indeed, since $p\in M^{reg}$, we have $P_p = p$ (\ref{Eq:Pp}), and ${\rm dim}(G\cdot p) \leq n-2$. 
	When $t>0$ is small enough, $\partial B^G_{t}(p)$ will be close to a cylinder in local charts, which implies $B^G_t(p)$ is convex.

	Now for any $r\in (0,r_0)$, we can apply Proposition \ref{Prop:exist replace 1} to get a $G$-replacement $V^*$ of $V$ in $\overline{\an(p,r,2r)}$. 
	The maximum principle \cite[Theorem 2]{white2009maximum} implies that $\|V^*\|\llcorner \an(p,r,2r)\neq 0$. 
	Thus $V^*\llcorner\an(p,r,2r)$ is a non-empty $G$-invariant minimal hypersurface with integer multiplicity (Proposition \ref{Prop:regul replace}). 
	This implies the existence of $q\in \an(p,r,2r)$ with $\Theta^n(\|V^*\|,q)\geq 1$. 
	Since $V^*$ is $G$-invariant, it's clear $\Theta^n(\|V^*\|,g\cdot q)\geq 1$ for all $g\in G$. 
	Hence, we can further require $q\in B_{2r}(p)\setminus B_r(p)$. 
	
	By the five-times covering lemma, there exists a finite set $\{p_i\}_{i=1}^N\subset G\cdot p$ so that 
	$$B_{4r}^G(p)\subset \bigcup_{i=1}^N B_{20r}(p_i),\quad B_{4r}(p_i)\cap B_{4r}(p_j)=\emptyset,~\forall i\neq j\in\{1,\dots,N\}.$$
	Then we have 
	\begin{eqnarray*}
		N\cdot\|V^*\|(B_{2r}(q)) &\leq & N\cdot\|V^*\|(B_{4r}(p)) \quad\quad\quad\quad\qquad\mbox{(by $q\in B_{2r}(p)\setminus B_r(p) $)}
		\\
		&\leq & \|V^*\|(B^G_{4r}(p)) = \|V\|(B^G_{4r}(p)) \quad\mbox{(by Proposition \ref{Prop:exist replace 1})}
		\\
		&\leq & \|V\|(\bigcup_{i=1}^NB_{20r}(p_i)) \leq  N\cdot\|V\|(B_{20r}(p)).
	\end{eqnarray*}
	As a result, the following inequality holds by the monotonicity formula:
	\begin{eqnarray*}
		\frac{\|V\|(B_{20r}(p))}{\omega_n(20r)^n} \geq \frac{\|V^*\|(B_{2r}(q))}{\omega_n(20r)^n} \geq \frac{1}{10^nC_M} \lim_{r\to 0} \frac{\|V^*\|(B_{2r}(q))}{\omega_n(2r)^n} \geq \frac{1}{10^nC_M}.
	\end{eqnarray*}
	After taking $r\to 0$, we get a uniform positive lower bound on the density $\Theta^n(\|V\|,p)$ for all $p\in \spt(\|V\|)\cap M^{reg}$. 
	Allard's rectifiability theorem now implies the rectifiability of $V\llcorner M^{reg}$ as well as $V$ by $\mathcal{H}^{n-1}(M\setminus M^{reg})=0$. 
	Finally, (\ref{Eq:volum ratio}) follows from the classical monotonicity formula. 
\end{proof}

For any stationary varifold $V\in \mathcal{V}_n(M)$, $p\in{\rm spt}(\|V\|)$ and $r<r_0={\rm Inj}(M)$, consider the rescaled reversed exponential map $\widetilde{{\bm \eta}}_{p,r}: B_{r_0}(p)\ni q\rightarrow \frac{1}{r}{\rm exp}_p^{-1}(q)\in T_pM$. 
As in \cite[Section 5.2]{de2013existence}, for any sequence $r_i\rightarrow 0$, there exists a subsequence of $\{\widetilde{{\bm \eta}}_{p,r_i\#} V\}_{i\in\N}$ converging to a stationary varifold $C\in \mathcal{V}_n(T_pM)$. 
We denote ${\rm VarTan}(V,p)$ to be all such limit varifolds. 
In the following lemma, we show the rectifiability of varifolds in ${\rm VarTan}(V,p)$ provided $p\in M^{reg}$ and $V$ is stationary in $M$ as well as $(G,\mathbb{Z}_2)$-almost minimizing of boundary type in regular annuli. 

\begin{lemma}\label{Lem:cone}
	Let $2\leq n\leq 6$ and $\mathcal{H}^{n-1}(M\setminus M^{reg})=0$. Suppose $V \in \mathcal{V}^G_n(M)$ is $(G,\mathbb{Z}_2)$-almost minimizing of boundary type in regular annuli and stationary in $M$. 
	For any $p\in{\rm spt}(\|V\|)\cap M^{reg}$ and $C\in{\rm VarTan}(V,p)$, we have: 
	\begin{itemize}
		\item[(i)] $C\in \mathcal{V}_n(T_pM)$ is rectifiable;
		\item[(ii)] $C$ is stationary in $T_pM$;
		\item[(iii)] $C$ is a cone, and has a product structure as $C=T_pG\cdot p\times W$ for some $W\in \mathcal{V}_{n-{\rm dim}(G\cdot p)}({\bf N}_pG\cdot p)$. 
	\end{itemize}
\end{lemma}
\begin{proof}
	Let $p\in{\rm spt}(\|V\|)\cap M^{reg}$ and $r_i\rightarrow 0$ such that $ C=\lim_{i\to \infty} V_i$, where $V_i=\widetilde{{\bm \eta}}_{p,r_i\#} V$. 
	By \cite[Corollary 42.6]{simon1983lectures}, $C$ is stationary. 
	We claim that: 
	\begin{eqnarray}
		&& {\rm spt}(\|V_i\|) {\rm ~converges~ to~} {\rm spt}(\|C\|) {\rm ~in ~the~ Hausdorff~ topology};
		\\
		&& \Theta^n(\|C\|,q)\geq \theta_0>0, ~\forall q\in{\rm spt}(\|C\|).
	\end{eqnarray}
	Indeed, if $q_i\in {\rm spt}(\|V_i\|)\rightarrow q\in T_pM$, then for any $\rho>0$, there exists $q_i\in \mathbb{B}_{\rho/2}(q)\subset T_pM$, and hence $\mathbb{B}_{\rho/2}(q_i)\subset\mathbb{B}_\rho(q) $. By the volume ratio bound in Lemma \ref{Lem:volum ratio bound}, we have the following inequality when $\rho$ is sufficiently small:  
	$$ c^{-1}\leq \frac{\|V_i\|(\mathbb{B}_{\rho/2}(q_i))}{(\rho/2)^n}\leq \frac{\|V_i\|(\mathbb{B}_\rho(q))}{(\rho/2)^n}$$  
	for some $c=c(M)>1$. 
	Taking $i\rightarrow\infty$, we have $\|C\|(\mathbb{B}_\rho(q))\geq c^{-1}(\rho/2)^n$ for any $\rho>0$ sufficiently small, which implies $q\in\spt(\|C\|)$ and $\Theta^n(\|C\|,q)\geq \theta_0$ for some uniform constant $\theta_0>0$. 
	The Allard's rectifiability theorem (\cite[Theorem 5.5]{allard1972first}) now implies $C$ is rectifiable. 
	Additionally, $C$ is cone by \cite[Theorem 19.3]{simon1983lectures}. 
	
	We now show the product structure of $C$ as in \cite[Lemma A.2]{liu2021existence}. 
	First, by \cite{moore1980equivariant}, there is an orthogonal representation $\rho:G\to O(\R^L)$ and an isometric $G$-equivariant embedding $F$ from $M$ to $\R^L$, i.e. $F(g\cdot x) = \rho(g)\cdot F(x) $. 
	For simplicity, we identify $M$ with its image $F(M)\subset \R^L$, and denote the action of $g$ on $x$ by $g\cdot x$, for every $x\in\R^L,g\in G$ (omit the notation $\rho$). 
	Define a map $\bleta_{p,r}:\R^L\to\R^L$ as $\bleta_{p,r}(x)=(x-p)/r$. 
	
	For any $w\in T_pG\cdot p$, there is a curve $g(t)$ in $ G$ so that $g(0)=e$ and $\frac{d}{dt}\big\vert_{t=0}g(t)\cdot p=w$. 
	By the orthogonal representation $\rho$, we can write $g(t)=I_L+tA(t)\subset O(L)$, which implies $A(0)\cdot p=w$. 
	A direct computation shows
	\begin{eqnarray}\label{Eq:splitting}
		(\bleta_{p, r_i}\circ g(r_i))(x) &=& \frac{g(r_i)\cdot x - g(r_i)\cdot p}{r_i}+\frac{g(r_i)\cdot p -  p}{r_i}
		\\
		&=& (\btau_{-A(r_i)\cdot p}\circ g(r_i)\circ \bleta_{p, r_i})(x),\nonumber
	\end{eqnarray}
	where $\btau_{-A(r_i)\cdot p}$ is the translation map in $\R^L$ as $x \mapsto x+A(r_i)\cdot p$. 
	
	On the other hand, we can write $\exp_p^{-1}(q)= (q-p) + f(q)\in T_pM\subset \R^L$ for some map $f$ satisfying $|f(q)|=O((\dist_M(p,q))^2)$. 
	Therefore, $$\widetilde{\bleta}_{p,r_i}(q) =\bleta_{p,r_i}(q) + \frac{f(q)}{r_i} .$$
	If $\{q_i\}_{i\in\N}\subset M$ satisfies $\widetilde{\bleta}_{p,r_i}(q_i) :=\frac{1}{r_i}{\rm exp}_p^{-1}(q_i)\to v\in T_pM$, then we have $\dist_M(q_i,p)=O(r_i)$ for $i$ large enough, which implies $\frac{|f(q_i)|}{r_i}=O(r_i)\to 0$ as $r_i\to 0$. 
	As a result, we can show $(\bmu_{r_i}\circ f )_\# V \to 0 $ as $r_i\to 0$, and 
	\begin{eqnarray*}
		C &= &\lim_{r_i\to 0}(\widetilde{\bleta}_{p,r_i})_\#V 
		\\
		&=& \lim_{r_i\to 0} ( {\bleta}_{p,r_i})_\# V =\lim_{r_i\to 0} ( {\bleta}_{p,r_i}\circ g(r_i))_\# V 
		\\
		&=&\lim_{r_i\to 0}  (\btau_{-A(r_i)\cdot p}\circ g(r_i)\circ \bleta_{p, r_i})_\# V  =(\btau_{-w})_\# C,
	\end{eqnarray*}
	where $ \bmu_{r_i}$ is the homothety map in $\R^L$ as $x\mapsto x/r_i$. 
	Hence, $C$ is invariant under the translation along $T_pG\cdot p$. 
	Note $C$ is rectifiable and $T_pM=T_pG\cdot p\times {\bf N}_pG\cdot p$, we have $C=T_pG\cdot p\times W$ for some $W\in \mathcal{V}_{n-{\rm dim}(G\cdot p)}({\bf N}_pG\cdot p)$. 
	\end{proof}

We now show the main result of this subsection that is the classification of tangent cones. 
\begin{proposition}\label{Prop:classification cone}
	Let $2\leq n\leq 6$ and $\mathcal{H}^{n-1}(M\setminus M^{reg})=0$. Suppose $V \in \mathcal{V}^G_n(M)$ is $(G,\mathbb{Z}_2)$-almost minimizing of boundary type in regular annuli and stationary in $M$. 
	For any $p\in{\rm spt}(\|V\|)\cap M^{reg}$ and $C\in{\rm VarTan}(V,p)$, we have: 
	\begin{itemize}
		\item[(i)] $C$ is a hyperplane in $T_pM$ with multiplicity $\Theta^n(\|V\|,p)\in\mathbb{N}$;
		\item[(ii)] $\Theta^n(\|V\|,p)\geq 1$ and $V\in \mathcal{IV}_n^G(M)$.
	\end{itemize}
\end{proposition}
\begin{proof}
	Let $p\in {\rm spt}(\|V\|)\cap M^{reg}$ and $r_i\rightarrow 0$ such that $ C=\lim_{i\to \infty} V_i$, where $V_i=\widetilde{{\bm \eta}}_{p,r_i\#} V$. 
	For any $\alpha\in(0,\frac{1}{8})$, denote $K_i:=\overline{\an(p,\alpha r_i,r_i)}$. 
	By Proposition \ref{Prop:exist replace 1}, there exists a $G$-replacement $V_i^*$ of $V$ in $K_i$. 
	For $r_i$ small enough, we have the convexity of $B_{r_i}^G(p)$. 
	Therefore, $\|V_i^*\|\llcorner{\rm Int}(K_i)\neq 0 $ by the maximum principle \cite[Theorem 5.1(ii)]{de2013existence} (see the arguments in \cite[Lemma 5.3]{liu2021existence} and \cite[Lemma 5.2]{de2013existence}). 
	Hence, $\Sigma_i^*:={\rm spt}(\|V_i^*\|)\cap {\rm Int}(K_i)$ is a non-trivial smooth embedded hypersurface that is minimal and stable in ${\rm Int}(K_i)$ (Proposition \ref{Prop:regul replace}). 
	By the Compactness Theorem (Remark \ref{Rem:compactness for G-varifold}), we obtain a limit varifold after passing to a subsequence: 
	$$C':=\lim_{i\to\infty} \widetilde{{\bm \eta}}_{p,r_i\#}V_i^*\in\mathcal{V}_n(T_pM). $$
	Moreover, the argument in Lemma \ref{Lem:cone} shows that $C'$ is rectifiable stationary in $T_pM$ and has a product sturctue $C'=T_pG\cdot p\times W'$, 
	for some $W'\in \mathcal{V}_{n-{\rm dim}(G\cdot p)}({\bf N}_pG\cdot p)$. 
	
	For $r_i$ small enough, we have $\frac{3}{4}\leq|d\exp_p|\leq \frac{5}{4}$ and
	$${\rm \mathbb{A}n}(T_pG\cdot p,2\alpha , 1/2)\cap \mathbb{B}_{1}(0)\subset \frac{1}{r_i}\cdot\exp_p^{-1}\big( \an( p,\alpha r_i, r_i)\cap B_{r_i}(p) \big),$$
	where $\mathbb{B}_{1}(0)=\{v\in T_pM:|v|<1\}$ and 
	$$ {\rm \mathbb{A}n}(T_pG\cdot p,s, t)=\big\{ v\in T_pM : s<{\rm dist}_{T_pM}(v,T_pG\cdot p)<t  \big\}.$$ 
	Denote $ A:= {\rm \mathbb{A}n}(T_pG\cdot p,2\alpha , 1/2)\cap \mathbb{B}_{1}(0)\subset T_pM $. 
	Thus $\Sigma_i:=(\widetilde{{\bm \eta}}_{p,r_i\#}\Sigma_i^*)\llcorner A$ is a smooth and stable minimal hypersurface in $A$ (under the rescaled pull-back metric which converges to the Euclidean metric). 
	Moreover, by the five-times covering lemma, we have a finite set $\{p_i^j\}_{j=1}^N\subset G\cdot p$ such that 
	$$B_{r_i}^G(p)\subset \bigcup_{j=1}^N B_{5r_i}(p_i^j),\quad B_{r_i}(p_i^{j_1})\cap B_{r_i}(p_i^{j_2})=\emptyset, \quad \forall j_1\neq j_2\in\{1,\dots,N\}.$$
	By Proposition \ref{Prop:exist replace 1} (ii) and the fact that $G$ acts as isometries,  we have: 
	\begin{eqnarray*}
		N\cdot \mathcal{H}^n(\Sigma_i) &\leq & N\cdot \frac{2^n}{r_i^n} \mathcal{H}^n(\Sigma_i^*\cap B_{r_i}(p)) ~\leq~  N\cdot \frac{2^n}{r_i^n}\|V_i^*\|(B_{r_i}(p))
		\\
		&\leq & \frac{2^n}{r_i^n}\|V_i^*\|(B_{r_i}^G(p)) ~=~ \frac{2^n}{r_i^n}\|V\|(B_{r_i}^G(p))
		\\ 
		&\leq & \frac{2^n}{r_i^n}\|V\|(\bigcup_{j=1}^NB_{5r_i}(p_i^j)) ~\leq ~  N\cdot \frac{2^n}{r_i^n}\|V\|(B_{5r_i}(p)).
	\end{eqnarray*}
	As a result, $\Sigma_i$ has uniformly bounded mass $\mathcal{H}^n(\Sigma_i)\leq 10^n\frac{\|V\|(B_{5r_i}(p))}{(5r_i)^n}\leq 10^nC_M\frac{\|V\|(B_{r_0}(p))}{r_0^n} $ by the monotonicity formula. 
	By the Compactness Theorem \cite[Theorem 2]{schoen1981regularity} for stable minimal hypersurfaces (under varying metrics), $\Sigma_i$ converges (up to a subsequence) to some smooth embedded stable minimal hypersurface $\Sigma$ in $T_pM$ (under the Euclidean metric). 
	Then, it's clear that
	$$\Sigma = {\rm spt}(\|C'\|)\cap A = {\rm spt}(\|C'\|)\cap{\rm \mathbb{A}n}(T_pG\cdot p,2\alpha , 1/2)\cap \mathbb{B}_{1}(0) $$
	is a smooth embedded stable minimal hypersurface in $A$. 
	
	By Proposition \ref{Prop:exist replace 1} (i),(ii), Lemma \ref{Lem:cone} (iii), and $C'=T_pG\cdot p\times W'$, we have
	$$\|W\|(\mathbb{B}_r^{n+1-{\rm dim}(G\cdot p)}(0)) = \|W'\|(\mathbb{B}_r^{n+1-{\rm dim}(G\cdot p)}(0)),~\forall r\in (0,\alpha)\cup(1,\infty) ,$$
	where $\mathbb{B}_r^{n+1-{\rm dim}(G\cdot p)}(0)\subset {\bf N}_pG\cdot p$. 
	Since $C$ as well as $W$ is a cone, we further have
	$$ \frac{\|W'\|(\mathbb{B}_r^{n+1-{\rm dim}(G\cdot p)}(0))}{r^{n-{\rm dim}(G\cdot p)}} = \frac{\|W'\|(\mathbb{B}_\rho^{n+1-{\rm dim}(G\cdot p)}(0))}{\rho^{n-{\rm dim}(G\cdot p)}},~\forall r,\rho\in (0,\alpha)\cup(1,\infty). $$
	As a result, $W'\in\mathcal{V}_{n-{\rm dim}(G\cdot p)}({\bf N}_pG\cdot p)$ is also a cone, and $W,W'$ agree outside ${\rm \mathbb{A}n}^{n+1-{\rm dim}(G\cdot p)}(0,\alpha , 1)$. Thus we have $W=W'$ and $C=C'$. 
	
	Finally, by the product structure of $C'$ and $T_pM=T_pG\cdot p\times {\bf N}_pG\cdot p$ (noting $T_pG\cdot p$ is flat), it's clear that $\Sigma\cap {\bf N}_pG\cdot p={\rm spt}(\|W'\|)\cap{\rm \mathbb{A}n}^{n+1-{\rm dim}(G\cdot p)}(0,2\alpha , 1/2)$ is a stable minimal hypersurface in ${\bf N}_pG\cdot p$.
	Moreover, since $\alpha\in(0,\frac{1}{8})$ is arbitrary, we have $W=W'$ is a stable minimal cone except possibly at the vertex $0\in {\bf N}_pG\cdot p$. 
	Therefore, $W$ is a hyperplane in ${\bf N}_pG\cdot p$ by the non-existence of stable cones in the Euclidean space ${\bf N}_pG\cdot p$ with dimensions $3\leq\dim({\bf N}_pG\cdot p)={\rm Cohom}(G)\leq7$ (\cite{simons1968minimal}). 
	Thus, $C=T_pG\cdot p\times W$ is a hyperplane in $T_pM$ with multiplicity $\Theta^n(\|V\|,p)$.\
\end{proof}



\subsection{Regularity theorem and min-max theorem}

\begin{theorem}[Regularity Theorem]\label{Thm:regul amv}
	Let $2\leq n\leq 6$ and ${\rm Cohom(G)}\geq 3$. 
	Suppose $M\setminus M^{reg}$ is a smoothly embedded submanifold of $M$ with dimension at most $n-2$. 
	If $V \in \mathcal{V}^G_n(M)$ is a $G$-varifold which is 
	\begin{itemize}
		\item[$\bullet$] stationary in $M$ and 
		\item[$\bullet$] $(G,\mathbb{Z}_2)$-almost minimizing of boundary type in regular annuli. 
	\end{itemize}
	Then $V$ is induced by closed, smooth, embedded, $G$-invariant minimal hypersurfaces.
\end{theorem}

\begin{proof}
	Using Proposition \ref{Prop:exist replace 2}, \ref{Prop:regul replace} and \ref{Prop:classification cone}, the proof of Schoen-Simon \cite[Page 789]{schoen1981regularity} would carry over under $G$-invariant restrictions. 
	We only point out some modifications. 
	Let $p\in {\rm spt}(\|V\|)$. 
	First, we can choose $r<\frac{1}{4}r(G\cdot P_p)$ sufficiently small so that $B^G_t(P_p)$ is mean convex in the sense of \cite{white2009maximum} for any $0<t\leq r$. 
	Indeed, this claim has been proved in Lemma \ref{Lem:volum ratio bound} for $p\in M^{reg}$ ($P_p=p$ in (\ref{Eq:Pp})). 
	As for the case that $p\in M\setminus M^{reg}$, $P_p$ is the connected component of $M\setminus M^{reg}$ containing $p$ by (\ref{Eq:Pp}). 
	Therefore, we still have the smoothness of the center $G\cdot P_p$ by the assumption on $M\setminus M^{reg}$. 
	Since ${\rm dim}(M\setminus M^{reg})\leq n-2$, the convexity of the tube $B^G_t(P_p)$ can be deduced similarly to Lemma \ref{Lem:volum ratio bound}.

	By the maximum principal \cite[Theorem 2]{white2009maximum} and the convexity of $B^G_r(P_p)$, if $W\in \mathcal{V}_n(M)$ is stationary in $B^G_r(P_p)$ and $W\llcorner B^G_r(P_p)\neq 0$, then 
	\begin{equation}\label{Eq:convex ball}
		\emptyset \neq {\rm spt}(\|W\|)\cap \partial B^G_r(P_p)  = {\rm Clos}\big[{\rm spt}(\|W\|)\setminus {\rm Clos}(B^G_r(p)) \big] \cap \partial B^G_r(p).
	\end{equation} 
	
	Since every small $G$-annuli around $P_p$ is contained in $M^{reg}$ (see (\ref{Eq:regular annuli})), Proposition \ref{Prop:classification cone} and all the results in the previous subsection can be applied to any $q\in \an(P_p,s,t)\in \ann_r(P_p)$. 
	Note the proof of smooth gluing in \cite[(7.18)-(7.38)]{schoen1981regularity} is localized. 
	Hence, we can construct successive $G$-replacements $V^*$ and $V^{**}$ on two overlapping annuli $\an(P_p,s,t),~\an(P_p,s_1,s_2)$, with $0<s_1<s<s_2<t$, and glue them smoothly across $\partial B^G_{s_2}(P_p)$. 
	Moreover, using Proposition \ref{Prop:classification cone} and (\ref{Eq:convex ball}), we have the unique continuation of $V$ up to the center $G\cdot P_p$ by letting $s_1\to 0$ (see \cite[Page 42 Step 3]{li2021min}). 
	Therefore, $V\llcorner (B^G_t(P_p)\setminus G\cdot P_p)$ is an integer multiple of a smooth embedded $G$-invariant hypersuface which is minimal and stable in $B^G_t(P_p)\setminus G\cdot P_p$. 
	Finally, $\mathcal{H}^{n-1}(G\cdot P_p)=0$ by our assumption, which implies the singularity at $G\cdot P_p$ can be removed using \cite[Corollary 2]{wickramasekera2014general}. 
\end{proof}

At the end of this section, we show the following min-max theorem for $G$-invariant hypersurfaces.

\begin{theorem}(Min-max theorem for $G$-invariant hypersurfaces)\label{Thm:G-minmax}
	Let $M^{n+1}$ be a closed manifold, and $G$ be a compact Lie group acting as isometries on $M$ of cohomogeneity at least $3$. 
	Suppose $2\leq n\leq 6$ and $M\setminus M^{reg}$ is a smoothly embedded submanifold of $M$ with dimension at most $n-2$.

	If ${\bf \Pi}\in\big[X,\Z_n^G(M;\mF;\mZ_2)\big]$ is a continuous $G$-homotopy class of boundary type with ${\bf L}({\bf \Pi})>0$. 
	There exists a stationary integral $G$-varifold $V\in\IV_n^G(M)$ so that 
	\begin{itemize}
		\item[$\bullet$] $\|V\|(M)={\bf L}({\bf \Pi})$;
		\item[$\bullet$] ${\rm spt}(\|V\|)$ is a closed smoothly embedded $G$-invariant minimal hypersurface.
	\end{itemize}
	Moreover, if $\{\Phi_i\}_{i\in\N}$ is a min-max sequence of ${\bf \Pi}$, we can further take $V\in{\bf C}(\{\Phi_i\}_{i\in\N})$.
\end{theorem}
\begin{proof}
	It follows directly from Theorem \ref{Thm:exist amv in critical set}, Remark \ref{Rem:boundary type amv}, and Theorem \ref{Thm:regul amv}. 
\end{proof}


\section{$(G,p)$-Sweepout and Min-max Families}\label{minmax.families}

To show the main Theorem \ref{main.thm}, we still need to prove the existence of a continuous $G$-homotopy class ${\bf \Pi} \in [X,\mathcal{Z}_n^G(M;\mF;\mathbb{Z}_2)]$ of boundary type with ${\bf L}({\bf \Pi})>0$. 
In this section, we introduce the $(G,p)$-sweepout and $(G,p)$-width as  generalizations of $p$-sweepout and $p$-width in \cite[Definition 4.1, 4.3]{marques2017existence}. 
We also show some useful lemmas which are parallel to those in \cite[Section 4]{marques2017existence}. 
Finally, we show the $\M$-continuous $(G,p)$-sweepouts form a subset in the set of $p$-sweepouts, which implies the $(G,p)$-width satisfies $\omega^G_p(M)\geq \omega_p(M)>0$. 
It follows directly that there exists a $G$-homotopy class ${\bf \Pi}\in [X,\mathcal{Z}_n^G(M;\mF;\mathbb{Z}_2)]$ of boundary type with ${\bf L}({\bf \Pi})>0$. 

\subsection{Almgren's isomorphism for $G$-cycles space}\label{almgren.iso}
As we pointed out in Remark \ref{Rem:Almgren isomorphism}, the Almgren's Isomorphism Theorem \cite[Theorem 3.2, 7.1]{almgren1962homotopy} (see also \cite[Theorem 4.6]{pitts2014existence}) holds under $G$-invariant restrictions (i.e. for $\mathcal Z_{n}^G(M;\mathbb{Z}_2)$ and $\mathcal Z_{n}^G(M;{\bf M};\mathbb{Z}_2)$). 
We also point out the double cover argument can be made as in \cite[Section 5]{marques2021morse} under $G$-invariant restrictions using Proposition \ref{Lem:compactness for G-current} and \ref{Lem:isoperimetric}. 

\begin{theorem}\label{Thm:isomorphism}
	$\mathcal Z_{n}^G(M;\mathbb{Z}_2)$ is weakly homotopically equivalent to to $\mathbb{RP}^\infty $, and  
	\begin{equation}
		\pi_1(\mathcal{Z}_n^G(M;\mathbb{Z}_2)) \cong \pi_1(\mathcal{Z}_n^G(M;{\bf M};\mathbb{Z}_2)) \cong H_{n+1}(M;\mathbb{Z}_2) \cong\mathbb{Z}_2.
	\end{equation}
\end{theorem}


Now we define the $p$-sweepouts in both continuous and discrete versions under $G$-invariant restrictions similarly to \cite[Section 4]{marques2017existence}.
\begin{definition}[$(G,p)$-sweepout]\label{defi.pclass}
	An $\F$-continuous map $\Phi:X\rightarrow \mathcal{Z}_n^G(M;\mathbb{Z}_2)$ is a {\em  $(G,p)$-sweepout}, if for all $x\in X$, $\Phi(x)=\partial Q(x)$ for some $Q(x)\in\mI_{n+1}^G(M;\mZ_2)$ and 
	$$\Phi^*(\bar \lambda^p) \neq 0 \in H^p(X;\mathbb{Z}_2),$$
	where $p\in \mathbb{N}$, $\bar{\lambda}$ is the generator of $H^1( \mathcal{Z}_n^G(M;\mathbb{Z}_2)) \cong \mZ_2 $, and $\bar{\lambda}^p :=\bar{\lambda}\smile\dots\smile\bar{\lambda}$ is the cup product of $\bar{\lambda}$ with itself for $p$-times.
\end{definition}

Our definition is a little bit different from the classical definition of $p$-sweepouts in \cite[Definition 4.1]{marques2017existence} since we only consider the maps of boundary type.

\begin{remark}\label{Rem-homo-sweep-is-sweep}
	By the isomorphism in Theorem \ref{Thm:isomorphism}, if an $\F$-continuous map $\Phi'$ is $G$-homotopic to a $(G,p)$-sweepout, then $\Phi'$ is also a $(G,p)$-sweepout.
\end{remark}

We say $X$ is {\em $(G,p)$-admissible} if there exists a $(G,p)$-sweepout $\Phi:X\rightarrow \mathcal{Z}_n^G(M;\mathbb{Z}_2)$ that has no concentration of mass on orbits. The set of all $(G,p)$-sweepouts of $M$ with no concentration of mass on orbits is denoted by $\mathcal P_p^G$. 

\begin{definition}\label{Def:G.p.width}
	The {\em $(G,p)$-width of $M$} is defined as 
	$$\omega_p^G(M):=\inf_{\Phi \in \mathcal P_p^G}\sup_{x\in {\rm dmn}(\Phi)}{\bf M}(\Phi(x)),$$
	where ${\rm dmn}(\Phi)$ is the domain of the definition of $\Phi$.
\end{definition}

\begin{definition}
	Let $\Pi \in  [X,\mathcal{Z}_n^G(M;{\bf M};\mathbb{Z}_2)]^{\#}$. 
	Then $\Pi$ is said to be a {\it class of (discrete) $(G,p)$-sweepouts} if for any $S=\{\phi_i\}_{i\in\N}\in \Pi$, the Almgren $G$-extension $\Phi_i:X\rightarrow \mathcal{Z}_n^G(M;{\bf M};\mathbb{Z}_2)$ of $\phi_i$ is a $(G,p)$-sweepout for every sufficiently large $i$. 
\end{definition}

Since our definition of $(G,p)$-sweepout needs $\Phi_i$ to be boundary type, Lemma \ref{Lem:isoperimetric} and Theorem \ref{Thm:interpolation} (i) imply that a class of discrete $(G,p)$-sweepouts is also boundary type. 
Moreover, as in \cite[Lemma 4.6]{marques2017existence}, the two definitions of $(G,p)$-sweepout under continuous and discrete settings are consistent by the following lemma.

\begin{lemma}\label{Lem:continuous discrete sweepout}
	Let 
	\begin{itemize}
		\item[$\bullet$] $\Phi:X\rightarrow \mathcal{Z}_n^G(M;\mathbb{Z}_2)$  be  a continuous map in the flat topology with no concentration of mass on orbits;
		\item[$\bullet$] $S=\{\phi_i\}_{i=1}^\infty$ be the $(X,{\bf M})$-homotopy sequence of mappings into $\mathcal{Z}_n^G(M;{\bf M};\mathbb{Z}_2)$ associated to $\Phi$ given by Theorem \ref{Thm:discritization} (i);
		\item[$\bullet$] $\Pi$ be the $(X,{\bf M})$-homotopy class of mappings into $\mathcal{Z}_n^G(M;{\bf M};\mathbb{Z}_2)$ associated with $S=\{\phi_i\}_{i=1}^\infty$.
	\end{itemize}
	Then $\Phi $ is a $(G,p)$-sweepout if and only  if $\Pi$ is a class of discrete $(G,p)$-sweepouts.
\end{lemma}
\begin{proof}
Let $\Phi_i$ be the Almgren $G$-extension of $\phi_i$, which is continuous in the ${\bf M}$-norm and has no concentration of mass on orbits (Lemma \ref{Lem:mass continu no concent}).
By Corollary \ref{Cor:discrete continuous homotopy} (ii), $\Phi_i$ is $G$-homotopic to $\Phi$ in the flat topology for every $i$ large enough. Thus the lemma follows from Remark \ref{Rem-homo-sweep-is-sweep}. 
\end{proof}

The same consistency between discrete and continuous definitions also holds for the $(G,p)$-width from the following result.

\begin{lemma}\label{pwidth.discrete}
	Let $\mathcal{D}_p^G$ be the set of all classes of discrete $(G,p)$-sweepouts $\Pi \in  [X,\mathcal{Z}_n^G(M;{\bf M};\mathbb{Z}_2)]^{\#},$ where $X$ is a $(G,p)$-admissible cubical subcomplex. 
Then $$\omega_p^G(M)=\inf_{\Pi\in \mathcal D_p^G}{\bf L}(\Pi).$$
\end{lemma}

\begin{proof}
See the proof of \cite[Lemma 4.7]{marques2017existence} with Theorem \ref{Thm:interpolation} and Lemma \ref{Lem:continuous discrete sweepout}. 
\end{proof}

\subsection{Proof of Main Theorem \ref{main.thm}}

\begin{proof}
By Theorem \ref{Thm:discritization}, \ref{Thm:interpolation} and Lemma \ref{Lem:continuous discrete sweepout}, we only need to consider the maps in $\mathcal{P}^G_p$ that are continuous in the ${\bf M}$-norm to compute $\omega^G_p(M)$. 
Note ${\bf M}$-continuous maps have no concentration of mass by \cite[Lemma 15.2]{marques2014min}. 
Thus, we have 
$\omega^G_p(M)\geq \omega_p(M)$ (see \cite[Section 4]{marques2017existence} for the definition of $\omega_p(M)$). 
As a result, the following inequality holds by the classical Almgren-Pitts min-max theory 
$$0<\omega_1(M)\leq\omega_p(M)\leq\omega^G_p(M)\leq {\bf L}({\bf\Pi}),$$
for all ${\bf\Pi}\subset\mathcal{P}^G_p$.
The main Theorem \ref{main.thm} now follows from Theorem \ref{Thm:G-minmax}. 
\end{proof}

\section{Upper Bounds of $(G,p)$-Width and the Dichotomy Theorem}\label{Sec-upper-bound}

Gromov\cite{gromov1988dimension} and Guth\cite{guth2009minimax} have studied the asymptotic behavior of the min-max volumes $\omega_p(M)$ as $p\rightarrow \infty$. 
In \cite[Section 5]{marques2017existence}, Marques and Neves have provided a modified proof for the upper bounds of $p$-width built by Guth. 
In this section, we adapt the `bend-and-cancel' argument in \cite[Theorem 5.1]{marques2017existence} to $G$-invariant settings and show the upper bounds of $(G,p)$-width depend on the cohomogeneity of $G$.

To begin with, there exists a Riemannian metric $g_{_{M/G}}$ on $M^{reg}/G$ so that $\pi: M^{reg} \to M^{reg}/G$ is a Riemannian submersion. 
Since $M/G$ can be triangulated, we can find an $l$-dimensional cubical subcomplex $K$ of $I^m$ for some $m$, and a bi-Lipschitz homeomorphism $F_1 :K \rightarrow M/G$, where $l={\rm Cohom}(G)$. 
For any $k\in \mathbb{N}$ and $\sigma\in K(k)_l$, let $\tilde{\sigma}:=\pi^{-1}\circ F_1(\sigma)$ be a $G$-invariant `cube' in $M$, and $G\cdot p_{\tilde{\sigma}}:=\pi^{-1}\circ F_1(p_\sigma)$ be the center orbit of $\tilde{\sigma}$, where $p_\sigma$ is the center point of $\sigma$.  
Denote then
\begin{itemize}
	\item[$\bullet$] $\widetilde{K}(k)_l :=\{\tilde{\sigma}: \sigma\in K(k)_l\}$ as the set of all $G$-invariant `cubes' in $M$,
	\item[$\bullet$] $\tilde{c}(k) := \{G\cdot p_{\tilde{\sigma}}: \sigma\in K(k)_l\}$ as the set of all center orbits.
\end{itemize}
By Lemma \ref{Lem:partition for any G}, we can assume that for all $k\in\mathbb{N}, \tilde{\sigma}\in \widetilde{K}(k)_l$, there exists a fundamental domain $\Omega$ of $\pi(\tilde{\sigma})$ with a bi-Lipschitz homeomorphism $\pi\llcorner\Omega$ satisfying 
\begin{equation}\label{Eq:no cut point in inter cube}
	\mbox{$G_p\supset G_{(\pi\llcorner\Omega)^{-1}(F_1(p_\sigma))}$, $\forall p\in\Omega$, }
\end{equation}
 and every orbit in $\mathring{\tilde{\sigma}}:=\pi^{-1}(F_1(\mathring{\sigma} ))$ is principal. 
 

\subsection{$G$-Morse functions and the $G$-equivariant deforming map}
Since Morse functions are often used to construct sweepouts, we show the following two lemmas about $G$-equivariant Morse functions. 
The first one comes from \cite{liu2021existence}:
\begin{lemma}\label{Lem:Morse sweepout}
	If $f:M\rightarrow [0,1]$ is a $G$-equivariant Morse function in the sense of \cite{wasserman1969equivariant}. 
	Then $\Phi(t)=f^{-1}(t)$ is a $(G,1)$-sweepout.
\end{lemma}

By \cite{wasserman1969equivariant}, $G$-equivariant Morse functions are dense in the space of $G$-equivariant smooth functions. 
Additionally, $\tilde{c}(k)$ contains only a finite number of orbits. 
Therefore, we have the following lemma similar to \cite[Lemma 5.3]{marques2017existence}:

\begin{lemma}\label{Lem:Morse func}
	For every $k\in\N$, there exists a $G$-equivariant Morse function $f:M\rightarrow \mathbb{R}$ in the sense of \cite{wasserman1969equivariant} such that 
	\begin{itemize}
		\item[(i)]  $f^{-1}(t)\cap \tilde{c}(k)$ contains at most one orbit for all $t\in \mathbb{R}$;
		\item[(ii)] no critical points of $f$ is contained in $\tilde{c}(k)$.
	\end{itemize}
\end{lemma}

\begin{proof}
	Let $\{U_i\}_{i=1}^N$ be open $G$-sets covering $M$, and $\{\phi_i\}_{i=1}^N$ be the partition of unity associated to $\{U_i\}_{i=1}^N$. 
	Without loss of generality, we can assume $\phi_i$ is $G$-invariant for each $i$ by considering $\int_G \phi_i(g\cdot x)~d\mu(g)$. 
	Since $\tilde{c}(k)=\{G\cdot p_{\tilde{\sigma}_j}\}_{j=1}^{3^{kl}}$ contains only a finite number of orbits, we can further require there exists only one $U_{i_j}$ covering $G\cdot p_{\tilde{\sigma}_j}$, for every orbit $G\cdot p_{\tilde{\sigma}_j}\in \tilde{c}(k)$. 
	Let $f_j:U_{i_j}\to \R^+$ be a smooth non-negative $G$-invariant function so that $f_j( p_{\tilde{\sigma}_j} )=1$, and $G\cdot p_{\tilde{\sigma}_j}$ contains no critical points of $f_j$. 
	(For example, let $h_j:=\eta_j\circ (\dist_M(\cdot, G\cdot q_j))^2$ and $f_j=h_j/h_j(p_{\tilde{\sigma}_j})$, where $\eta_j$ is a cut-off function and $G\cdot q_j$ is close to $G\cdot p_{\tilde{\sigma}_j}$.) 
	Define then $\tilde{f} := \sum_{j=1}^{3^{kl}} j\cdot f_j\cdot \phi_j$, which is a $G$-invariant smooth function on $M$ satisfying (i)(ii). 
	By the denseness of $G$-equivariant Morse functions (\cite{wasserman1969equivariant}), we can find $G$-equivariant Morse function $f$ close to $\tilde{f}$ as required.
\end{proof}

In the following proposition, we construct a deforming map compressing the complement of a neighborhood of $\tilde{c}(k)$ in $M$ into the skeleton $\pi^{-1}\circ F_1(K(k)_{l-1})$. 
\begin{proposition}\label{Prop:bend}
	 There exist constants $C_1>0$, $\epsilon_0>0$ depending only on $M$ and $G$, such that for any $k\in\mathbb{N}$ and $0<\epsilon\leq\epsilon_0$, we can find a $G$-equivariant continuous map $F:M\rightarrow M$ satisfying:
	\begin{itemize}
		\item[(i)] $F$ is $G$-homotopic to the identity relative to the skeleton $\pi^{-1}\circ F_1(K(k)_{l-1})$;
		\item[(ii)] $F(M\setminus B^G_{\epsilon 3^{-k}}(\tilde{c}(k)))\subset \pi^{-1}\circ F_1(K(k)_{l-1}) =\bigcup_{\tilde{\sigma}\in \widetilde{K}(k)_l } \tilde{\sigma}\setminus\mathring{\tilde{\sigma}}$ compressing $M\setminus B^G_{\epsilon 3^{-k}}(\tilde{c}(k))$ into the skeleton, and $F(B^G_{\epsilon 3^{-k}}(\tilde{c}(k))) \subset \bigcup_{\tilde{\sigma}\in \widetilde{K}(k)_l } \mathring{\tilde{\sigma}}$;
		\item[(iii)] the induced map $F'=\pi\circ F\circ \pi^{-1} :M/G \to M/G$ is Lipschitz continuous with $\big| D(F') \big|\leq C_1\epsilon^{-1}$. 
	\end{itemize}
\end{proposition}
\begin{proof}
	Let $C_1,\epsilon_0>0$ be given as in \cite[Proposition 5.4]{marques2017existence} for $F_1:K\to M/G$. 
	For every $\tilde{\sigma}\in  \widetilde{K}(k)_l$ and $0<\epsilon<\epsilon_0$, let $F_{\tilde{\sigma}}'$ be defined as in the proof of \cite[Proposition 5.4]{marques2017existence}, which satisfies: 
	\begin{itemize}
		\item[$\bullet$] $F_{\tilde{\sigma}}' (F_1(\sigma)\setminus \pi(B^G_{\epsilon 3^{-k}}(p_{\tilde{\sigma}}))) \subset F_1(\partial \sigma)$, and $F_{\tilde{\sigma}}'(\pi(B^G_{\epsilon 3^{-k}}(p_{\tilde{\sigma}}))) \subset F_1(\mathring{\sigma})$;
		\item[$\bullet$] $\big| D(F_{\tilde{\sigma}}')\big| \leq  C_1\cdot \epsilon^{-1}$;
		\item[$\bullet$] $F_{\tilde{\sigma}}'([p])=[p]$ for all $[p]\in F_1(\partial \sigma)$;
		\item[$\bullet$] $F_{\tilde{\sigma}}'$ is homotopic to the identity relative to $F_1(\partial\sigma)$.
	\end{itemize}
	Let $\Omega$ be given as in (\ref{Eq:no cut point in inter cube}). 
	Then, we define the $G$-map $F_{\tilde{\sigma}}: \tilde{\sigma}\to \tilde{\sigma}$ as $F_{\tilde{\sigma}}(q):= g\cdot (\pi\llcorner\Omega )^{-1}\circ F_{\tilde{\sigma}}' \circ \pi(p) $, if $q=g\cdot p\in \tilde{\sigma}$ for some $p\in \Omega$, $g\in G$. 
	We claim this map $F_{\tilde{\sigma}}$ is well defined. 
	Indeed, $F_{\tilde{\sigma}}={\rm id}$ on $\tilde{\sigma}\setminus\mathring{\tilde{\sigma}}$. 
	Suppose $p\in \Omega\cap \mathring{\tilde{\sigma}}$, and there is another $g'\in G$ with $g'\cdot p=g\cdot p$. 
	Thus, we have $g^{-1}g'\in G_p$. 
	By (\ref{Eq:no cut point in inter cube}), we conclude that $G_p= G_y$ and $G_y\subset G_x$ for all $y\in \Omega\cap \mathring{\tilde{\sigma}}$, $x\in \Omega\cap (\tilde{\sigma}\setminus\mathring{\tilde{\sigma}})$. 
	Hence, we have $g^{-1}g'\in G_{p} \subset G_{F_{\tilde{\sigma}}(p)}$ implying $g\cdot F_{\tilde{\sigma}}(p)=g'\cdot F_{\tilde{\sigma}}(p)$. 
	
	Since the homotopy between $F_{\tilde{\sigma}}'$ and $id$ is relative to $F_1(\partial\sigma)$, a similar construction gives the $G$-equivariant homotopy map from $F_{\tilde{\sigma}}$ to $id$ relative to $\tilde{\sigma}\setminus\mathring{\tilde{\sigma}}$. 
	Finally, define $F:M\rightarrow M$ by $F(p)=F_{\tilde{\sigma}}(p)$ if $p\in\tilde{\sigma} $. 
	The map $F$ is well defined and satisfies all the requirements.
\end{proof}

\subsection{Upper bounds of $(G,p)$-width}

\begin{theorem}\label{Thm-upper-bound}
	Let $G$ be a compact Lie group acting as isometries on a closed manifold $M^{n+1}$ with ${\rm Cohom}(G)=l\geq 3$. 
	Then there exists a constant $C=C(M,G)>0$, such that:
	$$ \omega_p^G(M)\leq C p^{\frac{1}{l}}, $$
	where $p\in \mathbb{N}$. 
\end{theorem}

When $G$ is trivial, the cohomogeneity $l$ of $G$ is ${\rm dim}(M)=n+1$. Hence, our upper bounds for $(G,p)$-width are coincide with the conclusions of Gromov\cite{gromov1988dimension} and Guth\cite{guth2009minimax}. 

\begin{proof}
	For any $p\in\mathbb{N}$ (sufficiently large), let $k\in \mathbb{N}$ such that $3^k\leq p^{\frac{1}{l}}\leq 3^{k+1}$. 
	
	Let $f:M\rightarrow \mathbb{R}$ be the $G$-equivariant Morse function defined in Lemma \ref{Lem:Morse func}. 
	By Lemma \ref{Lem:Morse sweepout}, the level sets of $f$ form a $(G,1)$-sweepout of $M$. 
	Now for each $a=(a_0,\dots,a_p)\in S^{p-1}\subset \mathbb{R}^{p}$, we define a polynomial $P_a(t):=\sum_{i=0}^p a_it^i$ and a map:
	$$\hat{\Psi} (a) := \partial\big\{p\in M : P_a(f(p))<0\big\}\in \mathcal{Z}^G_n(M;\mathbb{Z}_2).$$
	By $\mathbb{Z}_2$ coefficients, we have $\hat{\Psi} (a)=\hat{\Psi} (-a)$, and $\hat{\Psi}$ induces a map $\Psi: \mathbb{RP}^p\rightarrow \mathcal{Z}^G_n(M;\mathbb{Z}_2)$. 
	Similar to the proof of \cite[Theorem 5.1]{marques2017existence}, we can show that $\Psi$ is continuous in the flat topology and is a $(G,p)$-sweepout by Lemma \ref{Lem:Morse sweepout}.

	By Lemma \ref{Lem:Morse func} (i), we can choose $\epsilon$ small enough such that 
	\begin{equation}\label{Eq-no-intersect-around-center}
		f(B^G_{\epsilon 3^{-k}}(p_{\tilde{\sigma}_1})) \cap  f(B^G_{\epsilon 3^{-k}}(p_{\tilde{\sigma}_2})) = \emptyset ,
	\end{equation}
	for any $G\cdot p_{\tilde{\sigma}_1}\neq G\cdot p_{\tilde{\sigma}_2}\in \tilde{c}(k)$. 
	In what follows, we will denote $C$ as varying constants that depend only on $M$ and $G$ for simplicity. 
	
	Fix any $G\cdot p_{\tilde{\sigma}}\in \tilde{c}(k)$. 
	As we mentioned after (\ref{Eq:no cut point in inter cube}), 
	the dimension of $G\cdot p_{\tilde{\sigma}}$ is $n+1-l$. 
	For simplicity, let us denote $\Sigma_t:= {\rm spt}(f^{-1}(t)\llcorner B^G_{\epsilon 3^{-k}}(p_{\tilde{\sigma}}) ) $ to be the support of the restricted cycle $f^{-1}(t)\llcorner B^G_{\epsilon 3^{-k}}(p_{\tilde{\sigma}}) $. 
	Thus, by the co-area formula, we have 
	\begin{eqnarray}\label{Eq-tube-formula}
		{\bf M}(f^{-1}(t)\llcorner B^G_{\epsilon 3^{-k}}(p_{\tilde{\sigma}}) ) 
		&=&  \mathcal{H}^{n}(\Sigma_t)\nonumber
		\\
		&= & \int_{\pi(\Sigma_t )} \mH^{n+1-l}(\pi^{-1}([q])) d\mH^{l-1}([q])
		\\
		&\leq & C \mathcal{H}^{l-1}(\pi(\Sigma_t)),   \nonumber
	\end{eqnarray}
	where $C > \sup_{p\in M}\mathcal{H}^{n+1-l}(G\cdot p)$. 
	Since 
	$G\cdot p_{\tilde{\sigma}}$ contains no critical points of $f$ (Lemma \ref{Lem:Morse func}(ii)), we can choose $\epsilon>0$ even smaller so that 
	\begin{equation}\label{Eq-mass-on-slice}
		\mathcal{H}^{l-1}(\pi(\Sigma_t))  = \mathcal{H}^{l-1} \big( \pi(\spt(f^{-1}(t)))\cap B_{\epsilon 3^{-k}}(\pi(p_{\tilde{\sigma}})) \big)  \leq  2\omega_{l-1} (\epsilon 3^{-k})^{l-1}.
	\end{equation}
	Hence, we have
	\begin{equation}\label{Eq-mass-around-center}
		{\bf M}(f^{-1}(t)\llcorner B^G_{\epsilon 3^{-k}}(p_{\tilde{\sigma}})) \leq C (\epsilon 3^{-k})^{l-1}.
	\end{equation}
	For such $\epsilon$, take the map $F$ in Proposition \ref{Prop:bend} with the induced map $F'=\pi\circ F\circ\pi^{-1}:M/G\to M/G$, and define 
	$$\Phi(\theta):=F_\# \Psi(\theta):\mathbb{RP}^p\to \mathcal{Z}^G_n(M;\mathbb{Z}_2),$$ 
	which is also a $(G,p)$-sweepout since $F$ is $G$-homotopic to the identity. 
	
	 Now by Proposition \ref{Prop:bend} (ii)(iii) and (\ref{Eq-tube-formula})(\ref{Eq-mass-on-slice}), we have
	 \begin{eqnarray}\label{Eq-mass-compute-by-tube}
	 	 {\bf M}(F_\#(f^{-1}(t)\llcorner B^G_{\epsilon 3^{-k}}(p_{\tilde{\sigma}})))
	 	 &= & \mathcal{H}^n(F(\Sigma_{t} )) 
	 	 ~\leq ~ C~  \mathcal{H}^{l-1}(\pi(F(\Sigma_{t} ))) \nonumber
	 	 \\
	 	 &\leq &  C~ \big |D(F')\big|^{l-1}  \mathcal{H}^{l-1}(\pi(\Sigma_{t})) 
	 	 \\
	 	 &\leq & C~  \big |D(F')\big|^{l-1}   (\epsilon 3^{-k})^{l-1} \leq  C~ (3^{-k})^{l-1}.\nonumber
	 \end{eqnarray}
	 Furthermore, by (\ref{Eq-no-intersect-around-center}), 
	$$ {\bf M}(F_\#(f^{-1}(t)\llcorner B^G_{\epsilon 3^{-k}}(\tilde{c}(k))) \leq   C \cdot 3^{-k(l-1)}. $$
	 For any $\theta\in\mathbb{RP}^p$, the $G$-cycle $\Psi(\theta)$ consists at most $p$ level sets of $f$. 
	 Thus, we have 
	 \begin{equation}\label{Eq-mass-near-center}
	 	{\bf M}(F_\#(\Psi(\theta)\llcorner B^G_{\epsilon 3^{-k}}(\tilde{c}(k))) \leq   p\cdot C 3^{-k(l-1)}.
	 \end{equation}
	 
	 Set $B:=M\setminus B^G_{\epsilon 3^{-k}}(\tilde{c}(k))$. 
	 Then ${\rm spt}(F_\#(\Psi(\theta)\llcorner B))\subset \pi^{-1}\circ F_1(K(k)_{l-1}) = \bigcup \tilde{\sigma}\setminus \mathring{\tilde{\sigma}} $ (Proposition \ref{Prop:bend} (ii)). 
	 By $\mathbb{Z}_2$ coefficients, the multiplicity 
	 is at most one implying
	 \begin{equation*}
	 	{\bf M}(F_\#(\Psi(\theta)\llcorner B)) \leq  \mathcal{H}^n (\pi^{-1}\circ F_1(K(k)_{l-1})) \leq \mathcal{H}^{n}\big(~\bigcup \tilde{\sigma}\setminus \mathring{\tilde{\sigma}} ~\big).
	 \end{equation*}
	 Additionally, by the co-area formula (on each stratum), we have 
	 \begin{eqnarray*}
	 	\mathcal{H}^{n}(\tilde{\sigma}\setminus \mathring{\tilde{\sigma}} ) ~\leq ~ C~ \mathcal{H}^{l-1}(\pi(\tilde{\sigma} \setminus \mathring{\tilde{\sigma}} ) ) ~\leq ~  C~ |DF_1|^{l-1}\cdot \mathcal{H}^{l-1}(\partial I^l_{3^{-k}}),
	 \end{eqnarray*}
	 where $F_1$ is the bi-Lipschitz homeomorphism from $K$ to $M/G$, and $I^l_{3^{-k}}$ is a $l$-dimensional cube with sides of $3^{-k}$.
	As a result, 
	 \begin{eqnarray}\label{Eq-mass-out-center}
	 	{\bf M}(F_\#(\Psi(\theta)\llcorner B)) &\leq &  \mathcal{H}^{n}\big(~\bigcup \partial\tilde{\sigma}~\big) \nonumber
	 	\\
	 	&\leq & C_2 3^{kl}\cdot C\cdot |DF_1|^{l-1}\cdot \mathcal{H}^{l-1}(\partial I^l_{3^{-k}}) 
	 	\\
	 	&\leq & C_2 3^{kl} \cdot  C 3^{-k(l-1)} \leq  C 3^k, \nonumber
	 \end{eqnarray}
	where $C_2$ is the number of $l$-cells in $K$. 
	Combining (\ref{Eq-mass-near-center}) with (\ref{Eq-mass-out-center}), we have 
	 $$ \omega^G_p(M)\leq \sup {\bf M}(\Phi(\theta))\leq C\cdot (p3^{-k(l-1)}+3^k)\leq C(M,G)p^{\frac{1}{l}}$$
	 since $3^k\leq p^{\frac{1}{l}}\leq 3^{k+1}$.
\end{proof}

\subsection{Applications}
Since the Lie group $G$ is not assumed to be connected, we also need the following $G$-invariant analogs of the usual notions of connectivity. 
\begin{definition}
	Suppose $\Sigma$ is an embedded $G$-hypersurface, and $\{\Sigma_i\}_{i=1}^k$ are connected components of $\Sigma$. 
	We say $\Sigma$ is {\it $G$-connected}, if for any $i,j\in\{1,\dots,k\}$, there exists $g\in G$ such that $\Sigma_i=g\cdot\Sigma_j$. 
	%
\end{definition}

Note our $G$-equivariant Min-max Theorem \ref{Thm:G-minmax} is built under continuous settings, we can use the definition of $(G,p)$-width derectly to show the following proposition, which is a $G$-invariant continuous version of \cite[Proposition 4.8]{marques2017existence}. 

\begin{proposition}\label{Prop:infinte}
	Suppose $2\leq n\leq 6$, ${\rm Cohom}(G)\geq 3$, and $M\setminus M^{reg}$ is a smoothly embedded submanifold of $M$ with dimension at most $n-2$. 
	If there exists $p\in\N$ so that for every $(G,p)$-admissible $X$ and every class of continuous $(G,p)$-sweepout ${\bf \Pi}\in \big[ X, \Z_n^G(M;\mF;\mZ_2)  \big]$ we have
	$$\omega_p^G(M)<{\bf L}({\bf \Pi}), $$
	then there exist infinitely many distinct closed smooth embedded $G$-invariant minimal hypersurfaces with uniformly bounded area. 
\end{proposition}
\begin{proof}
	By Definition \ref{Def:G.p.width} and \ref{Def:width continuous}, there exists a sequence of $(G,p)$-admissible sets $\{X_k\}_{k=1}^\infty$, and a sequence of classes of $(G,p)$-sweepouts ${\bf \Pi}_k\in \big[ X_k, \Z_n^G(M;\mF;\mZ_2)  \big]$, so that ${\bf L}({\bf \Pi}_1)>\cdots>{\bf L}({\bf \Pi}_k)>{\bf L}({\bf \Pi}_{k+1})>\cdots$ and $\lim_{k\to\infty}{\bf L}({\bf \Pi}_k) = \omega_p^G(M)$. 
	Applying Theorem \ref{Thm:G-minmax} to each ${\bf \Pi}_k$ gives this proposition. 
\end{proof}

We also have the following theorem, which is parallel to \cite[Theorem 6.1]{marques2017existence}. 
\begin{theorem}\label{Thm:equality case}
	Suppose $2\leq n\leq 6$, ${\rm Cohom}(G)\geq 3$, and $M\setminus M^{reg}$ is a smoothly embedded submanifold of $M$ with dimension at most $n-2$. 
	If $\omega_p^G(M)=\omega_{p+1}^G(M)$ for some $p\in\N$, then there exist infinitely many distinct closed smooth embedded $G$-invariant minimal hypersurfaces. 
\end{theorem}
\begin{proof}
	First, suppose this theorem is not true. 
	Then, by Proposition \ref{Prop:infinte}, there is a class of $(G, p+1)$-sweepout ${\bf \Pi}\in \big[ X, \Z_n^G(M;\mF;\mZ_2)  \big]$ so that $\omega_{p+1}^G(M)={\bf L}({\bf \Pi})$. 
	Let $S=\{\varphi_i\}_{i\in\N}$ be the sequence of discrete mappings constructed in the pull-tight procedure (Proposition \ref{Prop:pulltight}). 
	Therefore, the Almgren $G$-extension $\Phi_i$ of $\varphi_i$ is also a $(G, p+1)$-sweepout, and $\Phi_i\in {\bf \Pi}$ for every $i$ sufficiently large (see Proposition \ref{Prop:pulltight} (ii) and Remark \ref{Rem-homo-sweep-is-sweep}). 
	Now the proof of \cite[Theorem 6.1]{marques2017existence} would carry over by using $\{\varphi_i\}_{i\in\N}$ and ${\bf \Pi}$ in place of $\{\phi_i\}_{i\in\N}$ and $\Pi$ in \cite[Page 602]{marques2017existence}. 
\end{proof}

Finally, the arguments in the proof of \cite[Section 7]{marques2017existence} would carry over under the $G$-invariant setting. 
Thus we get the following dichotomy result for $G$-invariant min-max minimal hypersurfaces as a generalization of \cite[Theorem 1.1]{marques2017existence}:
\begin{theorem}\label{Thm:dichotomy}
	Let $M^{n+1}$ be an orientable closed manifold with a compact Lie group $G$ acting as isometries on $M$. 
	Suppose $2\leq n\leq 6$ and ${\rm Cohom}(G)=l\geq 3$. 
	If $M\setminus M^{reg}$ forms a smoothly embedded submanifold of $M$ with dimension at most $n-2$. 
	Then: 
	\begin{itemize}
		\item[(i)]  either there exists a disjoint collection $\{\Sigma_1,\dots,\Sigma_l \}$ of $l$ closed smooth embedded $G$-connected $G$-invariant minimal hypersurfaces;
		\item[(ii)] or there exist infinitely many closed smooth embedded $G$-connected $G$-invariant minimal hypersurfaces.
		\end{itemize}
\end{theorem}

Since manifolds with positive Ricci curvature satisfy the Frankel property \cite{frankel1966fundamental} (i.e. any two closed smooth embedded minimal hypersurfaces intersect each other), we have the following corollary as a generalization of \cite[Corollary 1.5]{marques2017existence}:
\begin{corollary}\label{Cor:exist infinite ricci positiv}
	Let $M^{n+1}$ be an orientable closed manifold with a compact Lie group $G$ acting as isometries on $M$. 
	Suppose $2\leq n\leq 6$, ${\rm Cohom}(G)=l\geq 3$, and the Ricci curvature ${\rm Ric}_M>0$. 
	If $M\setminus M^{reg}$ forms a smoothly embedded submanifold of $M$ with dimension at most $n-2$. 
	Then there exist infinitely many closed smooth embedded $G$-connected $G$-invariant minimal hypersurfaces.
\end{corollary}


\section{Lower Bounds of $(G,p)$-Width}\label{Sec-lower-bound-witdh}

In this section, we generalize the proof of Marques and Neves in \cite[Section 8]{marques2017existence} to show the lower bounds of $(G,p)$-width. 
We start with some local constructions.

Recall $M^{reg}$ is the union of all principal orbits, which is an open and dense submanifold in $M$ 
(see \cite[Section 3.5]{wall2016differential}). 
It's not hard to find that the volume function $\mathcal{H}^{n+1-l}(G\cdot q)$ and the injectivity radius function ${\rm Inj}(G\cdot q)$ are both continuous on $M^{reg}$. 
Moreover, there exits a Riemannian metric $g_{_{M/G}}$ on $M^{reg}/G$ so that $\pi: M^{reg}\to M^{reg}/G$ is a Riemannian submersion. 
Hence, we can fix any point $q_0\in M^{reg}$ and choose a positive number $r_0>0$ sufficiently small such that: 
\begin{itemize}
	\item[(i)] $\overline{B}^G_{4r_0}(q_0)\subset M^{reg}$;
	\item[(ii)] there exists $C_1>0 $ so that $ \frac{1}{C_1}\leq \mathcal{H}^{n+1-l}(G\cdot q)\leq C_1$, for all $q\in \overline{B}^G_{4r_0}(q_0)$;
	\item[(iii)] ${\rm Inj}(G\cdot q)>2r_0$, for all $q\in \overline{B}^G_{r_0}(q_0)$;
	\item[(iv)] there exists $C_2>2 $ so that for all $q\in \overline{B}^G_{r_0}(q_0)$ the exponential map $\exp_{[q]}$ in $M^{reg}/G$ satisfies $\frac{1}{C_2}\leq\frac{1}{2}\leq |d \exp_{[q]}| \leq 2\leq C_2$ in $B_{2r_0}([q])=\pi(B^G_{2r_0}(q))$.
\end{itemize}
Since $q_0\in M^{reg}$ is arbitrarily fixed, we regard $r_0,C_1,C_2>0$ as constants depending only on $M$ and $G$. 
Now we can apply the co-area formula to see 
\begin{eqnarray*}
	{\rm vol}(B^G_r(q)) &=& \int_{B_r([q])} \mH^{n+1-l}(\pi^{-1}([y])) d\mH^{l}([y])
	\\
	&=& \int_{\mathbb{B}_r(0)}  \mH^{n+1-l}(\pi^{-1}\circ \exp_{[q]}(v)) J_{\exp_{[q]}}(v) d\mH^l(v),
\end{eqnarray*} 
for all $q\in  \overline{B}^G_{r_0}(q_0) $ and $r\in (0,2r_0]$. 
By (ii)(iv), we get the following inequality: 
\begin{eqnarray}\label{Eq-volum-bound}
	 \frac{C_3}{C_1C_2^{l}}\cdot r^{l} \leq {\rm vol}(B^G_r(q)) \leq  C_1 C_2^l C_3\cdot r^{l},
\end{eqnarray}
where $C_3>0$ is a constant depending only on $l$.  

Furthermore, given $q\in \overline{B}_{r_0}^G(q_0),~r\in(0,r_0]$, let $F:B_{r}^G(q)\rightarrow B_{r_0}^G(q)$ be a $G$-equivariant diffeomorphism defined by 
$ F(y):={\rm exp}^\perp_{G\cdot q}(z_y,\frac{r_0}{r}v_y )$, where $(z_y,v_y ) = ({\rm exp}^\perp_{G\cdot q})^{-1}(y).$ 
The induced map $F':B_r([q])\to B_{r_0}([q])$ of $F$ is $F'([y]):= \exp_{[q]}(\frac{r_0}{r} \exp_{[q]}^{-1}([y]) ) $. 
Suppose $T\in \mathcal{Z}_k^G(B_{r}^G(q),\partial B_{r}^G(q);\mathbb{Z}_2)$ is a $G$-invariant $k$-cycle for some $k\geq n+1-l$. 
Similar to (\ref{Eq-mass-compute-by-tube}), we have the following inequality by the co-area formula:
\begin{eqnarray}\label{Eq-mass-bound-under-push-2}
	{\bf M}(F_\#T)
	 	 &= & \int_{\pi(F(\Sigma))} \mH^{n+1-l}(\pi^{-1}([y])) d\mH^{k-n-1+l}([y])\nonumber
	 	 \\
	 	 &\leq &  C_1 \mathcal{H}^{k-n-1+l}(F'(\pi(\Sigma ))) \nonumber 
	 	 \\
	 	 &\leq & \big |D(F')\big|^{k-n-1+l}  C_1  \mathcal{H}^{k-n-1+l}(\pi(\Sigma)) 
	 	 \\
	 	 &\leq & \big |D(F')\big|^{k-n-1+l} C_1^2  \int_{\pi(\Sigma)} \mathcal{H}^{n+1-l}(\pi^{-1}([y])) d\mathcal{H}^{k-n-1+l}([y])\nonumber
	 	 \\
	 	 &= & C(M,G) \cdot\big |D(F')\big|^{k-1-n+l} {\bf M}(T),\nonumber
\end{eqnarray}
where $\Sigma=\spt(T)$. 
Therefore, by the definition of the flat metric $\mathcal{F}$, we can use the inequality above to show 
\begin{equation}\label{Eq-flat-bound-of-push}
	\mathcal{F}(F_\#T)\leq C(M,G) \cdot\big |D(F')\big|^{k-n+l} \mathcal{F}(T).
\end{equation}

Using the inequality above, we can estimate the local $G$-isoperimetric constant. 
Firstly, let us review some definitions. 
Suppose $q\in B^G_{r_0}(q_0)$. 
By Lemma \ref{Lem:isoperimetric} and \cite[Lemma 3.1]{marques2017existence}, we can find a positive constant $\nu_r(q) := \nu_{B^G_r(q),\partial B^G_r(q)}$ such that for all $T\in\mathcal{Z}^G_n(B^G_r(q),\partial B^G_r(q);\mathbb{Z}_2)$ with $\mathcal{F}(T)<\nu_r(q)$, 
there exists a unique $Q\in {\bf I}^G_{n+1}( B^G_r(q);\mathbb{Z}_2)$ satisfying 
$\partial Q-T\in {\bf I}^G_{n}(\partial  B^G_r(q);\mathbb{Z}_2) {\rm ~and~} {\bf M}(Q)\leq \nu_r(q).$

\begin{lemma}\label{Lem-scal-isoperimetric-const}
	Let $r_0=r_0(M,G)$ be defined as above. 
	Then there exists a positive constant $\alpha_0=\alpha_0(M,G)$ such that 
	\begin{equation}
		\nu_r(q) := \nu_{ B^G_r(q),\partial B^G_r(q)}>\alpha_0 r^{l},\quad l={\rm Cohom}(G),
	\end{equation}
	for any $q\in \overline{B}^G_{r_0}(q_0) $, $r\in(0,r_0]$.
\end{lemma}
\begin{proof}
	For any $G\cdot q\subset \overline{B}^G_{r_0}(q_0) $, we take $\alpha_0(q)>0$ satisfying 
	$$\nu_{r_0}(q)>4^{l} C \alpha_0(q)r_0^{l},$$
	where $C>0$ is the constant in (\ref{Eq-flat-bound-of-push}).
	Since the exponential map ${\rm exp}^\perp_{G\cdot q}$ provides a $G$-equivariant diffeomorphism on $B^G_{2r_0}(q)$, we can define $(z_y,v_y):=({\rm exp}^\perp_{G\cdot q})^{-1}(y)$ for all $y\in B^G_r(q)$, $r\in (0,r_0]$. 
	Take the $G$-equivariant diffeomorphism $F_r:B^G_r(q)\rightarrow B^G_{r_0}(q)$ given by $F_r(y) := {\rm exp}^\perp_{G\cdot q}(z_y,\frac{r_0}{r}v_y )$. 
	The induced map $F_r':B_r([q])\to B_{r_0}([q])$ of $F_r$ is $F_r'([y]):= \exp_{[q]}(\frac{r_0}{r} \exp_{[q]}^{-1}([y]) ) $. 
	Hence, we have $|D(F_r')|\leq 4\frac{r_0}{r}$ by (iv). 
	Then, for any 
	$$T\in\mathcal{Z}^G_n(B^G_r(p),\partial B^G_r(p);\mathbb{Z}_2) {\rm ~with~}\mathcal{F}(T)<\alpha_0(q)r^{l},$$
	we apply (\ref{Eq-mass-bound-under-push-2}) with $k=n$ to show the flat semi-norm of $F_{r\#}T$ can be bounded by 
	\begin{eqnarray*}
		\mathcal{F}(F_{r\#}T) ~\leq ~ C\cdot |D(F_r'\llcorner B_q)|^{l}\mathcal{F}(T)  ~\leq~   C~4^{l} \alpha_0(q)r_0^{l}   ~<~    \nu_{r_0}(q).
	\end{eqnarray*}
	By the definition of $\nu_{r_0}(q)$, there exists a unique $Q'\in {\bf I}^G_{n+1}(B^G_{r_0}(q);\mathbb{Z}_2)$ with 
	$$\partial Q'-F_{r\#}T\in {\bf I}^G_{n}(\partial B^G_{r_0}(q);\mathbb{Z}_2) {\rm ~and~} {\bf M}(Q')\leq \nu_{r_0}(q).$$
	Therefore, one can take $Q:=(F^{-1}_{r})_\#Q' \in {\bf I}^G_{n+1}(B^G_{r}(q);\mathbb{Z}_2)$ to be the unique $G$-invariant isoperimetric choice of $T$. 
	
	Thus we conclude that $\alpha_0(q)r^{l}<\nu_r(q)$ for $r\in(0,r_0]$. 
	The lemma then follows from the compactness of $\overline{B}^G_{r_0}(q_0)$. 
\end{proof}

The following lemma is a simple modification of the claim showed by Guth in the proof of \cite[Cup Product Theorem]{guth2009minimax}.

\begin{lemma}\label{Lem-disjoint-tube}
	Let $r_0>0$ be the constant defined in Lemma \ref{Lem-scal-isoperimetric-const} and $p\in\mathbb{N}$. 
	If numbers $0<r_p\leq\dots\leq r_1\leq r_0 $ satisfy 
	$$\sum_{i=1}^p r_i^l ~<~ \rho_0(M,G) := \frac{r_0^l}{C_1^2C_2^{2l} \cdot 2^l},$$
	then there exist $q_1,\dots,q_p\in B^G_{r_0}(q_0)\subset M^{reg}$ such that  
	$$ B^G_{r_i}(q_i)\cap B^G_{r_j}(q_j) = \emptyset,\quad {\rm for~} i\neq j\in\{1,\dots,p\} .$$
\end{lemma}
\begin{proof}
	By (\ref{Eq-volum-bound}), ${\rm vol}(B^G_{r_0}(q_0))\geq \frac{C_3}{C_1C_2^l}r_0^l$. 
	By the statement (i) above, we can choose $q_1\in B^G_{r_0}(q_0)\cap M^{reg}$. 
	Now we proceed inductively. 
	Suppose we have found $q_1,\dots,q_{j-1}\in B^G_{r_0}(q_0)\cap M^{reg}$ so that the $G$-tubes $B^G_{r_i}(q_i)$ are disjoint $(i=1,...,j-1)$. 
	Then, by (\ref{Eq-volum-bound}), the volume of $\cup_{i=1}^{j-1} B^G_{2r_i}(q_i)$ is bounded by 
	$$ \sum_{i=1}^{j-1}C_1C_2^lC_3\cdot (2r_i)^l < \frac{C_3}{C_1C_2^l}r_0^l\leq {\rm vol}(B^G_{r_0}(q_0)).$$
	Therefore, we can get an orbit $G\cdot q_j\in (B^G_{r_0}(q_0)\cap M^{reg})\setminus (\cup_{i=1}^{j-1} B^G_{2r_i}(q_i))$. 
	Since $r_j\leq r_i$ for any $i=1,\dots,j-1$, the $G$-tube $B^G_{r_j}(q_j)$ is disjoint with $\cup_{i=1}^{j-1} B^G_{r_i}(q_i)$. 
	Continuing to choose $G$-tubes in this way yields the lemma.
\end{proof}

Using Lemma \ref{Lem-scal-isoperimetric-const}, we have the following proposition which is an equivariant generalization of \cite[Proposition 8.2]{marques2017existence}. 

\begin{proposition}\label{Prop-local-mass-bound-of-sweepout}
	There exist positive constants $\alpha_0=\alpha_0(M,G)$ and $r_0=r_0(M,G)$ so that for any $(G,1)$-sweepout $\Phi:S^1\rightarrow \mathcal{Z}_n^G(M;\mathbb{Z}_2)$, we have 
	$$ \sup_{\theta\in S^1}{\bf M}(\Phi(\theta)\llcorner B^G_r(q) )\geq \alpha_0r^{n-{\rm dim}(G\cdot q)} =\alpha_0r^{l-1} $$
	for all $q\in \overline{B}^G_{r_0}(q_0)$ and $r\in(0,r_0]$.
\end{proposition}
\begin{proof}
	Let $\alpha_1=\alpha_0,r_1=r_0$, where $\alpha_0,r_0$ are the numbers in Lemma \ref{Lem-scal-isoperimetric-const}.
	Let $\rho=C_M$ be the constant in Lemma \ref{Lem:isoperimetric}. 
	For any $q\in \overline{B}^G_{r_0}(q_0), ~r\in (0,r_1]$, we choose $\delta$ small enough such that $(1+\frac{2}{r})\rho\delta<\alpha_1(\frac{r}{2})^{l}$. 
	Similar to Lemma \ref{Lem:isoperimetric}, we can adapt \cite[Proposition 1.22]{almgren1962homotopy} to give a $G$-invariant ${\bf M}$-isoperimetric choice lemma. 
	Moreover, by scaling considerations, we have a similar result for the ${\bf M}$-isoperimetric constant $\nu^{\bf M}_r(q)>\alpha_2(M,G) r^{l-1}$ as Lemma \ref{Lem-scal-isoperimetric-const}. 
	Combining the $G$-invariant ${\bf M}$-isoperimetric lemma with Lemma \ref{Lem-scal-isoperimetric-const}, the proof of \cite[Proposition 8.2]{marques2017existence} would carry over after adding `$G$-' in front of relevant objects.
\end{proof}

Finally, we can follow the same procedure in the proof of \cite[Theorem 8.1]{marques2017existence} with Lemma \ref{Lem-disjoint-tube} and Proposition \ref{Prop-local-mass-bound-of-sweepout} to give the lower bounds of $(G,p)$-width.

\begin{theorem}\label{Thm-lower-bound-of-width} 
	Let $G$ be a compact Lie group acting as isometries on a closed manifold $M^{n+1}$ with ${\rm Cohom}(G)=l\geq 3$. 
	There exists a positive constant $C=C(M,G)>0$ such that
	$$ \omega_p^G(M)\geq Cp^{\frac{1}{l}}, $$
	for any $p\in\mathbb{N}$. 
\end{theorem}
\begin{proof}
	By Lemma \ref{Lem-disjoint-tube}, there exists a positive constant $v=v(M,G)$ such that, for each $p\in \mathbb{N}$, one can find $p$ disjoint $G$-tubes $\{B^G_r(q_i)\}_{i=1}^p$ with $r=vp^{-\frac{1}{l}}$ and $q_i\in B^G_{r_0}(q_0)\subset M^{reg}$. 
	Let $\alpha_0$ be the constant in Proposition \ref{Prop-local-mass-bound-of-sweepout}. 
	
	For any $\Phi \in \mathcal{P}_p^G$ mapping $X$ into $\mathcal{Z}^G_n(M;{\bf M};\mathbb{Z}_2)$, we can follow the proof of \cite[Theorem 8.1, Claim 8.4]{marques2017existence} with Lusternik-Schnirelmann theory and Proposition \ref{Prop-local-mass-bound-of-sweepout} to show that there exists $x\in X$ satisfying 
	$$ {\bf M}(\Phi(x)\llcorner B^G_r(q_i) ) \geq \frac{\alpha_0}{6}r^{n-{\rm dim}(G\cdot q_i)} = \frac{\alpha_0}{6}r^{l-1} ,$$
	for $i=1,\dots,p$. 
	Hence, we have 
	$$ {\bf M}(\Phi(x))\geq \sum_{i=1}^p {\bf M}(\Phi(x)\llcorner B^G_r(q_i) ) \geq p \frac{\alpha_0}{6}r^{l-1}\geq \frac{\alpha_0}{6}v^{l-1}p^{\frac{1}{l}} =C p^{\frac{1}{l}},$$
	where $C>0$ is a constant depending only on $M$ and $G$. 
	Since $\Phi \in \mathcal{P}_p^G$ is arbitrary, the theorem follows from the definition of $(G,p)$-width.
\end{proof}

\appendix
	\addcontentsline{toc}{section}{Appendices}
	\renewcommand{\thesection}{\Alph{section}}
\section{Isometric Actions of Lie Groups}\label{Appendix-good-partition}
In this appendix, we collect some basic notations and facts on Lie group actions (see \cite{wall2016differential} for details). 
Let $G$ be a compact Lie group acting as isometries on a compact manifold $M$. 

For any closed subgroup $H$ of $G$, denote $(H)$ as the conjugate class of $H$ in $G$. 
Then we say $p\in M$ has $(H)$ orbit type if $(G_p)=(H)$, where $G_p:=\{g\in G : g\cdot p=p\}$ is the isotropy group of $p$. 
We also denote 
$M_{(H)}:=\{p\in M : (G_p)= (H)\}$
to be the union of points with $(H)$ orbit type, which is a disjoint union of smooth embedded submanifolds of $M$. 
Moreover, the projection $\pi: M_{(H)} \rightarrow X_{(H)}\subset M/G$ gives a Riemannian submersion. 
Since the number of orbit types is locally finite, there are only finite different orbit types in a compact manifold $M$. 
Additionally, the partition by orbit types is a {\em stratification} of $M$ as well as $M/G$. 
Besides, there exists a minimal conjugate class of isotropy group $(P)$ such that $M_{(P)}$ forms an open dense submanifold of $M$, which is called the {\em principal orbit type}. 
The {\em cohomogeneity} ${\rm Cohom}(G)$ of $G$ is defined as the co-dimension of a principal orbit.

Given a closed subgroup $H$ of $G$, a {\em slice} or {\em $H$-slice} is a smooth embedded submanifold (possibly with boundary) $S$ of $M$ satisfying 
\begin{itemize}
	\item[$\bullet$] $S$ is $H$-invariant; 
	\item[$\bullet$] $T_pM = T_pG\cdot p + T_pS$ for all $p\in S$; 
	\item[$\bullet$] if $p\in S$, $g\in G$ and $g\cdot p\in S$, then $g\in H$. 
\end{itemize}
For any $p\in M$, there exists a $G_p$-slice $S_p$ of $G\cdot p$ at $p$ so that $G\cdot S_p$ is an open $G$-neighborhood of $G\cdot p$ by the slice theorem \cite[Theorem 3.3.4, 3.3.5]{wall2016differential}. 
Indeed, since the Lie group $G$ acts by isometries, the slice $S_p$ can be given by the normal exponential map as $\exp^{\perp}_p(B_p)$, where $B_p = B\cap N_p(G\cdot p)$ and $B\subset N(G\cdot p)$ is open so that $\exp^{\perp}_{G\cdot p}\llcorner B$ is a $G$-equivariant diffeomorphism. 
One can also verify that $S_p/G_p \cong (G\cdot S_p)/G$ is an open neighborhood of $[p]=\pi(G\cdot p)$ in $M/G$. 
The slice $S_p$ indeed gives a bundle structure of a neighborhood of $G\cdot p$. 

In the paper \cite{verona1979triangulation}, Verona showed that there is a triangulation $\triangle$ on $M/G$ such that ${\rm the~points~in~Int}(s) {\rm ~have~the~same~orbit~type} $ for all $s\in\triangle$. 
Using this triangulation of orbit space $M/G$, Illman found in \cite{illman1983equivariant} that the barycentric subdivision of such triangulation satisfies
$(G_p)\subset (G_q){\rm,~where~}\pi(p)\in{\rm Int}(s),{\rm and~}\pi(q)\in\partial s {\rm,~for~all~}s\in\triangle.$ 
Indeed, the set $\{\pi(q)\in s: (G_{p_s})\subsetneq (G_q)\}$ is either empty or a face $s'$ of $s$ with $\dim(s')<\dim(s)$, where $G\cdot p_s$ is the center orbit of $\pi^{-1}(s)$. 
This helps Illman to build an equivariant triangulation on $M$. 
Specifically, let $\triangle_n$, $n\in\N$, be the standard $n$-simplex. 
We regard $\triangle_{n-1}$ as a face of $\triangle_n$ ($\triangle_{-1}=\emptyset$). 
Take any sequence $H_n\subseteq H_{n-1}\subseteq \cdots\subseteq H_0$ of closed subgroups of $G$. 
Then we define a relation $\backsim$ in $\triangle_n\times G$ as 
$$ (x_1,g_1)\backsim (x_2,g_2)~\iff~ g_1H_m = g_2H_m, ~x_1=x_2\in\triangle_m\setminus\triangle_{m-1}, ~0\leq m \leq n. $$
Denote the quotient space (with the quotient topology) and the natural projection map by:
\begin{equation}\label{Eq: equivariant simplex}
	\triangle_n(G; H_0,\dots,H_n) := \triangle_n\times G/\backsim, \quad P: \triangle_n\times G \to \triangle_n(G; H_0,\dots,H_n).  
\end{equation}
Denote $P((x,g)) = \langle x, g\rangle$ for $(x,g)\in\triangle_n\times G$. 
Then the action of $G$ on $ \triangle_n(G; H_0,\dots,H_n)$ can be defined as $g_1\cdot \langle x, g \rangle = \langle x, g_1\cdot g \rangle$, $\forall g_1\in G$. 
Therefore, $\triangle_n(G; H_0,\dots,H_n)$ is a $G$-space with the orbit space $\triangle_n$, which is known as the {\em $G$-equivariant $n$-simplex of type $(H_0,\dots,H_n)$}. 
One can easily verify that the isotropy group of $\langle x, g\rangle\in \triangle_n(G; H_0,\dots,H_n)$ is $G_{\langle x, g\rangle} = gH_mg^{-1}$, for any $x\in \triangle_{m}\setminus \triangle_{m-1}$ and $g\in G$. 
The {\em equivariant triangulation} of $M$ consists of a triangulation $f: K\to M/G$, so that for any $n$-simplex $t$ of $K$ there exist closed subgroups $H_n\subseteq H_{n-1} \subseteq \cdots\subseteq H_0$ of $G$ and an equivariant homeomorphism 
\begin{equation}\label{Eq: equivariant triangulation}
	\alpha : \triangle_n(G; H_0,\dots,H_n)\to \pi^{-1}(f(t)),
\end{equation}
which induces a {\em linear} homeomorphism from $\triangle_n $ to $t$ in the orbit space level. 
For simplicity, we sometimes use $\triangle:=\{s:=f(t): t\in K\}$ as the triangulation of $M/G$, and use $\widetilde{\triangle}:=\{\pi^{-1}(f(t)) : t\in K\}$ as the equivariant triangulation of $M$. 
Moreover, by \cite[Section 8.12]{eells2001harmonic}, the triangulation $f:K\to M/G$ can be bi-Lipschitz (Riemannian polyhedra). 

We also mention that Illman's equivariant triangulation can be barycentric subdivided, which keeps being $G$-equivariant. 
Thus we always assume ${\rm diam}(s)<{\rm Inj}(M)/2 $ for all $s\in \triangle$. 
Additionally, $\triangle$ gives a triangulation on every stratum of orbit type, i.e. $\triangle_{(H)}:= \{s\in\triangle: s\subset \Clos(X_{(H)}) = \Clos(\pi(M_{(H)})) \}$ is a triangulation on $X_{(H)}$ for every orbit type $(H)$. 
 
\begin{lemma}\label{Lem:partition for any G}
	Suppose $G$ is a compact Lie group acting as isometries on a compact Riemannian manifold $M$. 
	Let $\triangle$ be a Lipschitz triangulation of $M/G$ with $\sup_{s\in\triangle}{\rm diam}(s)<{\rm Inj}(M)/2$, which generates an equivariant triangulation $\widetilde{\triangle}:=\{\tilde{s}= \pi^{-1}(s): s\in \triangle \}$ of $M$. 
	Denote $G\cdot p_{\tilde{s}}$ to be the center orbit of $\tilde{s}\in\widetilde{\triangle}$. 
	Then $\triangle$ can be chosen with the following properties:
	\begin{itemize}
		\item[(i)] $(G_{p_{\tilde{s}}}) \subset (G_{p})$ for all $p \in \tilde{s}$, and 
		$(G_{p_{\tilde{s}}}) = (G_{p})$ for all $p \in \pi^{-1}({\rm Int}(s))$;
		\item[(ii)] the set $\{\pi(q)\in s: (G_{p_{\tilde{s}}})\subsetneq (G_q)\}$ is either empty or a face $s'$ of $s$ with $\dim(s')<\dim(s)$;
		\item[(iii)] $\widetilde{\triangle}_{(H)} = \pi^{-1}(\triangle_{(H)} )$ is an equivariant triangulation on $\Clos(M_{(H)})$; 
		\item[(iv)] for any orbit type $(H)$ and $\tilde{s}\in \widetilde{\triangle}_{(H)}$, 
			there is a fundamental domain $\Omega$ of $s$ so that $\pi\llcorner\Omega:\Omega\to s$ is a bi-Lipschitz homeomorphism, and $G_p\supset G_{(\pi\llcorner\Omega)^{-1}([p_{\tilde{s}}])}$, $\forall p\in\Omega$. 
	\end{itemize} 
\end{lemma}
\begin{proof}
	(i), (ii) and (iii) come from \cite{verona1979triangulation}\cite{illman1983equivariant}. 
	Let $f: K\to M/G$ be the triangulation of $M/G$, which generates $\triangle=\{s=f(t):t\in K\}$ and $\widetilde{\triangle}:=\{\tilde{s}=\pi^{-1}(f(t)) : t\in K\}$ as the equivariant triangulation of $M$. 
	Given any orbit type $(H)$, $M_{(H)}:=\{p\in M: (G_p)=(H)\}$ as well as $\Clos(M_{(H)})$ is a smooth $G$-invariant submanifold of $M$. 
	By (iii), $\widetilde{\triangle}_{(H)} = \pi^{-1}(\triangle_{(H)} )$ is an equivariant triangulation of $\Clos(M_{(H)})$. 
	For any $k$-cell $s=f(t)\in \triangle_{(H)}$, it follows from the definition in (\ref{Eq: equivariant triangulation}) that there exist closed subgroups $H_k\subseteq \cdots\subseteq H_0$ of $G$ and an equivariant homeomorphism $\alpha : \triangle_k(G; H_0,\dots, H_k)\to \pi^{-1}(s)$ so that $\alpha$ induces a linear homeomorphism from $\triangle_k$ to $t$ in the orbit space level. 
	Let $P$ be defined as in (\ref{Eq: equivariant simplex}). 
	Then we define $\Omega := \alpha(P(\triangle_k\times\{e\} )) = \alpha(\{\langle x, e\rangle: x\in\triangle_k \})$. 
	
	One can easily verify that $P : \triangle_k\times G\to \triangle_k(G; H_0,\dots, H_k)$ is a closed map (\cite[Lemma 2.1]{illman2000existence}), and thus $\Omega\subset \pi^{-1}(s)$ is homeomorphic to $\triangle_k$ as well as $s=f(t)$. 
	Note $\alpha$ is $G$-equivariant, and $G_{\langle x, e\rangle} = H_m$, for any $x\in \triangle_{m}\setminus \triangle_{m-1}$, $0\leq m\leq k$. 
	Therefore, we have $G_p\supset G_{(\pi\llcorner\Omega)^{-1}([p_{\tilde{s}}])}$ for every $p\in\Omega$. 
	Finally, we show $\pi\llcorner\Omega$ is bi-Lipschitz. 
	By \cite[Page 131]{illman2000existence}, there exists a real analytic structure on $M$ compatible with its smooth structure so that the action of $G$ is real analytic. 
	Hence, by \cite[Theorem 1]{murayama2013triangulation}, $M$ admits a piecewise smooth, subanalytic triangulation $\tau: P_M\to M$, which is unique up to piecewise linear homeomorphisms. 
	Additionally, by \cite[Theorem 8.2]{illman2000existence}, the triangulation $\triangle$ can be chosen as a subanalytic equivariant triangulation in the sense of \cite[Definition 7.1, 7.2]{illman2000existence}. 
	In particular, we have $\alpha$ is a subanalytic isomorphism, $\triangle_k(G; H_0,\dots, H_k)$ and $\tilde{s}=\pi^{-1}(s)$ are subanalytic closed sets. 
	By the proof of \cite[Proposition 7.4]{illman2000existence}, we see $\Omega\subset \tilde{s}$ is a subanalytic closed subset. 
	Then, using \cite[Theorem 3.8]{illman2000existence}, $\tau$ can be compatible with $\Omega$ and $\tilde{s}$, i.e. there are sub-complexes $P_\Omega\subset P_{\tilde{s}}\subset P_M$ so that $\tau: P_\Omega \to \Omega $ and $\tau:P_{\tilde{s}}\to \tilde{s} $ are piecewise smooth subanalytic triangulations of $\Omega$ and $\tilde{s}$. 
	Meanwhile, \cite[Theorem 3.8]{illman2000existence} also gives a subanalytic triangulation $\sigma : L \to \triangle_k(G; H_0,\dots, H_k)$ compatible with $ P(\triangle_k\times\{e\} )$, i.e. there is sub-complex $L'\subset L$ so that $\sigma: L'\subset L \to P(\triangle_k\times\{e\} )$ is a subanalytic triangulation of $ P(\triangle_k\times\{e\} )$. 
	Together, we have $\alpha\circ\sigma : L\to \tilde{s}$ is a subanalytic triangulation of $\tilde{s}$. 
	The uniqueness \cite[the subanalytic Hauptvermutung]{murayama2013triangulation} (see also \cite[Theorem 3.10]{illman2000existence}) implies that $L$ is piecewise linear homeomorphic to $P_{\tilde{s}}$. 
	Note $P(\triangle_k\times\{e\} )\cong \triangle_k$ is already a simplex. 
	Hence, the uniqueness also implies $\triangle_k\cong P(\triangle_k\times\{e\} )$ is piecewise linear homeomorphic to $L'$. 
	Therefore, we have $\alpha: \triangle_k\cong P(\triangle_k\times\{e\} ) \to \Omega$ is a piecewise smooth homeomorphism, and thus a bi-Lipschitz homeomorphism. 
	\begin{displaymath}    \xymatrix{ & L' \ar@{<->}[r]^{P.L.~~~} & P_\Omega \subset P_{\tilde{s}} \ar@{<->}[d]^{P.S.}_{\tau} \\        \triangle_k  \ar@{<->}[r]^{i_e~~~~~~} \ar@{<->}[ru]^{P.L.}  \ar@/_/[rrrr]_{Linear}   &     P(\triangle_k\times \{e\} ) \ar@{<->}[u]^{\sigma} & \Omega \subset \tilde{s} \ar@{<->}[l]_{~~~~\alpha:~Lip}     &  s \ar@{<-}[l]_{~~\pi} &  t \ar@{->}[l]_{f} }\end{displaymath}
	Moreover, by \cite[Section 8.12]{eells2001harmonic}, the triangulation $f:K\to M/G$ can be bi-Lipschitz (Riemannian polyhedra). 
	Indeed, note $\pi\llcorner\Omega$ is Lipschitz, and $\alpha$ induces a linear homeomorphism from $\triangle_k$ to $t$ in the orbit space level. 
	Then we have $f:t\to s$ is Lipschitz. 
	On the other hand, \cite[Theorem 1]{murayama2013triangulation} also shows that $f^{-1}\circ \pi\circ \tau : P_M\to K$ is piecewise linear, and thus Lipschitz. 
	For any $x,y\in t\in K$, let $[x]=f(x),[y]=f(y)\in s$. 
	Then for all $\tilde{x}\in\pi^{-1}([x])$, $\tilde{y}\in\pi^{-1}([y])$, we have $d_K(x,y) \leq C_1 d_{P_M}(\tau^{-1}(\tilde{x}), \tau^{-1}(\tilde{y}))\leq C_1C_2d_M(\tilde{x}, \tilde{y})$, where $C_1={\rm Lip}(f^{-1}\circ \pi\circ \tau)$, $C_2 = {\rm Lip}(\tau^{-1})$. 
	Hence, $d_K(x,y) \leq C_1C_2d_M(\pi^{-1}([x]), \pi^{-1}([y])) = C_1C_2d_{M/G}([x],[y])$, and $f^{-1}$ is Lipschitz, which also implies $(\pi\llcorner \Omega)^{-1}:s\to \Omega$ is Lipschitz. 
	Together, we have $\pi\llcorner \Omega : \Omega\to s$ is a bi-Lipschitz homeomorphism. 
\end{proof}

\section{An Interpolation Lemma}\label{Sec:proof-of-equivalence-a.m.v}

In this appendix, we build an interpolation lemma which is parallel to \cite[Lemma 3.8]{pitts2014existence}. 
We start with some preparatory work. 

Recall $l={\rm Cohom}(G)\geq 3$. 
Suppose $G\cdot p$ is a principle orbit and $\exp_{G\cdot p}^\perp$ is a $G$-equivariant diffeomorphism on $B^G_r(p)$. 
Let $g_{_{M/G}}$ be the induced Riemannian metric on $M^{reg}/G$ so that $\pi: M^{reg}\to M^{reg}/G$ is a Riemannian submersion. 
Define $\vartheta^{G\cdot p}:~ B^G_r(p)\rightarrow \mathbb{R}^+$ by
\begin{equation*}
	\vartheta^{G\cdot p} (q) := \frac{\mathcal{H}^{n+1-l}(G\cdot q)}{\mathcal{H}^{n+1-l}(G\cdot p)}. 
\end{equation*}
By the co-area formula, we have the following equality for any $G$-invariant $n$-rectifiable set $\Sigma\subset B^G_r(p)$: 
\begin{eqnarray}\label{E:weyl-tube-2}
	\mathcal{H}^{n}(\Sigma) 
	&=& \mathcal{H}^{n+1-l}(G\cdot p)\cdot \int_{\pi(\Sigma)} \vartheta^{G\cdot p }(q) ~d\mathcal{H}^{l-1}(q). 
\end{eqnarray}

Since ${\rm Clos}(W)\subset M^{reg}$ by assumptions, then there exists a positive lower bound on the injectivity radius of orbits $\inf_{p\in {\rm Clos}(W)} {\rm Inj}(G\cdot p)>2r_0>0$, so that for any $q\in W$ the normal exponential map $\exp_{G\cdot q}^\perp$ is a diffeomorphism on $B_{2r_0}^G(q)$. 
Moreover, for any $q\in W$, $\epsilon'>0$, there exists $r_{\epsilon'}=r(\epsilon',q)\in(0,r_0)$ so that 
\begin{equation}\label{E:bound-of-theta}
	\frac{1}{1+\epsilon'} \leq \inf_{p,p'\in B^G_{r_{\epsilon'}}(q)}\frac{\mathcal{H}^{n+1-l}(G\cdot p')}{\mathcal{H}^{n+1-l}(G\cdot p)}\leq \sup_{p,p'\in B^G_{r_{\epsilon'}}(q)}\frac{\mathcal{H}^{n+1-l}(G\cdot p')}{\mathcal{H}^{n+1-l}(G\cdot p)}\leq 1+\epsilon',
\end{equation}
which implies $\vartheta^{G\cdot p}(p')\in [\frac{1}{1+\epsilon'},1+\epsilon']$, for all $ p,p'\in B^G_{r_{\epsilon'}}(q)$. 


Note the exponential map $\exp_{[p]}$ in $M^{reg}/G$ can be written as $\pi\circ\exp_{G\cdot p}^\perp\circ ((d\pi)\llcorner {\bf N}_pG\cdot p )^{-1}$. 
Thus, for any $q\in W$, and $\epsilon,\epsilon'\in (0,1)$, there exists a $G$-neighborhood $Z=B^G_\rho(q)$ of $q$ in $M$ for some $\rho<r_{\epsilon'}$, such that $\pi(Z)$ satisfies all the properties in \cite[3.4 (4)]{pitts2014existence} with $\exp_{[p]}\big \vert_{\exp_{[p]}^{-1} (\pi(Z))}$ in place of $\exp_p$. 
Hence, we have 
$$E:=\exp^\perp_{G\cdot p}\big\vert_{(\exp^\perp_{G\cdot p})^{-1}(Z)}$$
is a diffeomorphism onto $Z$ if $G\cdot p\subset Z$. 
Combining (\ref{E:weyl-tube-2}) and (\ref{E:bound-of-theta}), we can generalize the inequalities in \cite[3.4 (7)]{pitts2014existence} into $G$-invariant versions as follow: 
\begin{lemma}\label{Lem:construct cone}
	Suppose $r>0$, $\lambda\in [0,1]$, $B^G_r(p)\subset\subset Z$. 
	Let $ \bmu_{\lambda}$ be the homothety map given by $x\mapsto \lambda x$. 
	Then we have
	$$ {\bf M}((E\circ{\bm \mu}_\lambda\circ E^{-1})_\# T)\leq \lambda^{l-1}(1+\epsilon(1-\lambda))(1+\epsilon')^2{\bf M}(T), $$
	for any $T\in \mathcal{Z}^G_n(B^G_r(p),\partial B^G_r(p);\mathbb{Z}_2)$. 
	Moreover, denote 
	$$T_\lambda := E_\#\big(\delta_{\bf 0}\mathbb{X} \big[E^{-1}_\#\partial T-({\bm \mu}_\lambda\circ E^{-1})_\#\partial T \big]\big),$$
	where ${\bf 0}$ is the zero-section in ${\bf N}(G\cdot p)$, and $\delta_{\bf 0}\mathbb{X} S:= h_\#([[0,1]]\times S)$ for a map $h:\mathbb{R}\times {\bf N}(G\cdot p)\to {\bf N}(G\cdot p)$ given by $h(t,v):=tv$.
	Then 
	\begin{itemize}
		\item[$\bullet$] $\partial T_\lambda =  \partial T - \partial [(E\circ \bmu_\lambda \circ E^{-1})_\#  T ] $;
		\item[$\bullet$] $\spt(T_\lambda)\subset \an(p,\lambda r,r)$;
		\item[$\bullet$] $ {\bf M}(T_\lambda)\leq 2r(l-1)^{-1}(1-\lambda^{l-1})(1+\epsilon')^2{\bf M}(\partial T)$.
	\end{itemize}
\end{lemma}
\begin{proof}
	Denote $E':= \exp_{[p]}\big\vert_{(\exp_{[p]})^{-1}(\pi(Z))} $. 
	Then we have $\pi\circ E\circ{\bm \mu}_\lambda\circ E^{-1} = E'\circ{\bm \mu}_\lambda\circ E'^{-1}\circ\pi$
	By $\mZ_2$ coefficient, we have 
	\begin{eqnarray*}
		 {\bf M}((E\circ{\bm \mu}_\lambda\circ E^{-1})_\# T) &=& \mH^n ((E\circ{\bm \mu}_\lambda\circ E^{-1})(\Sigma))
		 \\
		 &\leq & (1+\epsilon') {\rm vol}(G\cdot p) \mH^{l-1}\big(((E'\circ{\bm \mu}_\lambda\circ E'^{-1})(\pi(\Sigma))) \big)
		 \\
		 &\leq & (1+\epsilon') {\rm vol}(G\cdot p)  \lambda^{l-1}(1+\epsilon(1-\lambda)) \mH^{l-1}(\pi(\Sigma))
		 \\
		 &\leq & (1+\epsilon')^2 \lambda^{l-1}(1+\epsilon(1-\lambda)) \M(T),
	\end{eqnarray*}
	where $\Sigma=\spt(T)$, the second inequality comes from \cite[3.4 (7)]{pitts2014existence} with $\exp_{[p]}$ in place of $\exp_p$, and the first as well as the third inequality comes from (\ref{E:weyl-tube-2})
	(\ref{E:bound-of-theta}). 
	Moreover, claims in the first two bullets follow directly from the definition of $T_\lambda$. 
	The last inequality in the lemma follows from a similar procedure as above. 
\end{proof}

Using the above lemma, we can adapt the proof of \cite[Lemma 3.5]{pitts2014existence} in a straightforward way and get the following lemma:
\begin{lemma}\label{Lem:pre-interpolation 1}
	Given $\epsilon,\epsilon'\in(0,1)$ and $\delta >0$, let $Z$ be defined as above. 
	Suppose $B^G_r(p)\subset\subset Z$ and $T\in \mathcal{Z}^G_n(B^G_r(p),\partial B^G_r(p);\mathbb{Z}_2)$. 
	Assume additionally: 
	\begin{itemize}
		\item[$\bullet$] $\epsilon \mathbf{M}(T)< (l-1) \delta \text { and } \frac{2 r}{l-1} \mathbf{M}\left(\partial T\right)(1+\epsilon')^2+\delta \leq \mathbf{M}(T)$;
		\item[$\bullet$] $\| T \| (\partial B^G_s(p))=0, \text { for all } 0 \leq s \leq r$.
	\end{itemize}
	Then for any $\beta>0$ there exists a sequence 
	$T=R_0,R_1,\dots,R_m\in \mathcal{Z}^G_n(B^G_r(p),\partial B^G_r(p);\mathbb{Z}_2) $ 
	such that for each $i\in \{1,\dots,m\}$ we have 
	\begin{itemize}
		\item[$\bullet$] $\partial R_i=\partial T$;
		\item[$\bullet$] ${\bf M}(R_i)\leq (1+\epsilon')^2{\bf M}(T)+\beta$;
		\item[$\bullet$] ${\bf M}(R_i-R_{i-1})\leq \beta +(2\epsilon'+\epsilon'^2){\bf M}(T)$;
		\item[$\bullet$] $R_m=(\exp^\perp_{G\cdot p})_\#(\delta_{\bf 0}\mathbb{X}(\exp^\perp_{G\cdot p})^{-1}_\#\partial T)$;
		\item[$\bullet$] ${\bf M}(R_m)\leq \frac{2 r}{l-1} \mathbf{M}\left(\partial T\right)(1+\epsilon')^2\leq {\bf M}(T)-\delta$.
	\end{itemize} 
\end{lemma}
\begin{proof}
	The proof of \cite[Lemma 3.5]{pitts2014existence} would carry over by using Lemma \ref{Lem:construct cone} in place of \cite[3.4 (7)]{pitts2014existence} and using $G$-cycles, $G$-sets, $\exp_{G\cdot p}^\perp$ in place of relevant objects.
\end{proof}

The following lemma is a simple equivariant modification of \cite[Lemma 3.6]{pitts2014existence}. 
\begin{lemma}\label{Lem:zero boundary}
	Suppose $Z$ is an open $G$-set as above and $V\in \mathcal{V}^G_n(M)$ which is rectifiable in $Z$, then for $\mathcal{H}^n$-a.e. $p\in Z$, we have $\|V\|(\partial B^G_s(p))=0$ for all $s>0$ with $B^G_s(p)\subset\subset Z$. 
\end{lemma}
\begin{proof}
	Using the normal exponential map $\exp_{G\cdot p}^\perp$ in place of $\exp_p$, the original proof of \cite[Lemma 3.6]{pitts2014existence} would carry over. 
\end{proof}

Finally, we have the following $G$-invariant pre-interpolation lemma: 

\begin{lemma}\label{L:pre-interpolation}
	Suppose $L>0$, $\delta>0$, $W\subset\subset U$ is an open $G$-set with ${\rm Clos}(W)\subset M^{reg}$, and $T\in \mathcal{Z}^G_n(M;\mathbb{Z}_2)$. 
	Given $S_0\in \mathcal{Z}^G_n(M;\mathbb{Z}_2)$ with 
	$${\rm spt}(S_0-T)\subset W\quad {\rm and}\quad {\bf M}(S_0)\leq L, $$
	there exists $\epsilon=\epsilon(L,\delta,W,T,S_0)>0$ such that if $S\in \mathcal{Z}^G_n(M;\mathbb{Z}_2)$ satisfies 
	$$ {\rm spt}(S-T)\subset W, \quad {\bf M}(S)\leq L, \quad \mathcal{F}(S-S_0)\leq\epsilon,$$
	then there is a sequence $S=T_0,T_1,\dots,T_m=S_0\in \mathcal{Z}^G_n(M;\mathbb{Z}_2)$ such that for each $j\in\{1,\dots,m\}$ we have 
	\begin{itemize}
		\item[(i)] ${\rm spt}(T_j-T)\subset U$;
		\item[(ii)] ${\bf M}(T_j)\leq L+\delta$;
		\item[(iii)] ${\bf M}(T_j-T_{j-1})\leq\delta$.
	\end{itemize}
\end{lemma}
\begin{proof}
	Suppose the lemma is false, then there exists a sequence $\{S_j\}_{j=1}^\infty\subset \mathcal{Z}^G_n(M;\mathbb{Z}_2)$ satisfying ${\rm spt}(S_j-T)\subset W$, ${\bf M}(S_j)\leq L$, $\mathcal{F}(S_j-S_0)\to 0$, 
	but none of the $S_j$ can be connected to $S_0$ by a finite sequence in $\mathcal{Z}^G_n(M;\mathbb{Z}_2)$ satisfying (i)-(iii). 
	Up to a subsequence, we can suppose $|S_j|\to V\in\mathcal{V}^G_n(M)$ by Remark \ref{Rem:compactness for G-varifold}, and hence 
	$\|S_0\|(A)\leq \|V\|(A),$ 
	for any Borel $A\subset M$ by the lower-semicontinuous of mass. 
	
	Let $\alpha := \delta/5$. 
	Then the proof of this lemma can be divided into two cases: 
	\begin{itemize}
		\item[$\bullet$] {\bf Case 1}: $\|V\|(G\cdot q)\leq \alpha$ for all $q\in W$;
		\item[$\bullet$] {\bf Case 2}: $\{q\in W\vert~ \|V\|(G\cdot q)>\alpha\}\neq\emptyset$.
	\end{itemize}
	
	For the {\bf Case 1}, using Lemma \ref{Lem:slice} and \ref{Lem:isoperimetric}, the original proof of Pitts in \cite[Page 114-118]{pitts2014existence} would carry over with geodesic $G$-tubes $\{B^G_{r_i}(p_i)\}$, the distance function to an orbit $\dist_M(G\cdot p_i,\cdot)$ in place of geodesic balls $\{B_{r_i}(p_i)\}$ and $\dist_M(p_i,\cdot)$, respectively. 
	Consequently, for $j$ large enough, we can get a finite sequence in $\mathcal{Z}^G_n(M;\mathbb{Z}_2)$ satisfying (i)-(iii) and connecting $S_j$ to $S_0$, which is a contradiction. 
	
	As for the {\bf Case 2}, we mention the set $ \{q\in W : \|V\|(G\cdot q)>\alpha\}$ only contains a finite number of orbits since $\|V\|(M)\leq L$. 
	Hence, we firstly suppose $\{q\in W : \|V\|(G\cdot q)>\alpha\}= G\cdot q$, and use an induction argument to show the general case.
	
	Take $\epsilon,\epsilon'>0$ small enough such that 
	$$ 2\epsilon \|V\|(U)\leq (l-1)\frac{\alpha}{2}, {\rm ~and }~2(2\epsilon'+\epsilon'^2)\|V\|(U)\leq \frac{\alpha}{64}. $$ 
	Let $Z$ be the open $G$-neighborhood of $q$ given as in Lemma \ref{Lem:pre-interpolation 1} corresponding to $\epsilon,\epsilon'$. 
	By Lemma \ref{Lem:zero boundary}, we have $\|V\|(\partial \an(p,s,r))=0$ for $\mathcal{H}^n$-a.e. $p\in Z$ and all $0<s<r$ with $B^G_r(p)\subset\subset Z$, which implies $$\lim_{j\to \infty}\|S_j\|\llcorner {\rm Clos}(\an(p,s,r))= \|V\|\llcorner {\rm Clos}(\an(p,s,r)).$$
	Moreover, we have $\lim_{r\to 0}\|V\|({\rm Clos}(B^G_r(q))\setminus G\cdot q)= 0$ by  \cite[2.1(14b(iii))]{pitts2014existence}. 
	By a similar argument as in \cite[Page 119]{pitts2014existence}, one verifies that there exists a positive integer $J$ satisfying:
	
	For each $j\geq J$, there exist $p_j\in Z$, $0<\frac{s_j}{2}<r_j<s_j$ such that 
	\begin{itemize}
		\item[(a)] $p_j\to q$, $s_j\to 0$;
		\item[(b)] $B^G_{r_j/4}(q)\subset B^G_{r_j/2}(p_j)\subset B^G_{2s_j}(p_j)\subset Z$,
		\item[(c)] $\|S_j\|(\partial B^G_s(p_j))=0$, for all $0<s\leq r_j$;
		\item[(d)] $\langle S_j,d_{G\cdot p_j},r_j\rangle \in {\bf I}^G_{n-1}(Z;\mathbb{Z}_2)$, where $d_{G\cdot p_j}:=\dist_M(G\cdot p_j,\cdot)$;
		\item[(e)] $\lim_{j\to\infty} \|S_j\|({\rm Clos}(\an(p_j,s_j/2,2s_j)))=0$;
		\item[(f)] $16  \|S_j\|({\rm Clos}(\an(p_j,s_j/2,2s_j))) \geq 3r_j {\bf M}(\langle S_j,d_{G\cdot p_j},r_j\rangle)$;
		\item[(g)] $\epsilon \|S_j\|(U)\leq (l-1)\frac{\alpha}{2}$, and $(2\epsilon'+\epsilon'^2)\|S_j\|(U)\leq \frac{\alpha}{64}$;
		\item[(h)] $\frac{2r_j}{l-1}{\bf M}(\langle S_j,d_{G\cdot p_j},r_j\rangle)(1+\epsilon')^2+\frac{\alpha}{2} \leq \|S_j\|({\rm Clos}(B^G_{r_j}(p_j)))$;
		\item[(i)] $\lim_{j\to\infty} |S_j|\llcorner (M\setminus {\rm Clos}(B^G_{r_j}(p_j))) = V\llcorner (M\setminus (G\cdot q) )$.
	\end{itemize}
	We point out that (c) comes from Lemma \ref{Lem:zero boundary}, (g) comes from the choice of $\epsilon$ as well as $\epsilon'$, and (h) comes from (e)(f).
	
	Define then $\widetilde{S}_j \in \Z_n^G(M;\mZ_2)$ as 
	$$\widetilde{S}_j := S_j\llcorner(M\setminus {\rm Clos}(B^G_{r_j}(p_j)) ) + (\exp^\perp_{G\cdot p_j})_\# \big( \delta_{\bf 0}\mathbb{X}(\exp^\perp_{G\cdot p_j})^{-1}_\# \langle S_j,d_{G\cdot p_j},r_j\rangle \big).$$
	Since $(2\epsilon'+\epsilon'^2)\|S_j\|(U)\leq \frac{\delta}{2}$ by (g), and 
	$\partial ( S_j\llcorner \Clos(B^G_{r_j}(p_j)) ) = \langle S_j, d_{G\cdot p_j}, r_j\rangle,$
	we can apply Lemma \ref{Lem:pre-interpolation 1} with $r:=r_j$, $\delta:=\frac{\alpha}{2}$, $p:=p_j$, $\beta:=\frac{\delta}{2}$, $T:=S_j\llcorner B^G_{r_j}(p_j)$, to obtain a finite sequence $\{R_i^j\}_{i=0}^{m_j}\subset\Z^G_n(M;\mZ_2)$ satisfying (i)-(iii) so that  $R_0^j=S_j$ and $R_{m_j}^j=\widetilde{S}_j$. 
	Furthermore, we have the following results for $\widetilde{S}_j$:
	\begin{itemize}
		\item[$\bullet$] $\M(\widetilde{S}_j)\leq \M(S_j)-\frac{\alpha}{2} $, by the last inequality in Lemma \ref{Lem:pre-interpolation 1};
		\item[$\bullet$] $\M\Big(     (\exp^\perp_{G\cdot p_j})_\# \big( \delta_{\bf 0}\mathbb{X}(\exp^\perp_{G\cdot p_j})^{-1}_\# \langle S_j,d_{G\cdot p_j},r_j\rangle \big)     \Big) \leq C\M (\langle S_j,d_{G\cdot p_j},r_j\rangle) \to 0$, by Lemma \ref{Lem:construct cone}, (f) and (e).
	\end{itemize}
	The second bullet implies 
	$$\lim_{j\to\infty} \widetilde{S}_j = \lim_{j\to\infty} S_j=S_0,\quad {\rm and~}\lim_{j\to\infty} |\widetilde{S}_j| = V\llcorner (M\setminus G\cdot q)  .$$ 
	Thus, the sequence $\{\widetilde{S}_j\}_{j=J}^\infty$ are in the {\bf Case 1}. 
	Therefore, for sufficiently large $j$, $ \widetilde{S}_j$ can be connected to $S_0$ through a finite sequence $\{\widetilde{R}_i^j\}_{i=0}^{\tilde{m}_j}$ in $\Z^G_n(M;\mZ_2)$ satisfying (i)-(iii), which gives 
	$\{S_j=R_0^j,\dots,R_{m_j}^j= \widetilde{S}_j=\widetilde{R}_0^j,\dots, \widetilde{R}_{\tilde{m}_j}^j= S_0\} \subset \Z^G_n(M;\mZ_2)$ 
	satisfying (i)-(iii) as a contradiction.

	For the case that $\{q\in W : \|V\|(G\cdot q)>\alpha\}$ contains more than one orbit. 
	Noting ${\bf M}(\widetilde{S}_j)\leq {\bf M}(S)-\frac{\alpha}{2}$ and $\lim_{j\to\infty} |\widetilde{S}_j |=V\llcorner(M\setminus{G\cdot q})$, an induction argument shows that there is a finite sequence in $\mathcal{Z}^G_n(M;\mathbb{Z}_2)$ satisfying (i)-(iii) which connects $S_j$ to some $R_j\in \mathcal{Z}^G_n(M;\mathbb{Z}_2)$ with 
	\begin{itemize}
		\item[$\bullet$] ${\bf M}(R_j)\leq {\bf M}(S_j)$;
		\item[$\bullet$] $\lim_{j\to\infty}R_j=\lim_{j\to\infty}S_j=S_0$;
		\item[$\bullet$] $\lim_{j\to\infty}|R_j|=V\llcorner (M\setminus \{q\in W : \|V\|(G\cdot q)>\alpha\} )$.
	\end{itemize}
	Finally, applying the argument in {\bf Case 1} to $\{R_j\}$ gives a contradiction. 
\end{proof}

Combining Lemma \ref{L:pre-interpolation} with a covering argument as in \cite[Page 124]{pitts2014existence}, we obtain the following interpolation lemma: 

\begin{lemma}\label{L:interpolation}
	Suppose $L>0$, $\delta>0$, $W\subset\subset U$ is an open $G$-set with ${\rm Clos}(W)\subset M^{reg}$, and $T\in \mathcal{Z}^G_n(M;\mathbb{Z}_2)$. 
	There exists $\epsilon=\epsilon(L,\delta,W,T)>0$ such that if $S_1,S_2\in \mathcal{Z}^G_n(M;\mathbb{Z}_2)$ satisfy 
	$$ {\rm spt}(S_i-T)\subset W, \quad {\bf M}(S_i)\leq L, \quad \mathcal{F}(S_1-S_2)\leq\epsilon, $$
	for $i=1,2$, 
	then there is a sequence $S_1=T_0,T_1,\dots,T_m=S_2\in \mathcal{Z}^G_n(M;\mathbb{Z}_2)$ with 
	\begin{itemize}
		\item[(i)] ${\rm spt}(T_j-T)\subset U$;
		\item[(ii)] ${\bf M}(T_j)\leq L+\delta$;
		\item[(iii)] ${\bf M}(T_j-T_{j-1})\leq\delta$.
	\end{itemize}
	for each $j\in\{1,\dots,m\}$.
\end{lemma}


\bibliographystyle{amsbook}

\end{document}